\documentclass[reqno]{amsart}
\usepackage{amsmath}
\usepackage[hyphens]{xurl}
\usepackage[breaklinks,colorlinks=true,citecolor=blue,linkcolor=blue,urlcolor=blue]{hyperref}
\usepackage{amsthm,theoremref}
\usepackage{amssymb}
\usepackage[new]{old-arrows}
\usepackage{graphics}
\usepackage[all]{xy}

\def\un{{\rm 1\mkern-4mu I}}

\usepackage{multicol}
\usepackage[refpage,notocbasic]{nomencl}
\makenomenclature

\usepackage{etoolbox}

\renewcommand{\nomgroup}[1]{%
\item[\bfseries
\ifstrequal{#1}{D}{1. Index of definitions}{%
\ifstrequal{#1}{N}{2. Index of notation}{%
\ifstrequal{#1}{A}{Appendix}{}}}%
]}
\renewcommand{\nompreamble}{\begin{multicols}{2}}
\renewcommand{\nompostamble}{\end{multicols}}

\swapnumbers

\newtheorem{thm}{Theorem}[section]

\theoremstyle{definition}

\newtheorem{fact}[thm]{Fact}
\newtheorem{variant}[thm]{Variant}

\newtheorem{cor}[thm]{Corollary}
\newtheorem{lem}[thm]{Lemma}

\newtheorem{defn}[thm]{Definition}

\newtheorem{rmkdef}[thm]{Remark/Definition}
\newtheorem{rmkex}[thm]{Remark/Example}

\newtheorem{morenot}[thm]{More Notation}

\newtheorem{csummary}[thm]{Conclusion/Summary}
\newtheorem{revision}[thm]{Revision}

\newtheorem{example}[thm]{Example}
\newtheorem{hypo}[thm]{Hypothesis}
\newtheorem{setup}[thm]{Set Up}

\newtheorem{setupnot}[thm]{Set Up/Notation}
\newtheorem{altsetup}[thm]{Alternative Set Up}
\newtheorem{claim}[thm]{Claim}

\newtheorem{context}[thm]{Context}

\newtheorem{not/clar}[thm]{Notation/Clarification}

\newtheorem{ConsDef}[thm]{Construction/Definition}

\newtheorem{furnot}[thm]{Further Notation}

\newtheorem{var}[thm]{Variant}

\newtheorem{furtherfact}[thm]{Further Fact}

\newtheorem{notation}[thm]{Notation/Revision}
\newtheorem{Notation}[thm]{Notation}
\newtheorem{defrev}[thm]{Definition/Revision}

\newtheorem{bonus}[thm]{Bonus}

\theoremstyle{remark}

\newtheorem{scholion}[thm]{Scholion}

\newtheorem{rmk}[thm]{Remark}

\numberwithin{equation}{section}

\def\dar[#1]{\ar@<2pt>[#1]\ar@<-2pt>[#1]}
  \entrymodifiers={!!<0pt,0.7ex>+}

\begin{document}

\title{Thurston Vanishing}
\author{Michael McQuillan}
\address{Universit\'a degli studi di Roma 'Tor Vergata' \& HSE University, Moscow}
\date{\today. Support provided within the framework of the HSE university basic research programme.}

\maketitle

\begin{abstract}
We show how to extend Epstein's algebraic transversality principles, \cite{adamBuff}, for rational maps $f$ of ${\mathbb P}^1_{\mathbb C}$ to infinite forward invariant subsets of the Fatou set, \thref{thm:I2}. The key, at least conceptually, to doing this is to have a topos of $f$ invariant sheaves, and Grothendieck's six operations on the same in which Epstein's theory naturally takes place. Thus the resulting count of non-repelling invariant cycles, \thref{cor:I6}, is strictly better than the minimum of \cite{adam} and \cite{mitsu}. En passant (in functorially applying the Epstein/Thurston methodology at parabolic fixed points) we calculate the dualising sheaf of a real blow up, \thref{thm:I1}, which is a remarkably algebraic object of independent interest with the capacity to enormously simplify the theory of resurgent functions and Stokes' phenomenon. 
\end{abstract}

\section{Introduction}\label{Intro}

In a series of papers, e.g. \cite{adam}, \cite{adamBuff}, \cite{adamHub}, Epstein has systematically extrapolated, from Thurston's proof of the rigidity of post critically finite rational maps, an almost exhaustive local to global theory of infinitesimal deformations of dynamical systems associated to a rational map $f : {\mathbb P}^1_{\mathbb C} \to {\mathbb P}^1_{\mathbb C}$. Indeed, the only criticism one might reasonably make is that Epstein's endemic use of the snake lemma cries out for a re-interpretation of many of his constructions in terms of a suitable topus of ``$f$-equivariant'' sheaves. In principle providing such a topos might, a priori, have been nothing other than fluff. The guiding maxim, however, is Zariski's (possibly anecdotal) ire at Grothendieck since prior to the latter's introduction of sheaves and co-homology in algebraic geometry, only Zariski could do the subject, but subsequently it was the domain of ``any idiot'', and this is exactly what happens here since as soon as one has a topos theoretic interpretation of Epstein's constructions, one has the bonus of the formalism of Grothendieck's 6 operations, and this conceptual framework immediately reveals a number of directions in which Epstein's work may be enlarged.

\smallskip

As such for $f$ an endomorphism of a ringed or topological space, $X$, we introduce in \thref{cd:def1} the category, ${\rm Sh}_{X/f}$, of $f$-sheaves with action, i.e. a sheaf with a map,
\begin{equation}
\label{I1}
\varphi : f^* f \longrightarrow {\mathcal F}
\end{equation}
of which the archi-typical example is pull-back along $f$ of functions or differential forms but not, should $f$ be ramified, tangent vectors. Similarly, since there is a dearth of such objects which are meromorphic, we also have a category, ${\rm Sh}'_{X/f}$, of almost $f$-sheaves, i.e. pairs of sheaves ${\mathcal F}_0 , {\mathcal F}_1$ together with maps,
\begin{equation}
\label{I2}
{\mathcal F}_1 \underset{s}{\longleftarrow} {\mathcal F}_0 \underset{t}{\longrightarrow} f_* \, {\mathcal F}_1 \, .
\end{equation}
Thus, for example, a reduced divisor $D$ on a complex space is an $f$-sheaf if it is forward invariant, i.e.
\begin{equation}
\label{I3}
D \supseteq f(D)
\end{equation}
but a reduced almost $f$-divisor is a pain of divisors $D_0,D_1$ such that,
\begin{equation}
\label{I4}
D_1 \subseteq D_0 \cap f^{-1} (D_0) \, .
\end{equation}
Such pairs are the focus of \cite{adamBuff} wherein the notation is $A,B$. If, however, we have a forward invariant closed set $Z$, then from the inclusions,
\begin{equation}
\label{I5}
j : U := X \backslash Z \longhookrightarrow X \longhookleftarrow Z : i
\end{equation}
we obtain for any $f$-sheaf a short exact sequence of the same,
\begin{equation}
\label{I6}
0 \longrightarrow j_! (j^* {\mathcal F}) \longrightarrow {\mathcal F} \longrightarrow i_* \, i^* {\mathcal F} \longrightarrow 0
\end{equation}
and exploiting such extra sheaves is the principle source of new examples to which Epstein's ideas are employed.

\smallskip

The content of \S\ref{SS:cd} is, therefore, derived functors in the categories ${\rm Sh}_{X/f}$ and ${\rm Sh}'_{X/f}$ of which the principle example for ${\mathcal F}$ in the former and ${\mathcal G}$ in the latter is,
\begin{equation}
\label{I7}
{\rm Ext}_{X/f}^q ({\mathcal F},{\mathcal G}) := R^q {\rm Hom}_{{\rm Sh}'_{X/f}} ({\mathcal F},{\mathcal G})
\end{equation}
which we calculate using the simple spectral sequence of \thref{cd:fact3}. Indeed, the fact that this degenerates on the $2^{\rm nd}$ sheet is what underlies the considerable mileage that Epstein was able to extract from the snake lemma. On complex manifolds we also provide explicit Dolbeault type complexes of sheaves twisted by smooth differentials, \thref{cd:fact4}, which calculate \eqref{I7}, and critically for the Thurston-Epstein argument, the corresponding dual complexes, \thref{cd:fact6}. This intervention of duality is not, however, without its inconvenience since duality here is topological, and groups such as, 
\begin{equation}
\label{I8}
{\rm Ext}_{X/f}^2 (\Omega_x^{\otimes m} , j_! \, {\mathcal O}_U) \quad \mbox{or} \quad {\rm Ext}^1_{X/f} (\Omega_x^{\otimes m} , i_* \, i^* {\mathcal O}_Z) \, , \quad m \geq 0
\end{equation}
for $i$ of \eqref{I5} the inclusion of a closed sub-disc of a Siegel disc of multiplier $\lambda$ will be separated in their natural topology iff they're finite dimensional, which is iff the denominators $q_n$ of the sequence, $p_n/q_n$, of best rational approximations to $\lambda$ satisfy,
\begin{equation}
\label{I9}
\limsup_n \frac{\log (q_{n+1})}{q_n} = 0 \, .
\end{equation}
Ideally, therefore, a weak version of Douady's conjectured generalisation, \cite{D}, of Yoccoz's theorem, \cite{yoccoz}, would also fall under Epstein's methodology. It's dual nature, however, only tells us something we already know by direct calculation, i.e. the maximal separated quotients in, for example, \eqref{I8}. Although, morally, all of our problems of this type should be localisable at such invariant discs or annulli, the nature of homological algebra is that such problems tend to spread. As such most of our theorems will not be about Ext groups but rather their maximal separated quotients,
\begin{equation}
\label{I10}
{\rm Ext}^{\bullet} \longrightarrow \overline{\rm Ext}^{\bullet}
\end{equation}
where the latter is computed in the topology afforded by its Dolbeault resolution, while an adequate series of lemmas for resolving the problems this poses to diagram chasing are provided by \thref{cd:lem1} and \thref{cd:claim1}.

\smallskip

Otherwise, the homological algebra is routine, e.g. we quickly check, \thref{cd:fact1}, resp. \thref{cd:fact2}, that ${\rm Sh}_{X/f}$, resp. ${\rm Sh}'_{X/f}$, have enough injectives rather than proving that they're actual topoi. This is, however, true, \cite{jacopo}, and op. cit. even provides an explicit underlying site. Similarly one can certainly attempt, \cite{olivia}, a 2-sheaf solution in the spirit of \cite[Expos\'e VI]{sga1}. The reader who is, however, familiar with Grothendieck's theory of champs should proceed with great caution since unless one can guarantee some invertability of $f$, e.g. \thref{cd:cor1.bis}, even the most rudimentary assertions, e.g. \thref{cor:ct2}, cannot be sliced along a transversal, i.e. the categories ${\rm Sh}_{X/f}$, and ${\rm Sh}'_{X/f}$ have been constructed to allow, necessarily, a difference between forward and backward orbits of $f$. Thus, for example if $X$ is a point and, $f$ the identity, then a sheaf in complex vector spaces is a vector space with an endomorphism which may very well not be invertible.

\smallskip

As such although \S\ref{Sgrp} is devoted to the champ in analytic spaces generated by an endomorphism $f$ of a Riemann surface. The underlying idea of formally inverting $f$ is not the right one, and it is not what is being employed in Epstein's theorems. It does, however, yield an elegant alternative proof of the well known, cf.\cite[Lemma 8.5 and Theorem 10.15]{milnor},

\begin{revision} (\thref{C2grp}) \thlabel{Rev:I1}
Let $z \in {\mathbb P}^1$ be a fixed point of period $p$ of a rational map $f$ then the endomorphism $df^p$ of $T_{{\mathbb P}^1,z}$ is identical to multiplication by some $\lambda \in {\mathbb C}$, {\it the multiplier of $f$ at the cycle} $Z_{\bullet} = \{f^i (z) \mid 1 \leq i \leq p \}$ and should $\deg (f) > 1$,

\begin{enumerate}
\item[(1)] If $0 < \vert \lambda \vert < 1$, then the immediate basin of attraction of $Z$ contains a critical point of $f$.
\item[(2)] If $\lambda \in \mu_{\infty}$ then the immediate basin of attraction of an attracting petal contains a critical point of $f$.
\end{enumerate}
\end{revision}

In this latter case, i.e. a parabolic cycle, \thref{not:ct1}, $\lambda$ is primitive of order $r \geq 0$, $f^p$ is tangent to order $e > 0$ to rotation by $\lambda$, and $f^{pr}$ is formally conjugate to the time 1 flow of,
\begin{equation}
\label{I11}
\frac{z^{re+1}}{1 + \nu z^{re}} \ \frac{\partial}{\partial z}
\end{equation}
where $\nu \in {\mathbb C}$ is \'Ecalle's r\'esidu iteratif. In particular, the best, and as it happens necessary, way to understand such examples is after making a real blow up,
\begin{equation}
\label{I12}
\rho : \widetilde X \longrightarrow X = {\mathbb P}^1 \, , \quad E_{\bullet} = \varphi^{-1} (Z_{\bullet})
\end{equation}
in the parabolic cycle. This is plainly an $f$-equivariant operation, so $f$ lifts, we can discuss ${\rm Sh}_{\widetilde X / f}$, etc., and the cycle lifts to a cycle $\widetilde Z_{\bullet} \subset E_{\bullet}$ with $e$ distinct attracting orbits, a small neighbourhood of which are the attracting petals of item (2) of \thref{Rev:I1}, as discussed in more detail in \thref{DRev:2}.

\smallskip

At this juncture we arrive to our principle object of interest, namely, form the real blow up of ${\mathbb P}^1$ of \eqref{I12} but in all parabolic cycles (or if one prefers not to appeal to \thref{Rev:I1} in a finite subset which we'll a posteriori prove has bounded cardinality) and complement $\widetilde X$ in the closed set $\widetilde Z$ defined, \thref{setup:ct1}, as the union of,

\begin{enumerate}
\item[(I)] Attracting cycles, $Z_{\bullet}$, of \thref{Rev:I1} (1).
\item[(I')] Super attracting cycles, $Z_{\bullet}$, i.e. $\lambda = 0$ in op. cit.
\item[(II)] The parabolic attracting cycles, $\widetilde Z_{\bullet}$, of item (2) of op. cit. and \eqref{I12} et~seq.
\item[(III)] A closed invariant disc, $D_{\bullet}$, in each Siegel disc, i.e. $\vert \lambda \vert = 1$ in op. cit. and $f^p$ holomorphically conjugate to an irrational rotation.
\item[(III')] A closed invariant annulus, $A_{\bullet}$, in each Herman ring, i.e. a cycle of annuli whose return map is an irrational rotation. 
\end{enumerate}

\noindent Or, again, a finite subset of these if one wishes to prove such finiteness a posteriori. We therefore have a purely notational variant of \eqref{I5}, i.e.
\begin{equation}
\label{I13}
\widetilde U \overset{\widetilde j}{\longhookrightarrow} \widetilde X \overset{\widetilde i}{\longhookleftarrow} \widetilde Z \, .
\end{equation}
The important point, however, is that are many more $f$ (as opposed to almost $f$) divisors on $\widetilde U$ than on $\widetilde X$, e.g. the forward orbits of the critical points of \thref{C2grp}, the relevant class of which is almost $f$-divisors, $R$, \eqref{I4} on $\widetilde U$ such that for $\vert R \vert$ the underlying reduced divisor and $m > 1$ fixed there is some $N \geq 0$, for which,
\begin{equation}
\label{I14}
\frac1m \, \widetilde R - \vert R \vert \leq {\rm Ram}_f + N \cdot Z_{\rm CR} \, .
\end{equation}
Here $Z_{\rm CR}$ is the totality of all non-repelling cycles that we have not so far discussed, i.e. those of Cremer type where, by definition, the multiplier $\lambda \in S^1 \backslash \mu_{\infty}$ but locally $f$ is only formally conjugate to an irrational rotation. In such notation, morally, the content of Thurston vanishing is that, if $R$ satisfies \eqref{I14},
\begin{equation}
\label{I15}
{\rm Ext}^2_{\widetilde X / f} (\rho^* \, \Omega_X^{\otimes m} , \, \widetilde j_! \, {\mathcal O}_{\widetilde U} (-\widetilde R)) \, , \quad m \geq 1
\end{equation}
is zero whenever $\deg (f) > 1$. There are, however, several caveats to this assertion, to wit:

\begin{enumerate}
\item[(A)] It can have dimension $1$ if $m=1$ and $f$ is a Latt\`es example, i.e. multiplication on an elliptic curve modulo $\pm1$.
\item[(B)] As we've said in \eqref{I10} et seq., the key is to study the dual of \eqref{I15}, so unless we kown that \eqref{I9} holds at all Siegel discs and Herman rings we're only going to get vanishing of the maximal separated quotient of \eqref{I15}.
\item[(C)] Duality on real blow ups isn't the same as on Riemann surfaces, and although there is a dualising sheaf (rather than a complex) it has torsion, a notable part of which will have a non-zero dual in \eqref{I15}.
\end{enumerate}

\smallskip

Plainly if there were no parabolic cycles the only real caveat is (B), and the only thing to add is that, in such a case, the dual of \eqref{I15} is a group, \thref{not:v2}, of co-invariant meromorphic $m+1$ forms,
\begin{equation}
\label{I16}
{\mathbb H}_0 (U/f , (f^* \, \Omega_X^{\otimes m}) \cdot (R) \otimes \omega_U)
\end{equation}
where, in the immediate hypothesis, the dualising sheaf $\omega_X$ is the sheaf of differentials $\Omega_X$. This much is general nonsense, but Thurston's key observation is that the groups of \eqref{I16} have a natural norm, \eqref{Res33}, which if it's finite, \eqref{v36} et seq., force co-invariants to be eigenvectors for $f^*$, and from there to the Latt\`es caveat (A) is easy. As such the further contribution of Epstein was to handle cases where a priori (but not a posteriori) Thurston's norm might have been infinite by way of his dynamic residue of \thref{DRev:1} which, in the absence of parabolic cycles, is just as easy to calculate for an essential singularity along $X \backslash U$ as the meromorphic singularities encountered in \cite[\S3]{adam}.

\smallskip

As such, provided we're happy to work with maximal separated quotients, the real caveat is (C). To grasp the issue it is usefull to observe, as Jeremy Khan pointed out to me, if the closed set in \eqref{I5} gets too big then although forms on a neighbourhood of it will have a Thurston norm (adapted to the Hausdorff dimension of $Z$) there will, in general, be little that can be said about it. Consequently it's important to keep $Z$ small, e.g. points. However, at parabolic cycles even points on the initial $X$ aren't small enough, and one needs to replace them by the smaller cycle, $\widetilde Z_{\bullet}$, of \eqref{I12} et seq. in order to control Epstein's dynamical residue -- \thref{Dschol:2} is a very mild example of what can happen.

\smallskip

This problem results from the relation between invariant differentials with essential singularities on the repelling petals to those on the attracting petals, and the way out of it is to run, functorially with respect to the ideas, Thurston's argument on the real blow up which forces such differentials to have meromorphic growth on the repelling petals. Specifically on the real blow up of \eqref{I12} we have a sheaf ${\mathcal O}_{\widetilde X}$ of holomorphic functions which are $C^{\infty}$ up to the boundary, $E$, and it's dualising complex is not only a sheaf, $\omega_{\widetilde X}$, but a remarkably algebraic object of independent interest, i.e.

\begin{thm}\thlabel{thm:I1} (\thref{Dc:Fact4})
Let $\widetilde{\mathcal O}_{\widetilde X} \hookleftarrow {\mathcal O}_{\widetilde X}$ be the sheaf of holomorphic functions which, in polar coordinates are locally $\ell_1$ for $dr \, d \theta$, with $\widetilde\omega (nE)$ the locally free rank one $\widetilde{\mathcal O}_{\widetilde X}$ module of differentials with a pole of order $n$ around $E$ then the dualising sheaf, $\omega_{\widetilde X}$, of the category of ${\mathcal O}_{\widetilde X}$ modules is a non-split extension,
\begin{equation}
\label{I17}
0 \longrightarrow {\rm Tors} \, \omega_{\widetilde X} \longrightarrow \omega_{\widetilde X} \longrightarrow \varinjlim_n \, \widetilde\omega (nE) \longrightarrow 0 
\end{equation}
where the torsion is canonically isomorphic to,
\begin{equation}
\label{I18}
\rho^* \, {\mathcal H}^1_{{\rm Zar} , Z} (X,\omega_X) \, , \quad \mbox{i.e. non-canonically} \quad {\mathbb C} \left[\frac1z\right] \frac{dz}z
\end{equation}
for $z$ any local coordinate around a point of $Z$.
\end{thm}

As such the dual of \eqref{I15} is, modulo torsion, \thref{not:v2}, a space of co-invariants,
\begin{equation}
\label{I19}
{\mathbb H}_0 (\widetilde U / f , f^* \, \Omega_X^{\otimes m} (R) \otimes j^* \omega_{\widetilde X})
\end{equation}
which by \eqref{I13} and \eqref{I17} have meromorphic growth away from the attracting cycle (or directions/petals in more classical language) $\widetilde Z$, equivalently there is no wild behaviour on the repelling petals. Certainly, a priori, the behaviour on attracting petals is arbitrary, but we only need to know the sign of Epstein's dynamical residue, so curbing the wild behaviour is enough, \thref{Dclaim:3.b.bis.bis} and \thref{lem:v2}.

\smallskip

Putting all of this together proves that if $f$ isn't a Latt\`es example then the dual \eqref{I19} of \eqref{I15} wholly comes from the torsion in \eqref{I17}, or, what is equivalent,

\begin{thm}\thlabel{thm:I2} (\thref{thm:v1})
By way of notation let \eqref{I5} be the image of \eqref{I13} in $X$, then the map of maximal separated quotients,
\begin{equation}
\label{I20}
\overline{\rm Ext}^2_{X/f} (\Omega_X^{\otimes m} , j_! \, {\mathcal O}_U (-R)) \longrightarrow \overline{\rm Ext}^2_{\widetilde X/f} (\rho^* \, \Omega_X^{\otimes m} , \widetilde j_! \, {\mathcal O}_{\widetilde U} (-\widetilde R)) \, , \quad m \geq 1
\end{equation}
is zero unless $m=1$, and $f$ is a Latt\`es example, in which case it is of dimension $1$.
\end{thm}

To understand the consequences of this for the tangent space to the deformation space, we introduce the quotient,
\begin{equation}
\label{I21}
0 \longrightarrow \widetilde j_! \, {\mathcal O}_{\widetilde U} (-\widetilde R) \longrightarrow {\mathcal O}_{\widetilde X} \longrightarrow \widetilde {\mathcal Q}_{\rm big} \longrightarrow 0
\end{equation}
where for $i_{\bullet} : C_{\bullet} \hookrightarrow X$ any of the closed sets (I) -- (III') that we've omitted, we have a quotient,
\begin{equation}
\label{I22}
\widetilde {\mathcal Q}_{\rm big} \longrightarrow (i_{C_{\bullet}})_* \, i^*_{C_{\bullet}} {\mathcal O}_{\widetilde X} \, .
\end{equation}

In particular if $C_{\bullet} = \widetilde Z_{\bullet}$ post \eqref{I12}, then the right hand side of \eqref{I22} has, \thref{rmk:v3}, a diagonal subspace, whose fibre $\widetilde {\mathcal Q}_{\rm diag}$ is annihilated by the torsion, so from the long exact sequence of Ext's associated to \eqref{I21}.

\begin{cor}\thlabel{thm:I3} (\thref{thm:v2})
Unless $m=1$, and $f$ is a Latt\`es example, there is a surjection to the maximal separated quotient,
\begin{equation}
\label{I23}
{\rm Ext}^1_{\widetilde X/f} (\rho^* \, \Omega_X^{\otimes m} , {\mathcal O}_{\widetilde X}) \twoheadrightarrow \overline{\rm Ext}^1_{\widetilde X/f} (\rho^* \, \Omega_X^{\otimes m} , \widetilde {\mathcal Q}_{\rm diag}) \, .
\end{equation}
\end{cor}

In the algebraic setting (modulo a subtlety at parabolic cycles to be discussed below) where one replaces all of the closed cycles (I) -- (III) by a divisor (of any multiplicity) supported on them, with the further provision that the divisor $\rho (\widetilde R)$ of \eqref{I14} is a divisor on $X$ rather than $U$, this is the main theorem, of \cite{adam} and \cite[proposition 13]{adamBuff}, while Herman rings of (III') are described in \cite{adamHub}. Consequently the main practical difference between \thref{thm:I3} and the aforesaid works is that it can handle infinite orbits of critical points which land in the Fatou set and which, naturally, we call, \thref{defn:ct1}, tame ramification. It is, therefore, to be noted that the $\widetilde {\mathcal Q}_{\rm diag}$ of \eqref{I23} is a product of local objects which are trivial to calculate, and that a fairly exhaustive list of such calculations is carried out in \S\ref{SS:cd}, all of which are, basically, \cite[\S3]{adam} in different clothes, with the change from $m=1$ to arbitrary being trivial. Thus, for example, as one might expect, if $x$ is a fixed point with multiplier $\lambda$, and $\vert \lambda \vert \ne 0$ or $1$, then,
\begin{equation}
\label{I24}
\overline{\rm Ext}^1_{\widetilde X/f} (\Omega_X , {\mathcal O}_{X,x}) = {\rm Ext}^1_{\widetilde X/f} (\Omega_X , {\mathcal O}_{X,x}) = {\mathbb C} \cdot \partial
\end{equation}
where $\partial$ is the germ of an invariant derivation of ${\mathcal O}_{X,x}$, i.e. $z\partial/\partial z$ once we conjugate $f$ to $z \mapsto \lambda z$. This is equally true at Siegel discs or ${\mathcal O}_{NZ_{CR}}$, \eqref{I14}, so we recover,

\begin{example}\thlabel{thm:I4} (\cite[Theorem 1]{mitsu})
There is a deformation $f_{\varepsilon} \to f$ such that every non-repelling periodic cycle in $f$ is a limit of an attracting cycle of the same period.
\end{example}

This leaves us to address the difference between the left hand side of \eqref{I23} occasioned by real blowing up, where a simple calculation, \thref{cor:ct1}, reveals that the Leray spectral sequence is degenerate, i.e. we have an exact sequence,
\begin{equation}
\label{I25}
0 \longrightarrow {\rm Ext}^1_{X/f} (\Omega_X , {\mathcal O}_X) \longrightarrow {\rm Ext}^1_{\widetilde X/f} (\rho^* \, \Omega_X , {\mathcal O}_{\widetilde X}) \longrightarrow \prod_{\bullet} {\mathbb C} \partial_{\bullet} \longrightarrow 0
\end{equation}
where the product is taken over parabolic cycles and each $\partial_{\bullet}$ is the formally invariant field of \eqref{I11}. In particular the vanishing theorem can fail without real blowing up, and the smaller kernel in \eqref{I25} needn't surject onto the right hand side of \eqref{I23}. This can, however, be achieved in certain circumstances, to wit:

\begin{variant}\thlabel{thm:I5} (\thref{var:v4})
Factor the real blow up $\rho : \widetilde X \to X$ as, $\widetilde X \xrightarrow{ \, \overline\rho \, } \overline X \xrightarrow{ \, \sigma \, } X$, where the latter is the blow up in ``parabolic repelling'' cycles, i.e. ${\rm Re} (\nu) \geq 0$, \eqref{I11}, and suppose $\overline\rho (\widetilde R)$ is algebraic (i.e. finite) in a neighbourhood of such cycles then, unless $m=1$ and $f$ is a Latt\`es example, there is a surjection,
\begin{equation}
\label{I26}
{\rm Ext}^1_{\overline X / f} (\sigma^* \, \Omega_X^{\otimes m} , {\mathcal O}_{\overline X}) \twoheadrightarrow \overline{\rm Ext}^1_{\widetilde X / f} (\rho^* \, \Omega_X^{\otimes m} , \widetilde {\mathcal Q}_{\overline{\rm diag}}) \, .
\end{equation}
\end{variant}

A priori the algebraicity restriction of \thref{thm:I5} appears to be necessary since ``parabolic non-repelling'' is something of a misnomer. Specifically it really means that the Thurston norm of the dual of \eqref{I11} has more mass in attracting petals than non-repelling ones, or equivalently the dynamical residue of,
\begin{equation}
\label{I27}
\left\vert \frac{1+\nu z^{er}}{z^{{er}+1}} \right\vert^2 d \overline z \, dz
\end{equation}
is non-positive, so, parabolic non-repelling on average relative to the measure of \eqref{I27} is the more accurate description. However, more complicated measures than \eqref{I27}, can intervene locally once one allows $\widetilde R$ to be infinite, and they have no relation to the condition ${\rm Re}(\nu) \leq 0$. A specific example is provided in \thref{Dschol:2} with $r=e=1$, $\nu = 0$, and the meromorphic growth conditions of \thref{thm:v1}, but arbitrary dynamic residues.

\smallskip

Nevertheless when it comes to counting non-repelling cycles \thref{thm:I3} is always at least as good as \thref{thm:I5}, albeit for a non obvious reason. Specifically, Epstein, \cite[\S1]{adam}, has observed that the important quantity in such a count is infinite tails of critical orbits. This admits, however, \thref{defn:ct1}, the refinement,
\begin{equation}
\label{I29}
\begin{matrix}
\mbox{wild tails} &:= &\mbox{such infinite tails contained in the Julia set} \hfill \\
\mbox{tame tails} &:= &\mbox{orbits of the Fatou set, other than superattractors,} \hfill \\
&&\mbox{that contain at least one infinite tail.} \hfill
\end{matrix}
\end{equation}
and it's a theorem, \cite[Theorem A]{adamExtra} or \cite[Theorem 1]{berg}, that if, in the notation of \eqref{I11}, the tame ramification, counted with multiplicity, attracted to a parabolic cycle is exactly the number of attracting petals, $e_{\bullet}$, then it is parabolic repelling. This further fact implies that while at non-repelling parabolics the dimension of the left hand side of \eqref{I26} is $1$ less than that of \eqref{I24}, the right hand side of the latter is always at least $1$ more than that of the former. This aspect of counting is explained in detail in \thref{cor:v1} by way of the parameter,
\begin{equation}
\label{I28}
\delta_{\bullet} := \left\{\begin{matrix}
\mbox{$1$, if ${\rm Re}(\nu_{\bullet}) \leq 0$, and the ramification, with multiplicity,} \hfill \\
\mbox{of \thref{Rev:I1} (2) is exactly $e_{\bullet}$} \hfill \\
\mbox{$0$, otherwise.} \hfill
\end{matrix}\right.
\end{equation}
Similarly, at a Herman ring or Siegel disc, $\bullet$, we also, \thref{sum:ct1}, define,
\begin{equation}
\label{I29}
\varepsilon_{\bullet} = \left\{\begin{matrix}
\mbox{$1$, if $\bullet$ is a tame tail} \hfill \\
\mbox{$0$, otherwise.} \hfill
\end{matrix}\right.
\end{equation}

As such if in such cases we choose $D_{\bullet}$, resp. $A_{\bullet}$, of (III), resp. (III') to be a closed invariant disc, resp. annulus, which contains all such infinite tails of critical points, calculate the dimensions in \eqref{I23} or \eqref{I26}, and profit from \thref{bonus:v1}, that the kernels in the same acquire at least one dimension at any Herman ring, we obtain for any rational map of $\deg (f) > 1$,
\begin{eqnarray}
\label{I30}
&&\sum_{\bullet \in {\rm par}} (e_{\bullet} + \delta_{\bullet}) + \# (\mbox{non-repelling}) + \# ({\rm SD}) + \# ({\rm CR}) + 2 \# ({\rm HR}) \nonumber \\
&\leq &\# (\mbox{tame tails}) + \# (\mbox{wild tails})
\end{eqnarray}
wherein non-repelling excludes super attracting, while (SD), (HR) and (CR) are the obvious short hands for (III), (III'), and \eqref{I14}. In any case if one is prepared to suppose \thref{Rev:I1}, and \cite[Theorem A]{adamExtra}, this is equivalent to,
\begin{equation}
\label{I40}
\# ({\rm SD}) + \# ({\rm CR}) + 2 \# ({\rm HR}) \leq \sum_{\bullet \in {\rm SD}} \varepsilon_{\bullet} + \sum_{\bullet \in {\rm HR}} \varepsilon_{\bullet} + \# (\mbox{wild tails})
\end{equation}
wherein the presence of the $\varepsilon_{\bullet}$'s on the right results from the fact that locally about such dics or rings the extension \eqref{I22}, with $C_{\bullet}$ replaced by $D_{\bullet}$ or $A_{\bullet}$ as appropriate, is non-split in ${\rm Sh}_{X/f}$. However as Shisikura has observed, \cite[\S5]{mitsu}, one can remove, \thref{cor:v2}, $\varepsilon_{\rm SD}$ by using the perturbed map of \thref{thm:I4} which has one less wild tail for every Siegel disc. Regretably, however, the situation at Herman rings doesn't appear to admit a similar simplification. Certainly one can, \cite[\S6]{mitsu}, or, rather more simply, in the spirit of \cite[\S1]{milnor2}, \thref{cor:v3}, use the extra deformation of \thref{bonus:v1} to pinch Herman rings into Siegel discs, but this creates extra components and in so doing it may increase, albeit for somewhat stupid reasons, the number of wild tails. It cannot, however, increase the cardinality, $\# \{\mbox{wild ram}\}$, of the wild ramification, i.e. whose infinite tail is in the Julia set, counted without multiplicity, and so we obtain,

\begin{cor}\thlabel{cor:I6} (\thref{cor:v3})
Ler $f : {\mathbb P}^1_{\mathbb C} \to {\mathbb P}^1_{\mathbb C}$ be a rational map then,
\begin{equation}
\label{I41}
\# ({\rm SD}) + \# ({\rm CR}) + 2 \# ({\rm HR}) \leq \min \left\{\# (\mbox{wild ram}) , \sum_{\bullet \in {\rm HR}} \varepsilon_{\bullet} + \# (\mbox{wild tails})\right\}.
\end{equation}
\end{cor}

It's also easy enough to play with the degeneration argument, \thref{rmk:v4}, to see that various extremal cases of \eqref{I41} cannot occur, e.g., \cite[Theorem 3]{mitsu}, $\# ({\rm HR}) < \deg (f) - 1$, or \cite[Theorem A.1]{milnor2}, $2\#({\rm HR}) < \#({\rm wild ram})$, but in general, the obvious strategy for removing $\varepsilon_{\rm HR}$ is to redo everything with finitely many, rather than just $1$, rational map.

\smallskip

I am indebted to Adam Epstein for the many hours he has devoted to explaining to me transversality, and Thurston rigidity, to Jeremy Khan for curbing my ambition about what kind of invariant sets that I might apply to it, and to Jacopo Garofali for bringing to my attention that there was such a thing as a count of Herman rings. Nevertheless, the greatest debt is to C\'ecile whose impecable typesetting I have sorely missed.

\tableofcontents

%\end{document}

%%%%%%%%%%%%%%%

%\vglue 1cm
\newpage

\section{Groupoid of a rational map}\label{Sgrp}

Although of limited utility, it is convenient to have at our disposition the definition of a groupoid associated to the action of an endomorphism of a Riemann surface, to wit:

\begin{defn}\thlabel{Dgrp}
\nomenclature[D]{Groupoid \hyperref[Dgrp]{of self map}}{} 
Let $f : X \to X$ be a self map of a (not necessarily connected) Riemann surface, with $m,n \in {\mathbb Z}_{\geq0}$ then we define, a pure 1 dimensional (because it's defined by a Cartier divisor in $X \times X$) complex space:
\begin{equation}
\label{grp1}
R_{m,n} (f) = \{(x,y) : f^m (x) = f^n (y)\} \subseteq X \times X
\end{equation}
\nomenclature[N]{Arrows in a groupoid}{\hyperref[grp1]{$R$}$_{m,n} (f)$}
together with a composition, 
\nomenclature[N]{Composition in a groupoid}{\hyperref[grp2]{$C$}$^{k,\ell}_{m,n} (f)$}
\begin{equation}
\label{grp2}
\xymatrix{
R_{k,\ell} (f) \, {_{p_2}\times}_{p_1} R_{m,n} (f) \ar[rr]^-{C_{m,n}^{k,\ell}} &&R_{k+m,n+\ell} (f) : (x,y) \times (y,z) \mapsto (x,z) 
%R_{k,\ell} (f) \, {_{p_1}\times}_{p_2} R_{m,n} (f) \ar[u] \ar[r]_-{C_{m,n}^{k,\ell}} &R_{k+m , n+\ell} (f) \ar[u]
}
\end{equation}
wherein $p_1$, resp. $p_2$, is the first, resp. second projection, and, of course, $k , \ell \in {\mathbb Z}_{\geq 0}$. As such if we put the structure of a directed set on ${\mathbb Z}_{\geq 0}^2$ according to the rule,
\begin{equation}
\label{grp3}
(a,b) \longrightarrow (c,d) \quad \mbox{iff} \quad a \leq c \ \mbox{and} \ b \leq d \, .
\end{equation}
Then we have a directed system,
\begin{equation}
\label{grp4}
R_{a,b} (f) \longrightarrow R_{c,d} (f)
\end{equation}
for $(a,b) \to (c,d)$ as per \eqref{grp3} which not only respects the projections $p_i$, $i=1$ or $2$, i.e.
\begin{equation}
\label{grp5}
\xymatrix{
R_{a,b} (f) \ar[dd] \ar[rd]^{p_i} \\
&X \\
R_{c,d} (f) \ar[ru]_{p_i}
}
\end{equation}
commutes, but also the transposition,
\begin{equation}
\label{grp6}
\xymatrix{
R_{a,b} (f) \ar[d] \ar[rr]^{{\rm inv}_{a,b}} &&R_{b,a} (f) \ar[d] \\
R_{c,d} (f) \ar[rr]_{{\rm inv}_{c,d}} &&R_{c,d} (f)
} \quad (x,y) \mapsto (y,x)
\end{equation}
respects \eqref{grp3} as do the compositions \eqref{grp2}. Consequently to the directed limit,
\nomenclature[N]{All arrows in a groupoid}{\hyperref[grp7]{$R(f)$}}
\begin{equation}
\label{grp7}
R(f) := \varinjlim_{(a,b)} R_{a,b} (f)
\end{equation}
taken in the category of complex spaces there are maps,
\begin{equation}
\label{grp8}
R(f) \underset{p_i}{\longrightarrow} X \, , \ \mbox{resp.} \ R(f) \underset{\rm inv}{\longrightarrow} R(f) \, , \ \mbox{resp.} \ R(f) \, {_{p_1}\times}_{p_2} R (f) \underset{C}{\longrightarrow} R(f)
\end{equation}
deduced from \eqref{grp5}, resp. \eqref{grp6}, resp. \eqref{grp2}, along with a map, 
\begin{equation}
\label{grp9}
\Delta : X \longrightarrow R_{m,m} (f) \longrightarrow R(f) : x \mapsto (x,x) \, m \in {\mathbb Z}_{\geq 0}
\end{equation}
which determine the structure of a groupoid on $R(f)$ with source $p_1$, sink $p_2$, inverse inv, composition $C$, and identity $\Delta$.
\end{defn}

We can, therefore, usefully note the behaviour of \thref{Dgrp} under change of the objects, $X$, i.e.

\begin{fact}\thlabel{F1grp}
Suppose that we have a fibre square,
\begin{equation}
\label{grp9bis}
\xymatrix{
X' \ar[d] \ar[r]^{f'} &X' \ar[d] \\
X \ar[r]_f &X
}
\end{equation}
with $f$ nowhere locally constant. Then the groupoid $R(f') \rightrightarrows X'$ of \thref{Dgrp} maps to the groupoid deduced by refinement of objects, i.e. the fibre,
\nomenclature[D]{Groupoid, \hyperref[F1grp]{refinement of objects}}{} 
\begin{equation}
\label{grp10}
\xymatrix{
R(f)' \ar[d] \ar[rr] &&X' \times X' \ar[d] \\
R(f) \ar[rr]_-{p_1 \times p_2} &&X \times X
}
\end{equation}
and is an isomorphism whenever $X' \to X$ is an open embedding.
\end{fact}

\begin{proof}
We have in the notation of \eqref{grp1} a diagram in which every square is fibered,
\begin{equation}
\label{grp11}
\xymatrix{
&X' \times X' \ar[dl] \ar@{-}@<-1ex>[d]\ar@{}[d]^{\!\!\!\!(f')^m \!\times \!(f')^n} &&R_{mn}(f)' \ar[dll] \ar[ll] \ar[dd] &R_{mm}(f)' \ar[l] \ar[dd] \\
X \times X \ar[dd]_{f^m \times f^n} &R_{mn}(f) \ar[l] \ar@<-1ex>[d] \ar@{-}[d] \\
&X' \times X' \ar[ld] \ar@{}@<-1ex>[d] \ar[d] && X' \times_{\Delta} X' \ar@{-}[ll] \ar[dll] &\Delta_{X'} \ar[l] \\
X \times X &\Delta_X \ar[l]
}
\end{equation}
Consequently from the rightmost top arrow and $f$ nowhere locally constant we obtain maps,
\begin{equation}
\label{grp12}
R_{mn} (f') \longrightarrow R_{mn} (f)' \longrightarrow R_{mn} (f)
\end{equation}
which proves the first part, while $\Delta_{X'} = X' \times_{\Delta} X'$ as soon as $X' \to X$ is an embedding, and whence the second part.
\end{proof}

Needless to say following the standard procedure of \cite[Expos\'e VI]{sga1} we can associate the classifying champ, $[X/R(f)]$, in the 2-category of 2-sheaves in analytic spaces, to the groupoid $R(f) \rightrightarrows X$ of \thref{Dgrp}, and we have the following usefull observation,

\begin{fact}\thlabel{F2grp}
If $f : X \to X$ is a (not necessarily finite) \'etale covering then $X \to [X/R(f)]$ is too.
\end{fact}

\begin{proof}
In the category of differentiable manifolds the composition of coverings is a covering, and, of course, its a property stable under base change, so from the fibre square,
\begin{equation}
\label{grp13}
\xymatrix{
X \ar[d]_{f^m} &R_{m,n}(f) \ar[l] \ar[d] \\
X &X\ar[l]^{f^n}
}
\end{equation}
both projections of $R_{m,n}(f)$ of \eqref{grp2} are coverings, so $R_{m,n} (f)$ is a (not necessarily connected) Riemann surface. Better still, by \eqref{grp7}, $R(f)$ can be identified (as a Riemann surface) with a discrete union of components of,
\begin{equation}
\label{grp14}
\coprod_{m,n \geq 0} R_{m,n}(f)
\end{equation}
so either projection of $R(f) \to X$ is a covering, which in turn is exactly the fibre of $X \to [X/R(f)]$ over $X$.
\end{proof}

In certain classical situations of interest the conditions of \thref{F2grp} can be guaranteed, to wit:

\begin{cor}\thlabel{C1grp}
If $f : X \to X$ is proper and \'etale then it is a covering. In particular therefore if $X$ is an embedded open subset of a Riemann surface and $f$ extends to a proper map $\overline f : \overline X \to \overline X$ which is \'etale over $X$ with $\overline f^{-1} (X) = X$, then, $X \to [X/R(f)]$ is a covering.
\end{cor}

\begin{proof}
Proper \'etale maps are coverings, and we're supposing $f$ is the base change of a proper map so \thref{F2grp} applies.
\end{proof}

We can therefore apply this in the obvious way, i.e.

\begin{fact}\thlabel{F3grp}
Suppose $\overline f : \overline X \to \overline X$ is a proper map of a Riemann surface, and that for some attracting (but not super attracting), resp. point $x$ of period $p$ whose return map is tangent to the identity, the connected component containing $x$ of the basin of attraction, resp. the connected component of the basin of an attracting petal containing the petal, \thref{DRev:2}, $X'$, doesn't contain a critical point then the pair $(X', \overline f^p \mid_{X'})$ is isomorphic to ${\mathbb C}^{\times}$, resp. ${\mathbb C}$ with $x$ the point at $\infty$ and $f^p$ the map $s \mapsto \lambda^{-1} s$, resp. $s+1$, for $\lambda$ the multiplier of $f^p$ at $x$ in the former case.
\end{fact}

\begin{proof}
Let $V$ be a sufficiently small neighbourhood of $x$, resp. the petal close to $x$, then by definition $X'$ is the connected component of the total basin,
\begin{equation}
\label{grp15}
X := \coprod_{n \geq 0} \, \overline{f}^{-n} (V)
\end{equation}
containing $V$, while by \thref{C1grp} $X \to [X/R(f)] = [V/R_V]$, for $R_V \rightrightarrows V$ the induced groupoid, is a covering. As such $X' \to [V/R_V]$ is a connected component of a covering, so it, in turn, is a covering. On the other hand over $V$ some iterate of $f$ may, in a suitable local coordinate $s$, be expressed as,
\begin{equation}
\label{grp16}
s \longmapsto \lambda s \, , \quad \mbox{resp.} \ s \longmapsto s+1
\end{equation}
so $[V/R_V]$ is an elliptic curve, resp. ${\mathbb C}^{\times}$, of which $X'$ is an infinite covering, and whence it's ${\mathbb C}$ or ${\mathbb C}^{\times}$, resp. ${\mathbb C}$. Better still if $\Gamma$ is the covering group then,
\begin{equation}
\label{grp17}
R' := X' \times_{[V/R_V]} X' = R(f) \times_{X \times X} X' \times X' = X' \times \Gamma
\end{equation}
of which a connected component distinct from the diagonal is $f^{-p} (X') \cap X'$, while $f^{-q} (X') \cap X'$ is empty for $0 < q < p$. Consequently, $f^{-p} (X') \hookleftarrow X'$, $f^p \!\!\mid_{X'}$ is an element $f'$ of $\Gamma'$, and
\begin{equation}
\label{grp18}
R' = R(f') 
\end{equation}
in the sense of \eqref{grp7}. In particular, therefore, $X'$ cannot be ${\mathbb C}$ in the attracting case since $f'$ extends over $x$ with the multiplier $\lambda \ne 1$, which equally shows that if we identify $x$ with the point at $\infty$ then $f' = f^p \!\!\mid_{X'}$ is $s \mapsto \lambda^{-1} s$, while in the parabolic case there is only one possibility for $f'$.
\end{proof}

Needless to say the principle interest is rational maps, where,

\begin{cor}\thlabel{C2grp}
Let $f : {\mathbb P}^1 \to {\mathbb P}^1$ and $x$ an attracting (but not super attracting), resp. parabolic, fixed point then if $\deg (f) > 1$, at least $1$, resp. at least $e$ (the number of attracting petals), critical points converge to $x$, under forward itteration, i.e. $f^n(c) \to x$.
\end{cor}

\begin{proof}
By \thref{F3grp} if the conclusion is false then for $x$ identified with the point at $\infty$, $f^p \mid_{{\mathbb P}^1 \backslash \{\infty\}}$, resp. $f^{pr}$ for a return map tangent to a rotation of order $r$ is $s \mapsto \lambda^{-1} s$, resp. $s \mapsto s+1$, and whence $f$ has degree $1$.
\end{proof}

For periodic points not lying in the Julia set there are some other cases of interest where the argument of \thref{F3grp} applies but it gives less, so we confine their treatment to,

\begin{scholion}\thlabel{Sgrp}
If $x$ is a super attractor, resp. the centre of a Siegel disc then for any sufficiently small punctured neighbourhood, resp. neighbourhood $V$ of $x$, the classifying champ $[V/R_V]$ is isomorphic to that of the upper half plane $[H/\Gamma]$ for $\Gamma$ either isomorphic to,
\begin{equation}\label{grp19}
\begin{split}
{\mathbb Z}^2 & \ \mbox{generated by 2 parabolic elements, Siegel disc} \\
{\mathbb Z} \ltimes {\mathbb Z} & \ \mbox{for ${\mathbb Z}$ acting on ${\mathbb Z}$ by multiplication the order $n \in {\mathbb Z}_{>1}$}
\end{split}
\end{equation}
of the super attractor, to which we might as well add Herman rings, since these have the form $[H/\Gamma]$ for $\Gamma$ isomorphic to,
\begin{equation}
\label{grp20}
{\mathbb Z}^2 \ \mbox{and generated by 2 hyperbolic elements.}
\end{equation}

Now the argument of \thref{F3grp}, in the notation of the same, shows that,
\begin{equation}
\label{grp21}
X' \longrightarrow [H/\Gamma]
\end{equation}
is a $\Gamma' = {\mathbb Z}$ covering by a space, i.e. we have a short exact sequence,
\begin{equation}
\label{grp22}
0 \longrightarrow \Gamma'' \longrightarrow \Gamma \longrightarrow \Gamma' \longrightarrow 0
\end{equation}
which in the case of Herman rings and Siegel discs doesn't mean much beyond the definition since $V=X'$, albeit in the case of a super attractor it says that $X'$ is isomorphic to a punctured unit (as opposed to possibly smaller radius) disc with $f$ conjugate to $s \mapsto s^n$. The statement of \thref{F2grp} is, however, a little stronger since it says the whole basin, \eqref{grp15}, $X$ is a covering. Consequently as a subspace of a compactification $\overline X$, $X$ is topologically a direct sum,
\begin{equation}
\label{grp23}
\coprod_{\alpha}  X_{\alpha}
\end{equation}
where each $X_{\alpha} \to [H/\Gamma]$ is a covering, and we assert, 

\begin{claim}\thlabel{grpC}
Under the hypothesis of \thref{F2grp} every $X_{\alpha}$ in \eqref{grp23} has the form $f^{-n} X'$ for some $n > 0$ and,
\begin{equation}
\label{grp24}
f^n : X_{\alpha} \longrightarrow X'
\end{equation}
is a finite \'etale cover of a punctured disc, \eqref{grp19}, or an annulus, \eqref{grp20}.
\end{claim}

\begin{proof}
By definition there is a smallest $n > 0$ such that $X_{\alpha} \cap f^{-n} (X')$ is a finite \'etale cover of $X'$. Now suppose it isn't everything then some subsequence $x_k$ in $X_{\alpha} \cap f^{-n} (X')$ converges to a boundary, $\xi$, of the same with $\xi \in X_{\alpha}$. Further since $f$ is proper, $f^n(x_k)$ must converge to a boundary point of $X'$. However by \eqref{grp22} $f' = f^p$ maps $\partial X'$ to $X'$, while by definition there is some $m > n$ such that $f^m (\xi) \in X'$, which is nonsense.
\end{proof}
\end{scholion}

%\end{document}

%%%%%%%%%%%%%%%

\vglue 1cm

\section{Cohomology and Duality}\label{SS:cd}

In the presence of a group action we have the notion of an equivariant sheaf and, slightly more generally,

\begin{defn}\thlabel{cd:def1}
\nomenclature[D]{Sheaves \hyperref[cd:def1]{invariant by a monoid}}{} 
Let $\Sigma$ be a discrete monoid acting on a topological space, or more generally a site, $X$ by continuous, resp. geometric, morphisms then the category, ${\rm Sh}_{X/\Sigma}$, of $\Sigma$-equivariant sheaves are sheaves ${\mathcal F}$ on $X$ such that for all $\sigma \in \Sigma$ there is a map,
\nomenclature[N]{Category of sheaves invariant by a monoid}{\hyperref[cd:def1]{${\rm Sh}$}$_{X/\Sigma}$}
\begin{equation}
\label{cd1}
\varphi_{\sigma} : \sigma^* {\mathcal F} \longrightarrow F
\end{equation}
satisfying the co-cycle condition; $\varphi_{\sigma\tau} = \varphi_{\tau} \, \tau^* \varphi_{\sigma}$, $\forall \, \sigma , \tau \in \Sigma$. In the particular case, therefore, of $\Sigma = {\mathbb N}$ acting by an endomorphism $f$ \eqref{cd2} is determined by,
\begin{equation}
\label{cd2}
\varphi := \varphi_1 : f^* {\mathcal F} \longrightarrow {\mathcal F} \, , \quad \varphi_{n+1} = \varphi_n \, f^* \varphi_1 \, , \ n \geq 1
\end{equation}
and we write ${\rm Sh}_{X/f}$ 
\nomenclature[N]{Category of $f$-sheaves}{\hyperref[cd:def1]{${\rm Sh}$}$_{X/f}$}
rather than ${\rm Sh}_{X/\Sigma}$ while refering to the objects of the category as $f$-sheaves, while maps between sheaves are natural transformations of functors on viewing \eqref{cd1} as such from ${\rm B}_{\mathbb N}$ to sheaves on $X$. Similarly we define a category, ${\rm Sh}'_{X/f}$, 
\nomenclature[N]{Category of almost $f$-sheaves}{\hyperref[cd:def1]{${\rm Sh}'$}$_{X/f}$}
of almost $f$-sheaves as the category of pairs of sheaves ${\mathcal F}_0 , {\mathcal F}_1$ with maps,
\begin{equation}
\label{cd3}
{\mathcal F}_1 \underset{s}{\longleftarrow} {\mathcal F}_0 \underset{t}{\longrightarrow} f_* \, {\mathcal F}_1 \, .
\end{equation}
\nomenclature[D]{Almost \hyperref[cd3]{$f$-sheaf}}{}
Plainly, by adjunction, every $f$-sheaf is an almost $f$-sheaf with ${\mathcal F}_0 = {\mathcal F}_1$, $s = {\rm id}$, $t = \varphi$, but, equally plainly, there are many more almost $f$-sheaves.
\end{defn}

As ever we require to check,

\begin{fact}\thlabel{cd:fact1}
Let everything be as in \thref{cd:def1} then, the category of $\Sigma$-equivariant sheaves in abelian groups has enough injectives.
\end{fact}

\begin{proof}
Let ${\mathcal F} \xrightarrow{ \ i \ } {\mathcal I}$ be an inclusion of ${\mathcal F}$  in an injective and let $\varphi_{\sigma}$ equally be the adjoint of \eqref{cd1}, then we have maps,
\begin{equation}
\label{cd4}
{\mathcal F} \underset{\varphi_{\sigma}}{\longrightarrow} \sigma_* \, {\mathcal F} \underset{\sigma_* i}{\longrightarrow} \sigma_* \, {\mathcal I}
\end{equation}
so that all together we get maps,
\begin{equation}
\label{cd5}
(\sigma_* \, i \, \varphi_{\sigma})_{\sigma \in \Sigma} : {\mathcal F} \longrightarrow {\mathcal J} := \prod_{\sigma \in \Sigma} \sigma_* \, {\mathcal I} \, , \quad j_{\sigma} : \prod_{\tau \in \Sigma} \tau_* \, {\mathcal I} \longrightarrow \prod_{\tau \in \sigma\Sigma} \tau_* \, {\mathcal I} 
\end{equation}
where the letter defines (by adjunction) a $\Sigma$-equivariant sheaf. On the other hand for any $\Sigma$-sheaf ${\mathcal A}$,
\begin{equation}
\label{cd6}
{\rm Hom}_{{\rm Sh}_{X/\Sigma}} ({\mathcal A}, {\mathcal J}) = {\rm Hom}_{\rm Sh} ({\mathcal A}, {\mathcal I}) : a \longmapsto (\sigma_* \, a \, \alpha_{\sigma})_{\sigma \in \Sigma}
\end{equation}
where $\sigma \mapsto \alpha_{\sigma}$ is the action on ${\mathcal A}$, so ${\mathcal J}$ is injective in ${\rm Sh}_{X/\Sigma}$ because ${\mathcal I}$ was injective to start with.
\end{proof}

In the same vein we also have,

\begin{fact}\thlabel{cd:fact2}
The category of almost $f$-sheaves has enough injectives.
\end{fact}

\begin{proof}
Given any sheaf ${\mathcal G}$ on $X$ we can construct an almost $f$-sheaf ${\mathcal G}_f$ by way of,
\begin{equation}
\label{cd6.bis}
{\mathcal G} \underset{s}{\longleftarrow} {\mathcal G} \prod f_* \, {\mathcal G} \underset{t}{\longrightarrow} f_* \, {\mathcal G} \, .
\end{equation}
In particular, therefore, for any almost $f$-sheaf, $\underline{\mathcal F}$,
\begin{equation}
\label{cd6.ter}
{\rm Hom}_{{\rm Sh}_{X/f}} (\underline{\mathcal F} , {\mathcal G}_f) = {\rm Hom}_{{\rm Sh}_X} ({\mathcal F}_1 , {\mathcal G}) : a \longmapsto (a \, s , f_* \, a \, t) \,.
\end{equation}
Similarly if ${\mathcal G}_0 , {\mathcal G}_1$ are any sheaves then we can make an almost $f$-sheaf,
\begin{equation}
\label{cd7}
{\mathcal G}_f := ({\mathcal G}_1)_f \prod_{{\rm Sh}_{X/f}} ({\mathcal G}_0 , 0) \, .
\end{equation}
So that more generally,
\begin{equation}
\label{cd8}
{\rm Hom}_{{\rm Sh}'_{X/f}} (\underline{\mathcal F} , {\mathcal G}_f) = {\rm Hom}_{{\rm Sh}_X} ({\mathcal F}_0 , {\mathcal G}_0) \prod {\rm Hom}_{{\rm Sh}_X} ({\mathcal F}_1 , {\mathcal G}_1)
\end{equation}
and whence, ${\mathcal G}_f$ is injective in ${\rm Sh}'_{X/f}$ if ${\mathcal G}_0$ and ${\mathcal G}_1$ are in ${\rm Sh}_X$.
\end{proof}

Before progressing let us make,

\begin{rmk}\thlabel{cd:rmk1}
One can, of course, apply Giraud's axioms, \cite[Theorem 3]{gir}, to verify that $\Sigma$-sheaves and almost $f$-sheaves are topoi. Slightly more informatively, however, one can construct an explicit and elegant underlying site, \cite{jacopo}, whose sheaves are the topoi in question. Indeed if \'Et$_X$ were the Grothendieck topology of local homeomorphisms $U \xrightarrow{ \, u \, } X$ to a topological space then the category of $\Sigma$-equivariant opens, i.e. opens $U$ admitting maps,
\begin{equation}
\label{cd9}
\varphi_{\sigma} : \sigma^{-1} U \longrightarrow U
\end{equation}
such that $\varphi_{\sigma\tau} = \varphi_{\tau} \tau^* \varphi_{\sigma}$, will do.
\end{rmk}

Plainly, therefore, we get various spectral sequences of which the most pertinent is,

\begin{fact}\thlabel{cd:fact3}
Let everything be as in \thref{cd:def1} with ${\mathcal F}$ any $f$-sheaf of abelian groups, and ${\mathcal G} = ({\mathcal G}_0 , {\mathcal G}_1, s , t)$ an almost $f$-sheaf then there is a spectral sequence,
\begin{equation}
\label{cd10}
{\rm E}_1^{p,q} : {\rm Ext}_X^q ({\mathcal F} , {\mathcal G}_0) \underset{d_1^{0,q}}{\longrightarrow} {\rm Ext}_X^q (f^* {\mathcal F} , {\mathcal G}_1) \, , \quad p = 0 \ \mbox{or} \ 1
\end{equation}
where the differentials in degree $q > 0$ are derived from,
\begin{equation}
\label{cd11}
d_1^{0,0} : {\rm Hom}_X ({\mathcal F} , {\mathcal G}_0) \longrightarrow {\rm Hom}_X (f^* {\mathcal F} , {\mathcal G}_1) : a \longmapsto (sa) \cdot \varphi - t f^* a
\end{equation}
which abuts to ${\rm Ext}_{X/f}^{p+q} ({\mathcal F} , {\mathcal G})$, 
\nomenclature[N]{Ext of (almost) $f$-sheaves}{\hyperref[cd:fact3]{${\rm Ext}$}$^q_{X/f} ({\mathcal F} , {\mathcal G})$}
i.e. the derived functors of ${\rm Hom}_{X/f}$, where, as the notation suggests, these don't depend on whether they're understood in ${\rm Sh}'_{X/f}$ or ${\rm Sh}_{X/f}$ if ${\mathcal G}$ itself belonged to the latter.
\end{fact}

\begin{proof}
The functor ${\rm Hom}_{X/f}$ can be written as the composition of the functors,
\begin{equation}
\label{cd12}
{\rm Ab}'_{X/f} \xrightarrow{ \ \Theta \ } C_+ ({\rm Ab}) : {\mathcal G} \longmapsto \left[ {\rm Hom}_X ({\mathcal F} , {\mathcal G}_0) \xrightarrow{ \ d_1^{0,0} \ } {\rm Hom}_X (f^* {\mathcal F} , {\mathcal G}_1) \right]
\end{equation}
for $C_+ ({\rm Ab})$ the category of complexes, in non-negative degrees of abelian groups and $d_1^{0,0}$ as per \eqref{cd11} followed by ${\rm Hom} ({\mathbb Z} , \ )$ in $C_+ ({\rm Ab})$. Both of these functors are left exact while $R^q (\Theta)$ is the complex of length 2 given by \eqref{cd10} since, for example, it's plainly a $\delta$-functor and vanishes on the injectives constructed in \eqref{cd7}-\eqref{cd8}. Similarly the value of $\Theta$ on such injectives is an acyclic complex (in fact with 2-terms, and the differential even has a section) so the spectral sequence \eqref{cd10} is exactly that of a composition of functors whose first member sends injectives to acyclics.
\end{proof}

A trivial, but usefull corollary to which is,

\begin{cor}\thlabel{cd:cor1.bis}
Let everything be as in \thref{cd:fact3} and suppose in addition that ${\mathcal G}$ is an $f$-sheaf supported in a forward invariant set $Z$ such that $f\vert_Z$ is an isomorphism and the actions, \eqref{cd1}, for both ${\mathcal F}$ and ${\mathcal G}$ are invertible at $Z$ then there is a ${\mathbb Z}$-action on ${\rm Ext}_X^q({\mathcal F},{\mathcal G})$ and the abutment of \eqref{cd10} is,
\begin{equation}
\label{cd112}
{\rm E}_{\infty}^{p,q} = H^p ({\mathbb Z} , {\rm Ext}_X^q ({\mathcal F} , {\mathcal G})) \, .
\end{equation}
In particular if ${\mathbb Z}$ is a cycle of period $p$, i.e. $f$ permutes, the components of a decomposition $Z = Z_1 \coprod \cdots \coprod Z_p$, then for any $1 \leq i \leq p$,
\begin{equation}
\label{cd212}
{\rm Ext}_{X/f}^n ({\mathcal F},{\mathcal G}) = {\rm Ext}_{X/f^p}^n ({\mathcal F},{\mathcal G}\vert_{Z_i}) \quad n \geq 0 \, .
\end{equation}
\end{cor}

\begin{proof}
Under the hypothesis that ${\mathcal G}$ is supported on $Z$, we can replace ${\mathcal F}$ by ${\mathcal F}\vert_Z$ so everything has an invertible ${\mathbb Z}$ action, while the difference between \eqref{cd10} and \eqref{cd112} is just the definition of group cohomology. In particular, therefore, the groups on the left of \eqref{cd212} are equally,
\begin{equation}
\label{cd312}
{\rm Ext}_{[X/{\mathbb Z}]}^n ({\mathcal F}\vert_Z , {\mathcal G})
\end{equation}
for $[X/{\mathbb Z}]$ the classifying champ of the ${\mathbb Z}$-action, which, by hypothesis, can be sliced along any of the inclusions $Z_i \hookrightarrow Z$.
\end{proof}

Now while what we will actually need is a duality theorem for $f$-sheaves on real blow ups, from a homological algebra point of view there's no substantive difference with extending duality from complex manifolds to $f$-sheaves on the same, so we begin with:

\begin{fact}\thlabel{cd:fact4}
Let everything be as in \thref{cd:fact3}, and suppose in addition that $f$ is an endomorphism of a complex manifold $X$ with ${\mathcal F} , {\mathcal G}_0 , {\mathcal G}_1, \varphi , s , t$ of \thref{cd:def1}, sheaves of (not necessarily coherent) ${\mathcal O}_X$-modules and maps between them, then for any $0 \leq q \leq \dim X$ there is a map,
\begin{equation}
\label{cd13}
d_1 : {\rm Hom}_{{\mathcal O}_X} ({\mathcal F} , {\mathcal G}_0) \otimes_{{\mathcal O}_X} {\mathcal A}_X^{0,q} \longrightarrow {\rm Hom} (f^* {\mathcal F} , {\mathcal G}_1) \otimes_{{\mathcal O}_X} {\mathcal A}_X^{0,q} 
\end{equation}
$$
: a \otimes \omega \longmapsto s(a) \varphi \otimes \omega - tf^* a \otimes f^* \omega
$$
where ${\mathcal A}_X^{p,q}$ 
\nomenclature[N]{Sheaf of forms of type $(p,q)$}{\hyperref[cd:fact4]{${\mathcal A}$}$^{p,q}_X$}
is, as usual, the sheaf of smooth $(p,q)$ forms, such that the global sections over $X$ of the total complex of the bi-complex defined by \eqref{cd13} and the $\overline\partial$-resolution ${\mathcal A}^{0,\bullet}$ calculates ${\rm Ext}_{X/f}^q ({\mathcal F} , {\mathcal G})$.
\end{fact}

\begin{proof}
Pull back of differential forms affords the structure of an $f$-sheaf on ${\mathcal A}_X^{0,q}$, while, plainly, ${\mathcal O}_X$ is also an $f$-sheaf, so, understood in the category of almost $f$-sheaves,
\begin{equation}
\label{cd14}
{\mathcal G} \longrightarrow {\mathcal G} \otimes_{{\mathcal O}_X} {\mathcal A}_X^{0,\bullet}
\end{equation}
is a quasi isomorphism in the category of complexes of almost $f$-sheaves. As such we have an isomorphism,
\begin{equation}
\label{cd15}
{\rm Ext}_{X/{\mathcal F}}^q ({\mathcal F} , {\mathcal G}) \longrightarrow {\mathbb E}{\rm xt}_{X/{\mathcal F}}^q ({\mathcal F} , {\mathcal G} \otimes_{{\mathcal O}_X} {\mathcal A}_X^{0,\bullet}) \, .
\end{equation}
Now to calculate the hyperexts on the right of \eqref{cd15}, we can proceed exactly as in the proof of \thref{cd:fact3} except that $\Theta$ takes values in the category of bi-complexes. However exactly as in \eqref{cd13} tensoring with differentials is distributive over Hom, so what we need to calculate, in the category of bi-complexes, is $R^0 {\rm Hom} ({\mathbb Z} , \ )$ of,
\begin{equation}
\label{cd16}
\Gamma (d_1) : \Gamma (X , {\rm Hom}_{{\mathcal O}_X} ({\mathcal F} , {\mathcal G}_0) \otimes_{{\mathcal O}_X} {\mathcal A}_X^{0,q}) \longrightarrow \Gamma (X , {\rm Hom}_{{\mathcal O}_X} (f^* f , {\mathcal G}_1) \otimes_{{\mathcal O}_X} {\mathcal A}_X^{0,q})
\end{equation}
which is exactly the cohomology of the bi-complex defined by $\Gamma (d_1)$ and $\overline\partial$.
\end{proof}

In exactly the same vein, we also have,

\begin{var}\thlabel{cd:fact5}
Suppose instead $X = \widetilde X$ is the real blow up of a Riemann surface with ${\mathcal F} , {\mathcal G}_0 , {\mathcal G}_1$ sheaves of ${\mathcal O}_{\widetilde X}$ modules in the sense of \thref{Dc:not1} then the same conclusion holds provided we replace the Dolbeault complex by the complex,
\begin{equation}
\label{cd17}
{\mathcal A}_{\widetilde X}^0 \xrightarrow{ \ \overline\partial \ } {\mathcal A}_{\widetilde X}^{0,1} (\log E) 
\end{equation}
of \eqref{Dc5}.
\end{var}

\begin{proof}
Identical to that of \thref{cd:fact4}.
\end{proof}

As it happens our use of duality will be confined to complements of cycles on compact real blow ups, nevertheless it's appropriate to state things in greater generality to wit,

\begin{defn}
\thlabel{cd:def2}
Let everything be as in \thref{cd:def1} but with $f$ a proper endomorphism of a locally compact space, and for any closed subset $i_K : K \hookrightarrow X$ let $j_K : U \hookrightarrow X$ be the complementary open. In particular, therefore, if ${\mathcal F}$ is an $f$-sheaf, there is an almost $f$-sheaf,
\begin{equation}
\label{cd18}
\xymatrix{
f^* \left( (i_K)_* \, i_K^* \, {\mathcal F}\right) \ar[r] &(i_{f^{-1} (K)})_* \, i^*_{f^{-1} (K)} (f^* {\mathcal F}) \ar[r]^{\varphi} &(i_{f^{-1}(K)})_* \, i^*_{f^{-1} (K)} \, {\mathcal F} \ar[d] \\
&(i_K)_* \, i_K^* \, {\mathcal F} \ar[r] &(i_{K \cap f^{-1} (K)})_* \, i^*_{f^{-1} (K) \cap K} \, {\mathcal F} \, .
}
\end{equation}
Better still since $K \cap f^{-1} (K) \hookrightarrow K$ is compact if $K$ is, the direct limit of \eqref{cd18} over all compacts is an $f$-sheaf, so for ${\mathcal F}$ a sheaf of sets or abelian groups we get an exact sequence of $f$-sheaves,
\begin{equation}
\label{cd19}
0 \longrightarrow \varinjlim_K \, (j_K)_! \, j_K^* \, {\mathcal F} \longrightarrow {\mathcal F} \longrightarrow \varinjlim_K \, (i_K)_* \, i_K^* {\mathcal F} \longrightarrow 0
\end{equation}
the right most term of which we denote by ${\mathcal F}_c$. In particular, therefore, for any almost $f$-sheaf, ${\mathcal G}$, we have a sub-functor,
\begin{equation}
\label{cd20}
{\rm Hom}_c^{X/f} ({\mathcal F} , {\mathcal G}) := {\rm Hom}_{X/f} ({\mathcal F}_c , {\mathcal G}) \subset {\rm Hom}_{X/f} ({\mathcal F},{\mathcal G})
\end{equation}
\nomenclature[N]{Compactly supported maps of $f$-sheaves}{\hyperref[cd20]{${\rm Hom}$}$^{X/f}_{c} ({\mathcal F} , {\mathcal G})$}
of sections with compact support, and we denote the derived functors of the leftmost term in \eqref{cd20} by ${\rm Ext}_q^{X/f} ({\mathcal F},{\mathcal G})$.
\end{defn}
\nomenclature[N]{Compactly supported Ext of $f$-sheaves}{\hyperref[cd:def2]{${\rm Ext}$}$^{X/f}_{q} ({\mathcal F},{\mathcal G})$}

Combining these observations with our considerations to date, we have:

\begin{cor}\thlabel{cd:cor1}
For sheaves on a locally compact space, $X$, let ${\rm Ext}_c^q$ be the derived functors of maps ${\rm Hom}_c$, with compact support, then if everything else is as in \thref{cd:fact3}, albeit with $f$ proper, there is a spectral sequence enjoying,
\begin{equation}
\label{cd21}
{\rm E}_1^{p,q} : {\rm Ext}_c^q ({\mathcal F} , {\mathcal G}_0) \underset{d_1^{0,q}}{\longrightarrow} {\rm Ext}_c^q (f^* {\mathcal F} , {\mathcal G}_1) \, , \quad p = 0 \ \mbox{or} \ 1
\end{equation}
converging to ${\rm Ext}_{p+q}^{X/f} ({\mathcal F},{\mathcal G})$, wherein $d_1^{0,q}$ is derived from \eqref{cd11} with compact support.
\end{cor}

\begin{proof}
Apply \thref{cd:fact3} with ${\mathcal F} = {\mathcal F}_c$, cf. \eqref{cd19} et sq.
\end{proof}

For much the same reasons we also have,

\begin{cor}\thlabel{cd:cor2}
Let everything be as in \thref{cd:fact4}, resp. \thref{cd:fact5}, then ${\rm Ext}_{\bullet}^{X/f}$ $({\mathcal F},{\mathcal G})$ are the cohomology groups of the double complex associated to $\overline\partial$ and
\begin{equation}
\label{cd22}
\Gamma_c (d_1) : \Gamma_c (X , {\rm Hom}_{{\mathcal O}_X} ({\mathcal F} , {\mathcal G}_0) \otimes_{{\mathcal O}_X} {\mathcal A}_X^{0,\bullet}) \rightarrow \Gamma_c (X,{\rm Hom}_{{\mathcal O}_X} (f^* {\mathcal F},{\mathcal G}_1) \otimes_{{\mathcal O}_X} {\mathcal A}_X^{0,\bullet}) 
\end{equation}
resp. replace the Dolbeault complex by \eqref{cd17}.
\end{cor}

\begin{proof}
Apply \thref{cd:fact4}, resp. \thref{cd:fact5}, with ${\mathcal F} = {\mathcal F}_c$.
\end{proof}

In order to put all of this together we need,

\begin{lem}\thlabel{cd:lem1}
Let $E^{\bullet}$ be a finite complex of LF-spaces (i.e. a direct limit of Fr\'echet spaces) with continuous differentials such that the (naive) cohomology groups $H^{\bullet} (E^{\bullet})$ are finite dimensional then the $H^{\bullet} (E^{\bullet})$ are also the cohomology groups in the category of LF-spaces, i.e. the differentials are closed.
\end{lem}

\begin{proof}
By definition for each $i$ such that $E^i \ne 0$ we have a closed subspace $Z^i \subset E^i$ arising from the kernel of the differential and an exact sequence of vector spaces,
\begin{equation}
\label{cd23}
E^{i-1} \xrightarrow{ \ d \ } Z^i \longrightarrow H^i \longrightarrow 0 \, .
\end{equation}
In particular if we choose a finite dimensional (and necessarily closed) subspace $\widetilde H^i \subset Z^i$ lifting the quotient, then we have a surjection of vector spaces,
\begin{equation}
\label{cd24}
E^{i-1} \oplus \widetilde H^i \twoheadrightarrow Z^i \, .
\end{equation}
On the other hand by \cite[prop. 5]{DS}  $Z^i$ is a LF space, and whence, if much less trivially, by op. cit. Th\'eor\`eme 1, \eqref{cd24} is a topological surjection, from which, $d$ in \eqref{cd23} is closed.
\end{proof}

At the same time it's equally true,

\begin{lem}\thlabel{cd:lem2}
If $E^{\bullet}$ is a finite complex of locally convex topological vector spaces with closed differentials, and $(E^{\bullet})^{\vee}$ is its strong dual then $H^{\bullet} ((E^{\bullet})^{\vee})$ is the strong dual of $H^{\bullet} (E^{\bullet})$. In particular under the hypothesis of \thref{cd:lem1}, $H^{\bullet} ((E^{\bullet})^{\vee})$ is the separated dual of $H^{\bullet} (E^{\bullet})$.
\end{lem}

\begin{proof}
In the category of locally convex TVS, dualising is an exact functor by Hahn-Banach.
\end{proof}

Putting all of this together we therefore obtain,

\begin{fact}\thlabel{cd:fact6}
Let $X$ be a complex manifold of dimension $n$ or a real blow up of a Riemann surface, with $f$ a proper endomorphism and ${\mathcal F}$ an $f$-sheaf whose underlying sheaf on $X$ is a locally free ${\mathcal O}_X$-module (so, by definition, $C^{\infty}$ up to the boundary in the blow up case, cf. \thref{Dc:not1}) then if ${\mathcal G}$ is an almost $f$-sheaf such that ${\rm Ext}_q^{X/f} ({\mathcal F},{\mathcal G})$ is finite dimensional for all $q$, there is a $3^{\rm rd}$ quadrant spectral,
\begin{equation}
\label{cd24bis}
{\rm E}_1^{-p,-q} : {\rm Ext}_X^{n-q} ({\mathcal G}_1 , f^* {\mathcal F} \otimes_{\mathcal O} \omega_X) \underset{d_1^{\vee}}{\longrightarrow} {\rm Ext}_X^{n-q} ({\mathcal G}_0 , {\mathcal F} \otimes_{\mathcal O} \omega_X) \, , \quad p= -1 \ \mbox{or} \ -0
\end{equation}
converging to ${\rm Ext}_{p+q}^{X/f} ({\mathcal F},{\mathcal G})^{\vee}$, wherein $\omega_X$ is the dualising sheaf of $X$. Alternatively, irrespective of any finite dimensionality hypothesis, the abutment ${\rm E}_{\infty}^{-p,-q}$ of \eqref{cd24} is always in separated duality with the topological dual of the abutment of \eqref{cd21}.
\end{fact}

\begin{proof}
The dual of the bi-complex of \eqref{cd22} is, cf. \cite[Prop. 6.4]{verdier},
\begin{equation}
\label{cd25}
{\rm Hom}_X^{\rm cts} ({\mathcal G}_1 , f^* {\mathcal F} \otimes {\mathcal D}_X^{n,n-q}) \xrightarrow{ \ d_1^{\vee} \ } {\rm Hom}_X^{\rm cts} ({\mathcal G}_0 , {\mathcal F} \otimes {\mathcal D}_X^{n,n-q})
\end{equation}
where ${\mathcal D}_X^{\bullet , \bullet}$ 
\nomenclature[N]{Sheaf of distributions of type $(p,q)$}{\hyperref[cd25]{${\mathcal D}$}$^{p,q}_{X}$}
 are the sheaves of distributions on $X$, or \eqref{Dc58}, in the real blow up case. Further, as the notation suggests, we must be carefull to take continuous sheaf homomorphisms, which is an empty condition in the only case we'll need, i.e. ${\mathcal G}$ coherent. At the same time the double complex associated to \eqref{cd22} is a complex of LF spaces, so, by hypothesis, we may appeal to \thref{cd:lem2} to deduce that the cohomology of the double complex associated to \eqref{cd25} is the dual of ${\rm Ext}_q^{X/f} ({\mathcal F},{\mathcal G})$. Finally, cf. \cite[Prop. 6.2]{verdier}, modules of $C^{\infty}$ functions, resp. $C^{\infty}$ up to the boundary, on complex manifolds, resp. their real blow ups, are, by the existence of partitions of unity, $c$-soft, so, ${\mathcal D}_X^{n,\bullet}$, resp. \eqref{Dc58}, is an injective resolution of the dualising sheaf $\omega_X$, and whence \eqref{cd24} is the spectral sequence associated to the bi-complex \eqref{cd25} in the $\overline\partial$-direction. Similarly, in general, we can apply \cite[Lemma 2]{ramis} to the calculation of the dual of the bi-complex associated to \eqref{cd22} to deduce the separated duality statement irrespective of the dimension.
\end{proof}

To which we can usefully add,

\begin{rmkdef}\thlabel{rmk:cd2} 
Just as the original $d_1$ of \eqref{cd11} is the difference of 2 maps, so too is $d_1^{\vee}$ of \eqref{cd24}. The easier of the two is the adjoint of $a \mapsto (sa) \varphi$ in \eqref{cd11} since this is just,
\begin{equation}
\label{cd26}
\xymatrix{
{\mathcal G}_1 \ar[rr]^-{\gamma} &&f^* E \otimes \omega \ar[d]^{\varphi \otimes {\rm id}} \\
{\mathcal G}_0 \ar[u]^s \ar[rr]_-{d_1^{\vee}(\gamma)} &&E \otimes \omega
}
\end{equation}
The more subtle one is that the dual to $f^*$ on differential forms is the trace $f_*$ on distributions, i.e. the adjoint of $a \mapsto t(f^* a)$ in \eqref{cd11} is, 
\nomenclature[D]{Trace \hyperref[rmk:cd2]{of a distribution}}{}
\begin{equation}
\label{cd27}
\xymatrix{
t(f^* g) &{\mathcal G}_1 \ar[rr]_-{\gamma} &&f^* E \otimes \omega \ar[d]^{f_* = \, {\rm Trace}} \\
g \ar[u] &{\mathcal G}_0 \ar[u] \ar[rr]_-{d_1^{\vee}(\gamma)} &&E \otimes \omega
}
\end{equation}
\end{rmkdef}

We can now introduce some hypothesis which will cover the applications that we have in mind, to wit:

\begin{setup}\thlabel{cd:setup1}
Let $X$ be a complex manifold of dimension $n$ or a real blow up of a Riemann surface, with $f$ a proper endomorphism, ${\mathcal F}$ an $f$-sheaf whose underlying sheaf is a locally free ${\mathcal O}_X$ module, and suppose that we have a short exact sequence of almost $f$-sheaves of ${\mathcal O}_X$-modules,
\begin{equation}
\label{cd28}
0 \longrightarrow {\mathcal G}' \longrightarrow {\mathcal G} \longrightarrow {\mathcal G}'' \longrightarrow 0
\end{equation}
such that each of the complexes $R^q \Theta$ of vector spaces of any sheaf in \eqref{cd28}, for $\Theta$ the functor from almost $f$-sheaves to complexes of \eqref{cd12}, is a separated LF space (in some topology depending on $q$ and the sheaf) and all the maps in the long exact sequence of complexes,
\begin{equation}
\label{cd29}
0 \longrightarrow R^0 \Theta \, {\mathcal G}' \longrightarrow R^0 \Theta \, {\mathcal G} \longrightarrow R^0 \Theta \, {\mathcal G}'' \longrightarrow R^1 \Theta \, {\mathcal G}' \longrightarrow \mbox{etc.}
\end{equation}
are continuous. As such the various ${\rm E}_2^{p,q} = {\rm E}_{\infty}^{p,q}$ terms in the spectral sequence of \thref{cd:fact3} naturally inherit a topology from the complexes $R^q \Theta$ with respect to which the ${\rm E}_{\infty}^{0,q}$ are separated but a priori there is a non-trivial map,
\begin{equation}
\label{cd30}
{\rm E}_{\infty}^{1,q} \longrightarrow \overline {\rm E}_{\infty}^{1,q} \longrightarrow 0
\end{equation}
\nomenclature[N]{Maximal separated quotient}{of a T.V.S. $V$, \hyperref[cd30]{$\overline V$}} 
onto the maximal separated quotient.
\end{setup}

Of itself this set up is, perhaps, a little too general, but there is a particular variant that suits our purpose, to wit:

\begin{fact}\thlabel{cd:fact7}
Let everything be as in \thref{cd:setup1} and suppose moreover that,
\begin{enumerate}
\item[(a)] ${\rm Ext}_{X/f}^q ({\mathcal F},{\mathcal G})$, $\forall \, q$ and $R^0 \Theta \, {\mathcal G}$ are finite dimensional.
\item[(b)] $R^q \Theta \, {\mathcal G}'$, resp. $R^q \Theta \, {\mathcal G}''$, is concentrated in degree $q=1$, resp. $q=0$.
\end{enumerate}
Then on identifying ${\rm Ext}_{X/f}^{p+q} ({\mathcal F} , {\mathcal G}')$, resp. ${\rm Ext}_{X/f}^{p+q} ({\mathcal F},{\mathcal G}'')$ with the ${\rm E}^{p,q}_{\infty} = {\rm E}_2^{p,q}$ terms of \thref{cd:fact3} and their maximal separated quotients $\overline{\rm Ext}_{X/f}^{p+q} ({\mathcal F} , {\mathcal G}')$, resp. $\overline{\rm Ext}_{X/f}^{p+q} ({\mathcal F} , {\mathcal G}'')$, via \eqref{cd30} we have an exact sequence,
\begin{equation}
\label{cd31}
{\rm Ext}^1_{X/f} ({\mathcal F} , {\mathcal G}) \longrightarrow \overline{\rm Ext}^1_{X/f} ({\mathcal F} , {\mathcal G}'') \longrightarrow \overline{\rm Ext}^2_{X/f} ({\mathcal F} , {\mathcal G}') \longrightarrow {\rm Ext}^2_{X/f} ({\mathcal F} , {\mathcal G}) \longrightarrow 0 \, .
\end{equation}
\end{fact}

\begin{proof}
Just as in \eqref{cd24} it follows from \cite[prop. 5]{DS} that \eqref{cd29} is topologically exact and the transition maps $d_1^{p,q}$ for the sheaf ${\mathcal G}$ are closed. Now let $A$, resp. $B$, be the image of $R^0 \Theta \, {\mathcal G}''$ in $R^1 \Theta \, {\mathcal G}'$ in degree $0$, resp. $1$, then we can split \eqref{cd29} into short exact sequences in the topological category, to wit:
\begin{eqnarray}
\label{cd32}
\xymatrix{
0 \ar[r] &{\rm Hom}_X ({\mathcal F},{\mathcal G}_0) \ar[d]_{d_1} \ar[r] &{\rm Hom}_X ({\mathcal F},{\mathcal G}''_0) \ar[d]_{d_1} \ar[r] &A \ar[d]_{d_1} \ar[r] &0 \\
0 \ar[r] &{\rm Hom}_X (f^* {\mathcal F},{\mathcal G}_1) \ar[r] &{\rm Hom}_X (f^* {\mathcal F} , {\mathcal G}''_1) \ar[r] &B \ar[r] &0
} \nonumber \\
\xymatrix{
0 \ar[r] &A \ar[d]^{d_1} \ar[r] &{\rm Ext}_X^1 ({\mathcal F},{\mathcal G}'_0) \ar[d]_{d_1} \ar[r] &{\rm Ext}_X^1 ({\mathcal F},{\mathcal G}_0) \ar[d]_{d_1} \ar[d]_{d_1} \ar[r] &0 \\
0 \ar[r] &B \ar[r] &{\rm Ext}_X^1 (f^* {\mathcal F} , {\mathcal G}'_1) \ar[r] &{\rm Ext}_X^1 (f^* {\mathcal F} , {\mathcal G}_1) \ar[r] &0
}
\end{eqnarray}

Now certainly the long exact sequence for ${\rm Ext}_{X/f}$ applied to \eqref{cd28} follows from the snake lemma applied to the two diagrams in \eqref{cd32}, but there are equally two relevant topological variants of this, i.e.

\begin{claim}\thlabel{cd:claim1}
Suppose that we have a topologically exact diagram of separated locally convex vector spaces,
\begin{equation}
\label{cd33}
\xymatrix{
0 \ar[r] &E'\ar[d]_{d'} \ar[r] &E \ar[d]_d \ar[r] &E'' \ar[d]_{d''} \ar[r] &0 \\
0 \ar[r] &F' \ar[r] &F \ar[r] &F'' \ar[r] &0
}
\end{equation}
with kernels $K^{\bullet}$, cokernels $C^{\bullet}$, and maximally separated quotients $\overline C^{\bullet}$ then if $F'$ is finite, resp. $d''$ is closed with $K''$ finite,
\begin{equation}
\label{cd34}
C' \to \overline C \to \overline C'' \to 0 \, , \quad \mbox{resp.} \ K'' \to \overline C' \to \overline C \to C'' \to 0
\end{equation}
is exact.
\end{claim}

\begin{proof}
The content of the first assertion is that $\overline{d(E)} + F'$ is closed, which follows from Hahn-Banach. Similarly if $E_0$ is the fibre of $E$ over $K''$ with $V$ a finite dimensional space lifting $K''$ then we appeal to $\overline{d(E_0)} = \overline{d(E')} + V$ in the second, but also to the fact that the topology on the ${\rm Im} (d'')$ as a subspace of $F''$ is equally that of a quotient space of $E''$ by \cite[Prop. 5]{DS}.
\end{proof}

Now by hypothesis, \thref{cd:claim1} applies to the top, resp. bottom diagram in \eqref{cd32}, which can then be spliced together to obtain \eqref{cd31}.
\end{proof}

%\end{document}

%%%%%%%%%%%%%%%

\vglue 1cm

\section{Dynamical Residue}\label{SRes}

In \cite{adam} Epstein has defined a dynamical residue for Baire measures on a locally compact Hausdorff space invariant by a germ of an automorphism fixing $\infty$ whose definition we recall for convenience, to wit:

\begin{defrev}\thlabel{DRev:1}\cite[Lemma 2]{adam}
\nomenclature[D]{Dynamical \hyperref[DRev:1]{residue}}{}
Let $\overline X = X \cup \{\infty\}$ be the one point compactification of a non-compact locally compact Hausdorff space $X$, $f : U = \overline X \backslash K \to \overline X$, $K \hookrightarrow X$ compact a continuous map fixing infinity, and $\mu$ a signed Radon measure on $X$ such that $f^* \mu-\mu$ is absolutely integrable, then for $U \hookleftarrow V \ni \infty$ an open neighbourhood of $\infty$ we define, 
\nomenclature[N]{Dynamic residue in a neighbourhood $V$}{of $\infty$ \hyperref[Res1]{${\rm Res}$}$^V_f (\mu)$}
\begin{equation}\label{Res1}
{\rm Res}_f^V (\mu) = \int_{f(V)\backslash V} \mu - \int_{V \backslash f(U)} \mu
\end{equation}
and observe that the net of real numbers $\{{\rm Res}_f^V (\mu)\}_{V \ni \infty}$ has a unique limit ${\rm Res}_f(\mu)$, 
\nomenclature[N]{Dynamic residue at $\infty$}{\hyperref[DRev:1]{${\rm Res}$}$_f(\mu)$} 
the dynamic residue, which is actually equal to any ${\rm Res}_f^V (\mu)$ if $\mu$ is invariant.
\end{defrev}

\begin{proof}
Observe that if $V \subset W \cap f(W)$, $V,W$ open neighbourhoods of $\infty$,
\begin{equation}\label{Res2}
{\rm Res}_f^W (\mu) - {\rm Res}_f^V (\mu) = \int_{f(W\backslash V)} \mu - \int_{W \backslash V} \mu = \int_{W \backslash V} (f^* \mu -\mu)
\end{equation}
where, by hypothesis, the absolute value of the rightmost term in \eqref{Res2} is at most the value of a finite radon measure on $W$, while, by definition, any such measure is inner regular, so, given $\varepsilon > 0$, there exists $W_{\varepsilon} \ni \infty$ such that this never exceeds $\varepsilon$.
\end{proof}

The generality in which the dynamic residue is defined is much greater than that in which we will employ it, to wit:

\begin{rmkex}\thlabel{DRmk:1}
The only case of interest is $\overline X$ a closed disc, $X$ the same punctured in the origin (identified with $\infty$ therefore in the sense of \thref{DRev:1}), and $f : X \to X$ a germ of a holomorphic automorphism. In particular, therefore, there is a multiplier, 
\nomenclature[N]{Multiplier at a fixed point}{\hyperref[Res5]{$\lambda$}} 
\begin{equation}\label{Res3}
\lambda = df(0) \in {\rm End} (T_{X,0}) \, .
\end{equation}
Thus, for example, Stokes' theorem gives,
\begin{equation}\label{Res4}
{\rm Res}_f \left( \frac{d \overline z \, dz}{\vert z \vert^2} \right) = \lim_{\varepsilon \to 0} \, \int_{\vert z \vert = \varepsilon} f^* \left( \log \vert z \vert^2 \, \frac{dz}z \right) - (\log \vert z \vert^2) \, \frac{dz}z = \log \vert \lambda \vert^2 \, .
\end{equation}

More relevantly, however, is the general attracting case, i.e. $\vert \lambda \vert < 1$, wherein the conclusion that \eqref{Res4} is negative trivially generalises. Indeed in this case the map $f$ can be conjugated to,
\begin{equation}\label{Res5}
f : X \longrightarrow X : z \longmapsto \lambda z
\end{equation}
so that if we take $V$ to be the disc of radius $\varepsilon$ in \eqref{Res1}, $f(V) \subset V$, so, indeed, ${\rm Res}_f (d\mu) \leq 0$ for any measure $d\mu$ with $f^* \mu-\mu$ absolutely integrable.
\end{rmkex}

Our goal is to extend the conclusion of \eqref{Res5} et seq. to parabolic points, i.e. $\lambda$ a root of unity in \eqref{Res3}. The essential difference with \cite{adam} is that our measure isn't quite meromorphic, and, so, the subsequent calculation of the residue is more demanding. To this end, therefore, let us recall,

\begin{notation}\thlabel{DRev:2}
\nomenclature[D]{Petal, \hyperref[DRev:2]{attracting or repelling}}{} 
Let everything be as in \thref{DRmk:1} but with $f$ parabolic, i.e. $\lambda$ a root of unity, and observe, \cite[\S 3]{adam}, that for any $r \in {\mathbb Z}_{\geq 0}$,
\begin{equation}\label{Res6}
{\rm Res}_{f^r} (\mu) = r \, {\rm Res}_f (\mu)
\end{equation}
so that we may, without loss of generality, suppose that $\lambda = 1$. As such, we will systematically employ the Leray-Fatou flower theorem, e.g. \cite[10.7-10.9]{milnor}, in order to have, at our disposition, the best possible analogue of \eqref{Res5}, i.e. if $f$ is tangent to the identity to order $e+1$, $e \in {\mathbb Z}_{\geq 1}$, then there is a division of the punctured disc $\Delta^{\times} = X$ into simply connected open sets, the {\it petals} of $f$,
\nomenclature[N]{Petal, attracting}{\hyperref[Res7]{$P$}$^+_{\bullet}$} 
\nomenclature[N]{Petal, repelling}{\hyperref[Res7]{$P$}$^-_{\bullet}$} 
\begin{equation}\label{Res7}
P_1^+ , P_1^- , P_2^+ , P_2^- , \cdots , P_e^+ , P_e^-
\end{equation}
wherein the $P_{\bullet}^+$, resp. $P_{\bullet}^-$ are {\it attracting}, resp. {\it repelling petals}, i.e. $f (P_{\bullet}^+) \subset P_{\bullet}^+$, resp. $f^{-1} (P_{\bullet}^-) \subset P_{\bullet}^-$, and in either case $f$ is conjugate to,
\nomenclature[N]{Fatou coordinate}{\hyperref[Res8]{$s$}} 
\begin{equation}\label{Res8}
s \longmapsto s+1
\end{equation}
while (understanding the indices in ${\mathbb Z}/e$) $P_i^+$, resp. $P_i^-$, meets only $P^-_{i-1}$ and $P_i^-$, resp. $P_i^+$ and $P^+_{i+1}$, with the said intersection being simply connected non-empty.
\end{notation}

Furthermore such a map admits a unique invariant formal differential,
\nomenclature[N]{Invariant differential at a parabolic fixed}{point \hyperref[Res9]{$\omega$}} 
\begin{equation}\label{Res9}
\omega = \frac{(1+\nu \, x^e)}{x^{e+1}} \, dx \, , \quad x \in {\mathbb C} \left[[z]\right] \, , \quad x'(z) \ne 0 \, , \ x(z) = 0
\end{equation}
where $\nu$ is \'Ecalle's residu iteratif. 
\nomenclature[N]{Residu iteratif}{\hyperref[Res9]{$\nu$}} 
In particular if \eqref{Res9} were convergent then the coordinate $s$ of \eqref{Res8} is,
\begin{equation}\label{Res10}
s = t \, \log \, t^{-\nu/e} \, , \quad t = -1/e \, x^{-e} \, .
\end{equation}

Similarly if $W \in {\mathbb C} \{z\} \, dz$ is any convergent $1$-form satisfying,
\begin{equation}\label{Res12}
W-\omega \in {\mathbb C} \left[[z]\right] \, dz
\end{equation}
then $\overline W W$ satisfies the hypothesis of \thref{DRev:1}, and, \cite[Lemma 2]{adam},
\begin{equation}\label{Res13}
{\rm Res}_f (\overline W W) = e^{-2} \, {\rm Re}(\nu) \, .
\end{equation}

As this revision suggests the difficulty with this case is the mixture of attracting and repelling behaviour which we begin to address by,

\begin{lem}\thlabel{DLem:1}
Let $V , V_{\alpha} , V_{\beta} \subseteq {\mathbb C}$ be strips, $R < {\rm Re} (z) < R+1$, resp. ${\rm Re}(z) \in \alpha$, ${\rm Re} (z) \subset \beta$ where $\alpha \ni R$, resp. $\beta \ni R+1$, are open connected real intervals of length at most $1$ whose union contains $V$, then for $h_{\alpha\beta}$ a holomorphic function on $(V_{\alpha} \cap V_{\beta}) \coprod (V_{\alpha}) \cap (V_{\beta}^{-1})$ of size,
\begin{equation}\label{Res14}
\vert h_{\alpha\beta} (s) \vert \ll \vert s \vert^{-a} \, , \quad a > 2
\end{equation}
there are functions $h_{\alpha}$, resp. $h_{\beta}$, on $V_{\alpha}$, resp. $h_{\beta}$, such that for $R$ (in function of the implied constant in \eqref{Res14}) sufficiently large,
\begin{equation}\label{Res15}
h_{\alpha\beta} (s) =\left\{\begin{matrix}
h_{\alpha} - h_{\beta} \, , \quad s \in V_{\alpha} \cap V_{\beta} \hfill \\
h_{\alpha} (s) - h_{\beta} (s+1) \, , \, s \in V_{\alpha} \cap (V_{\beta}^{-1}) \hfill 
\end{matrix} \right. \quad \mbox{and} \ \vert h_{\alpha} \vert + \vert h_{\beta} \vert \ll \log^2 \vert R \vert \ \vert R \vert^{1-q} \, .
\end{equation}
\end{lem}

\begin{proof} 
We go through the standard steps for trivialising co-cycles in the plane. Start, therefore, with the interval  $[R,R+1]$ and identify end points to get a circle, $S^1$, along with opens $\alpha',\beta'$ defined by the images of $\alpha , \beta$ affording a covering $\alpha' \coprod \beta' \twoheadrightarrow S^1$ together with a partition of unity,
\begin{equation}\label{Res16}
1 = \rho_{\alpha} + \rho_{\beta} \, , \quad \rho_{\bullet} (s+1) = \rho_{\bullet}(s) \, , \quad \bullet = \alpha \ \mbox{or} \ \beta \, .
\end{equation}
As such we get $C^{\infty}$ functions in $V_{\alpha}$, resp. $V_{\beta}$, defined by,
\begin{equation}\label{Res17}
H_{\alpha} = \rho_{\beta} \, h_{\alpha\beta} \, , \quad H_{\beta}(s) = \left\{ \begin{matrix}
- \rho_{\alpha} \, h_{\alpha\beta} , \quad s \in V_{\alpha} \cap V_{\beta} \hfill \\
-\rho_{\alpha} \, h_{\alpha\beta}(s-1) \, , \quad s \in (V_{\alpha} + 1) \cap V_{\beta}
\end{matrix} \right.
\end{equation}
together with a translation invariant (0,1) form on $U:= V_{\alpha} \cup V_{\beta}$,
\begin{equation}\label{Res18}
\omega := \overline\partial H_{\alpha} = \overline\partial H_{\beta} \, , \quad \left\vert \frac{\omega}{d\overline s} \right\vert \ll \vert s \vert^{-a} \, .
\end{equation}
Now solve the $\overline\partial$-equation for $\omega$ in the usual way, i.e.
\begin{equation}\label{Res19}
H(\zeta) = \int_{z \in {\mathbb C}^{\times}} \frac{\omega \, dz}{z-\zeta} \, , \quad \zeta = \exp(2\pi it) \in {\mathbb C}^{\times}
\end{equation}
which for $\zeta = \exp (2\pi it)$, $R \leq {\rm Re} (t) \leq R+1$ leads us to divide ${\mathbb C}^{\times}$ into the image, under the exponential, of a disc, $D_{\varepsilon} (t)$, of radius $\varepsilon$ about $t$, and its complement, $U'$, say. Now on $D_{\varepsilon} (t)$, for $\varepsilon$ sufficiently small,
\begin{equation}\label{Res19.bis}
z - \zeta = 2\pi i \zeta s (1+0(s)) \, , \quad z = \zeta \exp (2\pi is) \, , \quad \vert s \vert < \varepsilon
\end{equation}
so the contribution to the integral over the image of $D_{\varepsilon} (t)$ is at worst $O(\vert t \vert^{-a})$ by \eqref{Res18}. As to the rest it follows from \eqref{Res19.bis} that in the sub-region $2 \vert \zeta \vert > \vert z \vert$ the integrand is at most of order,
\begin{equation}\label{Res19.bis.bis}
\frac1{2\varepsilon} \left\vert\frac\omega{d \overline s} \right\vert \cdot \frac{d\overline z \, dz}{\vert z \vert^2}
\end{equation}
while in the complementary region, $\vert z - \zeta \vert < 3/2 \vert \zeta \vert$, so, altogether the integrand is of the order \eqref{Res19.bis.bis}. Now to estimate such an integrand we can return to the $s$ variable, and first look at the region $\vert {\rm Im} \, (s) \vert < R$ where by \eqref{Res18} it's of order,
\begin{equation}\label{Res20}
R^{-a} \int_{U \cap \vert {\rm Im} \, (s) \vert < R} d\overline s \, ds \ll R^{1-a}
\end{equation}
while everything else is at worst,
\begin{equation}\label{Res21}
R^{1-a} \log^2 R \int_{U \cap \vert {\rm Im} \, (s) \vert > R} \frac{d\overline s \, ds}{\log^2 \vert {\rm Im} \, (s) \vert \, \vert {\rm Im} \, (s) \vert} \ll R^{1-a} \log^2 R \, .
\end{equation}

Consequently we have an $f$-invariant solution,
\begin{equation}\label{Res22}
\overline\partial H = \omega
\end{equation}
with $\vert H \vert \ll (\log^2 R) \cdot R^{1-a}$ for an implied constant (supposing $R$ sufficiently large) no worse than that of \eqref{Res14} and intervals $\alpha , \beta$ of the length asserted in \thref{Dc:lem1}, so $h_{\alpha} = H_{\alpha} - H$, $h_{\beta} = H_{\beta} - H$ satisfy \eqref{Res15}.
\end{proof}

Which we can apply as follows,

\begin{cor}\thlabel{Dcor:1}
Let $V$ be an open neighbourhood of the strip $R \leq {\rm Re} (s) \leq R+1$, and $\omega = h(s) \, ds^{\otimes m}$ a meromorphic $m$-form on $V$ such that for $f(s) = s+1$,
\begin{equation}\label{Res23}
\vert h(s+1) - h(s) \vert \ll \vert s \vert^{-a}
\end{equation}
then on identifying the line ${\rm Re} (s) = R$ with ${\rm Re} (s) = R+1$ and the resulting space with ${\mathbb C}^{\times} \underset{\sigma}{\longhookrightarrow} {\mathbb C}$ there is a meromorphic 1-form $W(\sigma) \left(\frac{d\sigma}{\sigma}\right)^{\otimes m}$ on ${\mathbb C}^{\times}$ such that,
\begin{equation}\label{Res24}
\vert W \!\!\mid_V - \, h \vert \ll (\log^2 \vert R \vert) \cdot \vert R \vert^{1-a} \, .
\end{equation}
\end{cor}

\begin{proof}
Take $V_{\alpha} , V_{\beta}$ as in \thref{DLem:1}, and define,
\begin{equation}\label{Res25}
h_{\alpha\beta}(s) = \left\{\begin{matrix}
h(s+1) - h(s) \, , &s \in V_{\alpha} \cap (V_{\beta} - 1) \\
0 \, , &s \in V_{\alpha} \cap V_{\beta} \, . \hfill
\end{matrix}\right.
\end{equation}
Now apply the lemma to find $h_{\alpha} , h_{\beta}$ satisfying \eqref{Res15} then not only do we have,
\begin{equation}\label{Res26}
h \!\mid_{V_{\alpha}} + \, h_{\alpha} \!\mid_{V_{\alpha\beta}} = h \!\mid_{V_{\beta}} + \, h_{\beta} \!\mid_{V_{\alpha\beta}}
\end{equation}
but the resulting function, $W$, say, on a neighbourhood of $R \leq {\rm Re} (z) \leq R+1$ satisfies $W(s+1) = W(s)$, while $\sigma^{-1} d\sigma$ is a multiple of $ds$, so $W(\sigma) \, \frac{d\sigma}{\sigma}$ does the job.
\end{proof}

Which, in turn, we can build on by way of,

\begin{cor}\thlabel{Dcor:2}
Suppose in addition to the hypothesis of \thref{Dcor:1} that $h$ is holomorphic and there is constant $k$ such that,
\begin{equation}\label{Res27}
\vert h(s) \vert \ll \vert {\rm Im} (s) \vert^k \, , \quad s \in V
\end{equation}
then $W(\sigma)$ in \eqref{Res24} is a constant $\lambda_R$.
\end{cor}

\begin{proof}
In the notation of \thref{Dcor:1},
\begin{equation}\label{Res28}
\vert \log (\sigma) \vert = 2 \pi \, \vert{\rm Im} \, (s) \vert
\end{equation}
so the growth of $\sigma \mapsto W(\sigma)$ whether at $0$ or $\infty$ identified to ${\mathbb P}^1 \backslash {\mathbb C}^{\times}$ is sub-meromorphic by \eqref{Res24} and \eqref{Res27}.
\end{proof}

We thus arrive to our final variant,  

\begin{cor}\thlabel{Dcor:3} 
Suppose in addition to the hypothesis of \thref{Dcor:1} and \thref{Dcor:2} $h(s) \, ds^{\otimes m}$ is holomorphic in a half plane ${\rm Re} (s) > X$, resp. ${\rm Re} (s) < - X$, and satisfies \eqref{Res27} in the same then the constants $\lambda_R$ have a well defined limits, $\lambda_{\infty}$, as $R \to +\infty$, resp. $-\infty$ such that,
\begin{equation}\label{Res29}
\vert \lambda_R - \lambda_{\infty} \vert \ll \vert R \vert^{1-a} \, .
\end{equation}
In particular if $\lambda_{\infty} = 0$, and $a > \max \{2,m/2 + 1\}$,
\begin{equation}\label{Res30}
\lim_{R \to \infty} \underset{R \leq {\rm Re} (s) \leq R+1 \atop \vert {\rm Im} (s) \vert \leq R}{\int} \vert h \vert^{2/m} d\overline s \, ds = 0 \, .
\end{equation}
\end{cor}

\begin{proof}
>From \eqref{Res23}, $\vert \lambda_{R+1} - \lambda_R \vert \ll R^{-a}$, whence \eqref{Res29}. Thus by \eqref{Res24},
\begin{equation}\label{Res31}
\vert h \vert \ll (\log^2 R) R^{1-a}
\end{equation}
as soon as $\lambda_{\infty} = 0$, and we conclude to \eqref{Res30} provided $a > m/2+1$ in addition to the hypothesis $a > 2$ of \eqref{Res14}. 
\end{proof}

\bigskip

To apply these considerations, let us put them in their proper,

\begin{context}\thlabel{Dcon:1}
Let everything be as in \thref{DRev:2}, with $\omega$ a meromorphic $m$-form on $X = \Delta^{\times}$ such that,
\begin{equation}\label{Res32}
f^* \omega - \omega \in {\mathfrak m} \, \Omega (\log 0)^{\otimes m}
\end{equation}
where ${\mathfrak m}$ is the maximal ideal at the origin. As such, for $m > 2$, the measure,
\begin{equation}\label{Res33}
\omega = W dz^{\otimes m} \longmapsto \vert \omega \vert := \vert W \vert^{2/m} d\overline z \, dz
\end{equation}
satisfies the conditions of \thref{DRev:1} in the strong quantitative form,
\begin{equation}\label{Res34}
f^* \vert \omega \vert - \vert \omega \vert = O \left(\frac{d \overline z \, dz}{\vert z \vert^{2-1/m}}\right).
\end{equation}
\end{context}

Equivalently, if on a petal, we express $\omega$ in the Fatou coordinate, $s$, of \eqref{Res8} then by \eqref{Res9}-\eqref{Res10},
\begin{equation}\label{Res35}
f^* \omega - \omega = O \left(\frac{1}{\vert s \vert^{m+1/e}} \right) ds^{\otimes m} \, .
\end{equation}
Equally by \eqref{Res9}--\eqref{Res10} the attracting, resp. repelling, petals are approximately, in $\Delta$, sectors of width $2\pi/e$ centred on the directions $z^e \in {\mathbb R}_{>0}$, resp. $z^e \in {\mathbb R}_{<0}$, and we further suppose,

\begin{hypo}\thlabel{Dhyp:1}
In the above notations there are $2e$ sectors $S_i$ of width $\varepsilon > 0$ centred on the $2e$ directions $z^e \in {\mathbb R} (1)$ and $k > 0$ such that $W$ of \eqref{Res33} satisfies the meromorphic growth estimate,
\begin{equation}\label{Res36}
\vert W(z)\vert \ll \vert {\rm Im} (z^{-e})\vert^k \, .
\end{equation}
\end{hypo}

Now we can put everything together by way of,

\begin{claim}\thlabel{Dclaim:1}
Let everything be as in \thref{Dcon:1} and \thref{Dhyp:1}, and suppose further that $\omega$ is holomorphic in an attracting, resp. repelling, petal $P$ with Fatou coordinate $s$, then for any $2 + 2/em > b > 2$, there is a constant $\lambda_P$ such that,
\begin{equation}\label{Res37}
\left\vert \frac{\vert \omega \vert_p \vert}{d\overline s \, ds} - \lambda_P \right\vert \ll \vert {\rm Re} (s) \vert^{1-b}
\end{equation}
whenever $R(s) \to \infty$, resp. ${\rm Re} (s) \to -\infty$.
\end{claim}

\begin{proof}
The conclusion of \thref{Dcor:1} is, by \eqref{Res35}, valid on the petal with $a = m+1/2 > m \geq 2$. Similarly any strip, ${\rm Re} (s) \in [R,R+1]$, eventually (i.e. $\vert {\rm Im} (s) \vert \sim R/\varepsilon$) meets a sector where the meromorphic growth estimate of \eqref{Res36} is valid, so that \eqref{Res27} holds on the given strip, and whence \thref{Dclaim:1} by \thref{Dcor:2} and \thref{Dcor:3}.
\end{proof}

Slightly more generally we also have,

\begin{claim}\thlabel{Dclaim:1.bis}
Let everything be as above in \thref{Dclaim:1} but suppose that rather than being holomorphic $\omega$ is meromorphic in an attracting petal with poles on an invariant divisor, $D$, with finitely many orbits then for any strip, ${\rm Re}(s) \in [R,R+1]$, there are constants $\lambda_R^{\infty}$, resp. $\lambda_R^0$ such that $W$ of \eqref{Dcor:1} satisfies
\begin{equation}\label{Res37.bis}
W \longrightarrow \lambda_R^{\infty} \, , \quad \mbox{resp.} \ \longrightarrow \lambda_R^0 \, , \ \quad R \leq {\rm Re} (s) \leq R+1
\end{equation}
as ${\rm Im} (s) \to \infty$, resp. ${\rm Im} (s) \to - \infty$. Moreover provided ${\rm Re} (s) \gg 0$, the functions $\vert W_R \vert^{2/m}$ of $x = {\rm Re} (s)$, $y = {\rm Im} (s)$ also satisfy,
\begin{equation}\label{Res37.bis.bis}
\left\vert \vert W_R \vert^{2/m} (x,y) - \vert W_{R+n} \vert^{2/m} (\xi,y) \right\vert \ll R^{1-b} \, , 
\end{equation}
$$
R \leq x \leq R+1 \, , \quad y \in {\mathbb R} (1) \, , \quad R+n \leq \xi \leq R + n + 1 \, , \quad n \geq 0 \, .
$$
In particular, there are well defined limits $\vert \lambda_P^{\infty} \vert$, resp. $\vert \lambda_P^0 \vert$, of $\vert \lambda_R^{\infty} \vert^{2/m}$, resp. $\vert \lambda_R^0 \vert^{2/m}$, as $R \to \infty$.
\end{claim}

\begin{proof}
Exactly as above by \thref{Dhyp:1} the meromorphic growth condition of \eqref{Res27} holds, so, the proof of \thref{Dcor:2} implies, for the same reason, that $\omega(\sigma)$ viewed as a meromorphic function of $\sigma \in {\mathbb C}^{\times} \hookrightarrow {\mathbb P}^1$ with poles along $D$ has submeromorphic growth at $0$ and $\infty$, and whence \eqref{Res37.bis}. Similarly, the estimate \eqref{Res37.bis.bis} follows for the same reason as that of \eqref{Res37}.
\end{proof}

Continuing in this vein we further assert,

\begin{claim}\thlabel{Dclaim:2}
Let everything be as in \thref{Dclaim:1.bis} (so $\omega$ of \thref{Dcon:1} might not be holomorphic but it has at worst a meromorphic pole along an invariant divisor in at most $1$ petal) then all of the constants $\lambda_P , \lambda_P^0 , \lambda_P^{\infty}$ of \thref{Dclaim:1} and \thref{Dclaim:1.bis} are equal irrespective of the petal $P$ to a constant $\lambda (\omega)$.
\end{claim}

\begin{proof}
Consider first the case that $\omega$ is holomorphic in consecutive petals with Fatou coordinates $s'$, and $s''$ respectively. As such by \eqref{Res9} for any fixed $N > 0$ on the intersection of these petals we can ensure,
\begin{equation}\label{Res38}
ds' = ds'' (1+0(\vert s' \vert^{-N})) = ds'' (1+0(\vert s'' \vert^{-N})) \, .
\end{equation}
At the same time the intersection of these petals contains one of the half planes $\vert {\rm Im} (s') \vert > Y'$, resp. $\vert {\rm Im} (s'') \vert > Y''$, for $Y \gg 0$, so to compare the constants $\lambda'_{m'}$, resp. $\lambda''_{n''}$ of \thref{Dcor:2} in the strips ${\rm Re} (s') \in [m',m'+1]$, resp. ${\rm Re} (s'') \in [n'',n''+1]$, $m' , n'' \in {\mathbb Z}$ respecting the asymptotics post \eqref{Res37}, we may, by \eqref{Res24} and \eqref{Res38}, employ,
\begin{equation}\label{Res39}
\omega = \lambda'_{m'} (ds')^{\otimes m} + 0 \left( \frac{\log^2 \vert n' \vert}{\Vert n' \Vert^{m+1/2}} \right)
\end{equation}
and, similarly, for the other petal in the region $\vert {\rm Im} (s'') \vert > Y$, where, in either case, without loss of generality,
\begin{equation}\label{Res40}
\omega = \left(f^{\vert m'-n'' \vert}\right)^* \omega + 0 \left( \frac{1}{(\vert m'\vert + \vert n''\vert)^{(m-1) + 1/e}} \right)
\end{equation}
by \eqref{Res35}, so, from \eqref{Res38}--\eqref{Res40}, $\lambda' = \lambda''$.

As to the case where $\omega$ has, without loss of generality, a pole in the petal described by the $s''$ variable, we have the function $W$ of \thref{Dcor:1}, which, following \thref{Dclaim:1.bis} we regard as a function of $x = {\rm Re} (s'')$, and $y = {\rm Im} (s'')$. Thus by \eqref{Res24}, \eqref{Res38}--\eqref{Res40} for any $m+1/e > a$ we have the estimate, for either $y \gg 0$, or $-y \ll 0$,
\begin{equation}\label{Res40.bis}
\vert \lambda'_{m'} - W (x,y) \vert \ll 0 \left( \frac{1}{(\vert m'\vert + \vert n''\vert)^{(m-1) + 1/e}} \right) , \quad n'' \leq x \leq n''+1
\end{equation}
so it suffices, according to the context, to choose $y$, depending on $n''$, either very large or very negative, to conclude that $\lambda'$ is equal to $\lambda^{\infty}$ or $\lambda^0$ as appropriate.
\end{proof}

\begin{lem}\thlabel{Dclaim:3} 
Let everything be as in \thref{Dclaim:2} then, in the vector space of $m$-forms satisfying \eqref{Res32} and \eqref{Res36} the map $\omega \mapsto \lambda (\omega)$ is linear and, if $\omega$ is holomorphic,
\begin{equation}\label{Res41}
{\rm Res}_f (\vert \omega \vert) = \vert \lambda_{\omega} \vert^{2/m} \, e^{-2} \, {\rm Re} (\nu)
\end{equation}
where $\nu$ is the r\'esidu iteratif of \eqref{Res9}.
\end{lem}

\begin{proof}
All the estimates respect finite linear combinations, so, asymptotically on strips in any petal with Fatou coordinate $s$,
\begin{equation}\label{Res42}
c_1 \, \lambda (\omega_1) \, ds^{\otimes m} + c_2 \, \lambda (\omega_2) \, ds^{\otimes m} \sim c_1 \, \omega_1 + c_2 \, \omega_2 \sim \lambda (c_1 \, \omega_1 + c_2 \, \omega_2) \, ds^{\otimes m}
\end{equation}
for any $m$-forms $\omega_i$ and constants $c_i$, from which the linearity of $\omega \mapsto \lambda_{\omega}$.

As to the computation of the residue a convenient basis of neighbourhoods in which we can compute \eqref{Res1} is provided by \cite[Thm 1]{Camacho} according to which $f$ is $C^0$ conjugate to,
\begin{equation}\label{Res43}
F_e (\zeta) = \frac{e \, \zeta}{(1-e \, \zeta^e)^{1/e}} \, , \quad \zeta = \zeta (z)
\end{equation}
or, in other words, the diagram with $C^0$ horizontals, $c(z) = e \, \zeta (z)$,
\begin{equation}\label{Res44}
\xymatrix{
\Delta \ar[d]_f \ar[r]^c &\Delta \ar[d]^{F_1} \\
\Delta \ar[r]_c &\Delta
}
\end{equation}
commutes. The Fatou coordinates $s_1$ of $F_1$, however, is exactly $-1/z$ on both attracting and repelling petals, so if we take,
\begin{equation}\label{Res45}
V_1(R) := \left\{ s_1 \mid \min \{ \vert {\rm Re} (s_1) \vert , \, \vert {\rm Im} (s_1) \vert \} > R \right\}
\end{equation}
then $F_1 (V_1 (R)) \backslash V_1(R)$, resp. $V_1 (R) \mid F_1 (V_1(R))$ is, for $R \gg 0$,
\begin{equation}\label{Res46}
\vert {\rm Im} (s_1) \vert \leq R \quad \mbox{and} \quad {\rm Re} (s_1) \in [-R , 1-R] \, , \quad \mbox{resp.} \ [R,1+R] \, .
\end{equation}
Consequently if $V(R) = c^{-1} V_1(R)$ then $f(V(R)) \backslash V(R)$, resp. $V(R) \backslash f(V(R_1))$ are commensurable (e.g. contained in the double for $R \gg 0$) with domains of the form,
\begin{equation}\label{Res47}
\vert {\rm Im} (s) \vert \leq R \quad \mbox{and} \quad {\rm Re} (s) \in [-R,1-R] \, , \quad \mbox{resp.} \ [R,1+R]
\end{equation}
in the Fatou coordinates of \eqref{Res8} employed in the statement of \eqref{Res30}, so, by op. cit. the difference,
\begin{equation}\label{Res48}
{\rm Res}_f^{V(R)} (\vert \omega \vert) - {\rm Res}_f^{V(R)} (\vert \lambda_{\omega} \vert^{2/m} \, \overline W W)
\end{equation}
for $W$ as in \eqref{Res12} tends to zero as $R \to \infty$, and whence \eqref{Res41} by \eqref{Res13}.
\end{proof}

In order to address how \thref{Dclaim:3} changes in the presence of poles we require,

\begin{ConsDef}\thlabel{Dcon:2}
Let $X$ be a compact Riemann surface then for any 2 points $x,y \in X$ there is a solution to the equation,
\begin{equation}\label{Res102}
dd^c \, G_{x,y} = \delta_x - \delta_y \, , \quad d^c = \frac{\partial - \overline\partial}{4\pi}
\end{equation}
which is unique up to a constant. As such if a finite group, $\Gamma$, acts on the right of $X$ with coarse moduli $\pi : X \to S := X/\Gamma$ and $\pi (x) \ne \pi (y)$ with neither a branch point, there is a well defined real number,
\begin{equation}\label{Res103}
g_{x,y} (\gamma) := G_{x,y} (x^{\gamma}) - G_{x,y} (y^{\gamma}) \, , \quad \gamma \in \Gamma \, .
\end{equation}
In this context observe, for $y$ fixed, $\hbar (y)$ not a branch point,
\begin{equation}\label{Res104}
\begin{matrix}
(1) &g_{x,y} (\gamma) = g_{y,x} (\gamma) \, , \ \gamma \in \Gamma \, . \hfill \\
(2) &g_{x^{\delta} , y^{\delta}} (\gamma) = g_{x,y} (\delta \,  \gamma \, \delta^{-1}) \, , \ \gamma , \delta \in \Gamma \, . \hfill \\
(3) &g_{x,y} (\gamma) \to +\infty \, , \ x \to y^{\gamma - 1} \, . \hfill \\
(4) &g_{x,y} (\gamma) \to - \infty \, , \ x \to c \, , \ \gamma \in {\rm stab}_{\Gamma} (c) \, . \hfill
\end{matrix}
\end{equation}
In our immediate context we arrive to such a pair starting from an embedding ${\mathbb G}_m \hookrightarrow {\mathbb P}_{\mathbb C}^1$ together with a choice amongst the complement, $\{0,\infty\}$, so as to get a zero cycle,
\begin{equation}\label{Res105}
\delta_0 - \delta_{\infty} \, .
\end{equation}
Consequently, up to a multiplicative constant, there is a unique function $\sigma$ on ${\mathbb G}_m$ such that,
\begin{equation}\label{Res106}
dd^c \log \vert \sigma \vert^2 = \delta_{\infty} - \delta_0
\end{equation}
so that the choice \eqref{Res105} affords a unique differential form,
\begin{equation}\label{Res107}
\tau = \frac{d\sigma}{\sigma} \, .
\end{equation}
Now suppose for some $m \geq 2$ we have a meromorphic $m$-form $w$ such that,
\begin{equation}\label{Res108}
w-\tau^{\otimes m} \ \mbox{is regular at} \ \{0,\infty\}
\end{equation}
and otherwise the poles are simple. As such if $w \ne \tau^{\otimes m}$ there is a connected $\mu_m$-cover $\pi : X \to {\mathbb P}^1$, and a differential form $v$ on $X$ such that,
\begin{equation}\label{Res109}
\pi^* w = v^{\otimes m} \, , \quad v^{\alpha} = \alpha \cdot v \, , \quad \alpha \in \mu_m \, .
\end{equation}
Plainly $\pi$ is unramified over $0$ and $\infty$, so the choice of $v$ in \eqref{Res108} determines unique points $0_v$, resp. $\infty_v$, in $\pi^{-1} (0)$, resp. $\pi^{-1} (\infty)$, such that,
\begin{equation}\label{Res110}
v-\pi^* \, \frac{d\sigma}{\sigma} \ \ \mbox{is regular at $0_v$, resp. $\infty_v$}
\end{equation}
and we define,
\begin{equation}\label{Res110.bis}
g_w := \sum_{\alpha \in \mu_m \backslash \{1\}} (1-\overline\alpha) \, g_{0_v , \infty_v} (\alpha)
\end{equation}
which by item (1), resp. (2) of \eqref{Res104} is independent of the choices \eqref{Res105}, resp. \eqref{Res107}.

Finally in the \thref{Dcon:1} what we a priori have is a sequence of forms $w_R$, and constants $\lambda^0_R , \lambda_R^{\infty}$ satisfying the estimates of \eqref{Res24} and \eqref{Res37.bis.bis}. However  by \eqref{Res34}, cf. \eqref{Res37.bis.bis}, these $w_R$ converge, as $R \to \infty$, on ${\mathbb P}_{\mathbb C}^1$ to a meromorphic $m$-form $w$ with at worst the same poles, while by \thref{Dclaim:2} either $\lambda (\omega)$ of op. cit. is zero, or $\lambda (\omega)^{-1} w$ satisfies \eqref{Res108}, so we may define,
\begin{equation}\label{Res111}
\gamma_{\omega} = \left\{\begin{matrix}
g_{\lambda (\omega)^{-1} w} &, &\lambda(\omega) \ne 0 \ \mbox{and} \ w \ne \tau^m \lambda (\omega) \\
0 &, &\mbox{otherwise.} \hfill
\end{matrix}\right.
\end{equation}
Finally, although not part of the definitions we may usefully note that by \eqref{Res37.bis.bis},
\begin{equation}\label{Res112}
\underset{R \leq {\rm Re} (s) \leq R + 1 \atop \vert {\rm Im} (s) \vert \leq R}{\int} \left\vert \vert w \vert - \vert w_R \vert \right\vert \ll R^{2-b} \, .
\end{equation}
\end{ConsDef}

With this preamble in mind, we have,

\begin{fact}\thlabel{Dclaim:3.bis}
Let everything be as in \thref{Dclaim:3}, but with $\omega$ having at worst simple poles along an invariant divisor in an attracting petal with finitely many orbits then,
\begin{equation}\label{Res113}
{\rm Res}_f (\vert \omega \vert) \leq \vert \lambda_{\omega} \vert \left\{ e^{-2} \, {\rm Re} (\nu) + \frac{\gamma_{\omega}}{2\pi}\right\} .
\end{equation}
\end{fact}

\begin{proof}
If $\omega$ doesn't have a pole in the given attracting petal $P$ there is nothing to do. Otherwise let $V(R)$ be as in \eqref{Res45} et seq. with $P_R$ its intersection with $P$ then, exactly as in op. cit. we have,
\begin{equation}\label{Res114}
{\rm Res}_f^{V(R)} (\vert\omega\vert) - {\rm Res}_f^{V(R)} \left( \vert \lambda_{\omega}\vert^{2/m} \, \overline W W \right) + o(R) = \underset{P_R \backslash f(P_R)}{\int} \left( \vert \lambda_{\omega}\vert^{2/m} d\overline s \, ds - \vert \omega \vert \right)
\end{equation}
where $s$ is the Fatou coordinate in the attracting petal, so if $\lambda_{\omega} = 0$ is zero we're also done. Otherwise $\lambda_{\omega} \ne 0$, thus, without loss of generality $\lambda_{\omega} = 1$, while by \eqref{Res112} we can replace $\vert \omega \vert$ in \eqref{Res114} by $\vert w \vert$ of \eqref{Res111}, i.e. we require to estimate,
\begin{equation}\label{Res115}
\underset{P_R \backslash f(P_R)}{\int} d\overline s \, ds - \frac{\vert w \vert}{4\pi^2}
\end{equation}
from above. In particular we have the $\mu_m$-cover $\pi : X \to {\mathbb P}^1$ of \eqref{Res108} et seq. Furthermore by construction, i.e. \eqref{Res44} et seq., $P_R - f (P_R)$ is a fundamental domain for the complement of 2 connected open sets $N_R(0)$, $N_R(\infty)$ forming a basis of the topology at $0$ and $\infty$ which decrease with $R$, i.e. \eqref{Res115} is identically,
\begin{equation}\label{Res116}
\frac{1}{4\pi^2} \underset{U_R := {\mathbb P}^1 \backslash N_R(0) \cup N_R(\infty)}{\int} \left( \frac{d \overline \sigma \, d\sigma}{\vert \sigma \vert^2} - \vert w \vert \right)
\end{equation}
for $\sigma$ as in \eqref{Res106}. Now on $X$, by \eqref{Res109},
\begin{equation}\label{Res117}
\pi^* \vert w \vert = \overline v v \, , \quad v - \sum_{\alpha \in \mu_m} \alpha \cdot (\alpha^* \partial g_{0_V , \infty_V}) = \theta \in \Gamma (X,\omega_X)
\end{equation}
so we have,
\begin{eqnarray}\label{Res118}
m \underset{U_R}{\int} \vert \omega \vert & =&\underset{\pi^{-1} (U_R)}{\int} \overline v v \geqq \sum_{\alpha,\beta \in \mu_m} \overline\alpha \beta \underset{\pi^{-1} (U_R)}{\int} \alpha^* \overline\partial g_{0_V , \infty_V} \beta^* \partial g_{0_V , \infty_V}  \\
&+ & \sum_{\alpha \in \mu_m} \overline\alpha \underset{\pi^{-1} (U_R)}{\int} \alpha^* \overline\partial g_{0_V , \infty_V} \cdot \theta - \sum_{\alpha \in \mu_m} \alpha \underset{\pi^{-1} (U_R)}{\int} \alpha^* \partial g_{0_V , \infty_V} \overline\theta \nonumber
\end{eqnarray}
wherein the integrals of the ``$\overline\partial g \cdot \theta$'' terms tend to zero as $R \to \infty$, by Stokes, while,
\begin{equation}\label{Res119}
m \underset{U_R}{\int} \frac{d\overline \sigma \, d\sigma}{\vert \sigma \vert^2} = \sum_{\alpha , \beta \in \mu_m} \ \underset{\pi^{-1} (U_R)}{\int} \alpha^* \overline\partial g_{0_V , \infty_V} \, \beta^* \partial g_{0_V , \infty_V} \, .
\end{equation}
As such up to $o(R)$ an upper bound for \eqref{Res116} is,
\begin{equation}\label{Res120}
\frac{2\pi}m \sum_{\alpha , \beta \in \mu_m} (1-\overline\alpha \beta) \, g_{0_V , \infty_V} \, (\alpha\beta^{-1}) = 2\pi \, \gamma_{\omega}
\end{equation}
and whence \eqref{Res113} after correction by the constant in \eqref{Res116}.
\end{proof}

In principle there's absolutely no reason why equality cannot occur in \eqref{Res113} which we address by way of,

\begin{scholion}\thlabel{Dschol:2}
In the above notations of \thref{Dclaim:3.bis} and its proof, suppose further that $m=2$, and denote the involution by,
\begin{equation}\label{Res121}
* : X \longrightarrow X : x \longmapsto x^*
\end{equation}
then we have an anti-symmetric function,
\begin{equation}\label{Res122}
g = g_{0_V , \infty_V} - g_{0^*_V , \infty^*_V}
\end{equation}
which in turn yields a quadratic differential $\partial g^{\otimes 2} = w$ on ${\mathbb P}^1$ whose square root in the sense of \eqref{Res108} et seq. is again $X$, with $v$ of op. cit. equal to $\partial g$, so that in this case $\theta$ of \eqref{Res117} is zero, and the estimate \eqref{Res120} of \eqref{Res116} is sharp. A particular example that achieves this is when $X = {\mathbb P}^1$ and there are 2 simple poles. In which case the general form is,
\begin{equation}\label{Res123}
w_c = \frac{(\sigma + c)^2}{(\sigma + 1)(\sigma + c^2)} \left( \frac{d\sigma}{\sigma} \right)^{\otimes 2} \ , \quad c \in {\mathbb C} \backslash \{1\}
\end{equation}
with $\pi$ branched over $-1$ and $-c^2$, and,
\begin{equation}\label{Res124}
g_w = 4 \log \frac{4c}{(1+c)^2} \, .
\end{equation}
Consequently we will get examples where \eqref{Res113} is sharp, and $\gamma_{\omega}$ of op. cit. is given by \eqref{Res124} if, for say $e=2$, we can find a quadratic differential which limits on \eqref{Res123} in right strips, $\sigma^{-2} d\sigma^{-2}$ in left strips, and satisfies the various hypothesis of \thref{Dcon:1} and \thref{Dhyp:1}. To this end suppose that $f$ is,
\begin{equation}\label{Res125}
z \longmapsto z/(z+1)
\end{equation}
so the Fatou coordinate is identically $s = z^{-1}$ in both attracting the petal $P_{\alpha}$ and the repelling one $P_{\beta}$, while for some $Y \gg 0$,
$$
P_{\alpha} \cap P_{\beta} = \{ s \mid \vert {\rm Im} (s) \vert > Y \} \, .
$$
Just as in \thref{DLem:1} we go through the steps for solving the $\overline\partial$-equation, starting with a partition of unity,
\begin{equation}\label{Res126}
1 = \rho_{\alpha} + \rho_{\beta} \, , \quad \rho_{\bullet} : {\mathbb R} \to {\mathbb R}_{\geq 0}
\end{equation}
of the real line with $\rho_{\alpha} \, \rho_{\beta} = 0$ on $\vert x \vert \geq 1/2$, which on pulling back by $x = {\rm Re} (s)$ we regard as a partition of unity in the $s$-variable, and where the relevant co-cycle is,
\begin{equation}\label{Res127}
h_{\alpha \beta} := \frac{\sigma^2 w_c}{(d\sigma)^2} - 1 \, \Bigl\vert_{P_{\alpha} \cap P_{\beta}} \, , \quad \vert h_{\alpha \beta} \vert \ll e^{-\vert {\rm Im} (s)\vert}
\end{equation}
from which we have the $C^{\infty}$ functions,
\begin{equation}\label{Res128}
H_{\alpha} = \rho_{\beta} \, h_{\alpha \beta} \bigl\vert_{P_{\alpha}} \, , \quad H_{\beta} = -\rho_{\alpha} \, h_{\alpha \beta} \bigl\vert_{P_{\alpha}} \, , \quad h_{\alpha \beta} = H_{\alpha} - H_{\beta} \, .
\end{equation}
In particular, therefore, we get a $C^{\infty} (0,1)$ form,
\begin{equation}\label{Res129}
\theta = \overline\partial \, H_{\alpha} = \overline\partial \, H_{\beta} \, , \quad {\rm supp} (\theta) \subseteq \{\vert{\rm Res}(s)\vert \leq 1/2 \}
\end{equation}
and unlike \eqref{Res19}, we solve the $\overline\partial$-equation here by way of,
\begin{equation}\label{Res130}
H(s) = \frac1{2\pi} \underset{{\mathbb C} \vert {\rm Re} (t) \vert , \atop \vert {\rm Im} (t) \vert \geq R}{\int} \frac{\theta \, \partial t}{t-s}
\end{equation}
which by \eqref{Res127} and \eqref{Res129} not only converges but satisfies,
\begin{equation}\label{Res131}
\vert H \vert \ll \vert s \vert^{-1} \, .
\end{equation}
As such $h_{\alpha} = H_{\alpha} - H$, resp. $h_{\beta} = H_{\beta} - H$ are holomorphic on $P_{\alpha}$, resp. $P_{\beta}$ while also satisfying the growth estimate \eqref{Res131} thanks to \eqref{Res127} and the choice of \eqref{Res126}. In particular the quadratic differential, $\widetilde\omega_c$, defined by,
\begin{equation}\label{Res132}
w_c - h_{\alpha} \, ds^{\otimes 2} \ \mbox{on} \ P_{\alpha} \, , \quad ds^{\otimes 2} (1-h_{\beta}) \ \mbox{on} \ P_{\beta}
\end{equation}
is not only non-zero but limits on $w_c$, resp. $ds^{\otimes 2}$ as ${\rm Re} (s) \to \infty$, resp. ${\rm Re} (s) \to -\infty$. Such a form certainly satisfies \eqref{Res36}, but, may not, a priori satisfy \eqref{Res34}. However from \eqref{Res130} we have the growth estimate,
\begin{equation}\label{Res133}
\vert f^* \, \widetilde\omega_c - \widetilde\omega_c \vert \ll \vert s \vert^{-2}
\end{equation}
which in turn implies that in the variable $z$ of \eqref{Res125},
\begin{equation}\label{Res134}
z^2 (f^* \, \widetilde\omega_c - \widetilde\omega_c)
\end{equation}
is holomorphic. Consequently we can, cf. \thref{cd:fact6} and \thref{lem:ct1}, find a quadratic differential $q_c$, such that,
\begin{equation}\label{Res135}
z^3 q_c \ \mbox{is holomorphic, and} \ f^* q_c - q_c = f^* \widetilde\omega_c - \widetilde\omega_c \ \mbox{mod holomorphic.}
\end{equation}
In particular, therefore, $q_c$ satisfies the growth estimate of \eqref{Res131}, so the asymptotic behaviour of $\omega_c  = \widetilde\omega_c - q_c$ is still the same of that of $\widetilde \omega_c$ post \eqref{Res132}. Unlike $\widetilde\omega_c$, however, it satisfies \eqref{Res34}, and we have a family of examples where \eqref{Res113} is sharp, but $\gamma_{\omega_c}$ of op. cit. can be arbitrarily large by \eqref{Res124}.
\end{scholion}

To which we can usefully add,

\begin{fact}\thlabel{Dclaim:3.b.bis.bis}
Let everything be as in \thref{Dclaim:3.bis} but with $\omega$ holomorphic in repelling petals, and suppose $\lambda_p$ of \eqref{Res37} in every such is zero, with othewise arbitrary behaviour in attracting petals, then,
\begin{equation}\label{Res136}
{\rm Res}_f (\vert\omega\vert) \leq 0 \, .
\end{equation}
\end{fact}

\begin{proof}
Again for $V_1(R)$ as in \eqref{Res45} et seq., let $P_-(R)$, resp. $P_+(R)$ be its intersection with repelling, resp. attracting petals, then by op. cit. and \eqref{Res1},
\begin{equation}\label{Res137}
{\rm Res}_f (\vert\omega\vert) \leq \lim_{R \to \infty} \underset{f(P_-(R)) \backslash P_-(R)}{\int} \vert \omega \vert \, .
\end{equation}
The domain of integration, however, in \eqref{Res137} is \eqref{Res46} in every repelling petal, so, by hypothesis, and \eqref{Res47} et seq., the right hand side of \eqref{Res137} is zero.
\end{proof}

Related to this is the matter of difference equations in the Fatou coordinate of \eqref{Res8} which for want of a better place to put it we address in,

\begin{scholion}\thlabel{DSchol:1}
Let $h(s)$ be holomorphic on a sector $S$, i.e. $\vert {\rm arg} (s) \vert < \varepsilon$ for a suitable determination of the argument, around the real axis of an attracting petal, \eqref{Res7}, and suppose $h$ is flat, i.e. in the Fatou coordinate $s$,
\begin{equation}\label{Res49}
\forall \, n \geq 0 \, , \quad \exists \, c_N , C_N > 0 \ni \vert s \vert \leq c_n \Rightarrow \vert h(s) \vert \leq C_n \vert s \vert^{-n}
\end{equation}
then there is a unique flat function $g$, on a possibly smaller sector, such that,
\begin{equation}\label{Res50}
g(s+1) - g(s) = h(s) \, .
\end{equation}
\end{scholion}

\begin{proof}
We'll use uniqueness repeatedly on domains of the form,
\begin{equation}\label{Res51}
V(R,r) := \{s \mid {\rm Re} (s) > R \, , \ \vert {\rm Im} (s) \vert < r \}
\end{equation}
so let's start with it. Plainly it's equivalent to showing that a solution of \eqref{Res50} with $h=0$ on $V(R,r)$ is zero. Now such a solution is equivalently a holomorphic function, $g$, say, on a neighbourhood of the unit circle $C \hookrightarrow {\mathbb C}^{\times}$ in the notation of \eqref{Res23} et seq. As such if the limit,
\begin{equation}\label{Res52}
\lim_{s \to \infty \atop s \in {\mathbb R}} g(s)
\end{equation}
exists then it is equal to $g(c)$ for any $c \in C$, so $g$ is constant, whence $0$ if it's flat.

As to the existence we first solve the equation on a domain of the form \eqref{Res51}. To this end observe that \thref{DLem:1} is valid as stated with exactly the same proof if on the domains $V , V_{\alpha} , V_{\beta}$ of op. cit. we further impose the condition, $\vert {\rm Im} (s) \vert < r$ for some $r$. Now apply this to,
\begin{equation}\label{Res53}
h_{\alpha\beta} (s) := \left\{\begin{matrix}
h(s) &, &s \in V_{\alpha} \cap (V_{\beta} - 1) \\
0 &, &s \in V_{\alpha} \cap V_{\beta} \, . \hfill
\end{matrix} \right.
\end{equation}
As such, from the second condition in \eqref{Res53}, $h_{\alpha}$, $h_{\beta}$ of \thref{DLem:1} patch to a function $g_R$ on $V$ which satisfies,
\begin{equation}\label{Res54}
g_R (s+1) - g_R(s) = h(s) \, , \quad s \in V_{\alpha} \cap (V_{\beta} -1)
\end{equation}
by the first condition in op. cit. Now consider defining $g_{R+n}$ on open neighbourhoods of $R < {\rm Re} (s) < R+n+1$, ${\rm Im} (s) < r$ by way of,
\begin{equation}\label{Res55}
g_{R+n} (s) = g_{R+n-1} (s-1) + h(s-1) \, , \quad g_R = g_{R+0} = g_{R - 1}
\end{equation}
then what we've just said establishes the consistency of this formula for $n=0$, after which it's automatic. Better still we have,
\begin{equation}\label{Res55.bis}
g^R(s) := g_{R+n} (s) = \sum_{i=1}^n h(s-i) + g_R (s-n) \, , \quad R+n < {\rm Re} (s) < R+n+1
\end{equation}
so by \eqref{Res49} and \eqref{Res15} we have that,
\begin{equation}\label{Res56}
\lim_{{\rm Re} (s) \to \infty} g^R(s) \quad \mbox{exists and is} \ 0(R^{-a}) \ \mbox{any} \ a \, .
\end{equation}

In particular if, in a minor abus de language, we adjust $g^R$ by a constant then we can further suppose that the limit in \eqref{Res56} is zero without prejudice to the estimate on the size. Without, however, changing $r$ which governs the imaginary part this is equally true of $g^{R'}$ for any $R' > R$, while by the uniqueness considerations of \eqref{Res51} et seq., $g^R$ restricts to $g^{R'}$ so by \eqref{Res56},
\begin{equation}\label{Res57}
\lim_{s \to \infty \atop s \in {\mathbb R}} g^R (s+t) = 0 \, , \quad t \in {\mathbb R} \, , \quad \vert t \vert < r \, .
\end{equation}

However by \eqref{Res55.bis} the difference between $g^R(s)$ and the left hand side of \eqref{Res57} is,
\begin{equation}\label{Res58}
\sum_{i=0}^{\infty} h(s+i) = 0 (\vert s \vert^{-a}) \, , \quad \forall \, a
\end{equation}
wherein the indicated bounded follows from \eqref{Res49}.

Finally, therefore, if we increase $R$ to $R+1$ we can also increase $r$ from around $R\varepsilon$ to $R\varepsilon + \varepsilon$, so, again by the uniqueness of \eqref{Res51} et seq., $g^{R+1}$ glues to $g^R$. As such we certainly obtain a holomorphic solution $g$ of \eqref{Res50} on a sector, while the asymptotic estimate of \eqref{Res58} only depnds on $c_n , C_n$ of \eqref{Res44}, so $g$ is flat in the sense of op. cit.
\end{proof}

%\end{document}

%%%%%%%%%%%%%%%

\vglue 1cm

\section{The Count}\label{ss:count}

We wish to contrast invariant cycles of endomorphisms of Riemann surfaces with their ramification, and to this end following \cite[\S 1]{adam} we introduce,

\begin{defn}\thlabel{defn:ct1}
Let $f : X \to X$ be an open endomorphism of a Riemann surface and $R'$ the subset of critical points whose forward orbit under $f$ is infinite, then an infinite tail 
\nomenclature[D]{Tail, \hyperref[defn:ct1]{infinite}}{} 
is an equivalence class of points in $R'$ under the relation,
\begin{equation}
\label{ct1}
x \sim y \quad \mbox{iff} \quad \exists \, a,b \in {\mathbb Z}_{\geq 0} \ni f^a(x) = f^b (y) \, .
\end{equation}
Amongst all such tails we further distinguish the {\it tame tails}, 
\nomenclature[D]{Tail, \hyperref[defn:ct1]{tame}}{} 
$R_t ({\mathbb C})$, wherein the forward orbit of some critical point (whence any by \eqref{ct1}) is eventually contained in the Fatou set of $f$, and its complement, $R_w({\mathbb C})$ of {\it wild tails}. 
\nomenclature[D]{Tail, \hyperref[defn:ct1]{wild}}{} 
In particular the ramification (Weil divisor) ${\rm Ram}_f$, of $f$ (understood with multiplicity) divides itself into,
\begin{equation}
\label{ct3}
{\rm Ram}_f = {\rm Ram}_b + {\rm Ram}_t + {\rm Ram}_w
\end{equation}
according to whether the orbit of a critical point has bounded cardinality, ${\rm Ram}_b$, 
\nomenclature[N]{Ramification, bounded}{\hyperref[ct3]{${\rm Ram}$}$_b$}
is a tame tail, ${\rm Ram}_t$, 
\nomenclature[N]{Ramification, tame}{\hyperref[ct3]{${\rm Ram}$}$_t$} 
or a wild one, ${\rm Ram}_w$.
\nomenclature[N]{Ramification, wild}{\hyperref[ct3]{${\rm Ram}$}$_w$}
\end{defn}

Similarly we will need some notation to distinguish the different types of cycles that may occur, i.e.

\begin{notation}\thlabel{not:ct1}
Let $f : X \to X$ be an open endomorphism of a Riemann surface then a finite, of cardinality $p$, forward orbit, $Z$, of a point will be called a cycle of order $p$, and according as the multiplier $\lambda$ of \eqref{Res3} satisfies,

\begin{enumerate}
\item[(+)] $\vert \lambda \vert < 1$, then we say $Z$ is attracting and use an indice or suffix $+$.
\nomenclature[D]{Cycle, \hyperref[not:ct1]{attracting}}{}
\item[$(++)$] $\vert \lambda \vert =0$, then we say $Z$ is super attracting and use an indice or suffix $++$.
\nomenclature[D]{Cycle, \hyperref[not:ct1]{super attracting}}{}
\item[$(0+)$] $\lambda \in \mu_{\infty}$, and \'Ecalle's residu iteratif, $\nu$ of \eqref{Res9} of some iterate enjoys ${\rm Re}(\nu) \leq 0$, then we say $Z$ is parabolic attracting and employ an indice or suffix $0+$.
\nomenclature[D]{Cycle, \hyperref[not:ct1]{parabolic attracting}}{}
\item[$(0-)$] $\lambda \in \mu_{\infty}$ and $\nu$ as above but with ${\rm Re}(\nu) > 0$, then we say $Z$ is parabolic repelling and employ an indice or suffix $0-$.
\nomenclature[D]{Cycle, \hyperref[not:ct1]{parabolic repelling}}{}
\item[(SD)] If $\lambda \in S^1 \backslash \mu_{\infty}$, and $f^p$ admits a conjugation to $z \mapsto \lambda z$, then we say $Z$ is a Siegel disc, and employ an indice or suffix (SD).
\nomenclature[D]{Cycle, \hyperref[not:ct1]{Cycle, Siegel disc}}{}
\item[(CR)] If $\lambda \in S^1 \backslash \mu_{\infty}$, and is not a Siegel disc, then we say that $Z$ is Cremer, and employ an indice or suffix (CR).
\nomenclature[D]{Cycle, \hyperref[not:ct1]{Cremer}}{}
\end{enumerate}

All such cycles except those of Cremer type belong to the Fatou set, or its boundary in the case of those which are parabolic, to which in addition there may be a domain,

\begin{enumerate}
\item[(HR)] Isomorphic to a maximal annulus, $A$, such that $f^p(A) = A$ for some $p$, and $f^p \!\!\mid_A$ is conjugate to $z \mapsto \lambda z$, $\vert \lambda \vert = 1$, for $z : A \hookrightarrow {\mathbb C}$ an embedding of $A$ as a standard annulus, which we refer to as a Herman ring and employ an indice or suffix HR.
\nomenclature[D]{Cycle, \hyperref[not:ct1]{Herman ring}}{}
\end{enumerate}

In particular, therefore, tame tails may be partitioned according as to whether they accumulate at an attractor, a super attractor, a parabolic point, or eventually belong to a Siegel disc or Herman ring, so, we write,
\begin{equation}
\label{ct4}
{\rm Ram}_t = {\rm Ram}_t^+ + {\rm Ram}_t^{++} + {\rm Ram}_t^{0+} + {\rm Ram}_t^{0-} + {\rm Ram}_t^{{\rm SD}} + {\rm Ram}_t^{{\rm HR}}
\end{equation}
to reflect this partition.
\nomenclature[N]{Tame ramification at an attracting}{cycle \hyperref[ct4]{${\rm Ram}$}$^+_t$}
\nomenclature[N]{Tame ramification at a super attracting}{cycle \hyperref[ct4]{${\rm Ram}$}$^{++}_t$}
\nomenclature[N]{Tame ramification at an attracting}{parabolic cycle \hyperref[ct4]{${\rm Ram}$}$^{0+}_t$}
\nomenclature[N]{Tame ramification at a repelling}{parabolic cycle \hyperref[ct4]{${\rm Ram}$}$^{0-}_t$}
\nomenclature[N]{Tame ramification at a Siegel disc}{\hyperref[ct4]{${\rm Ram}$}$^{{\rm SD}}_t$}
\nomenclature[N]{Tame ramification at a Herman ring}{\hyperref[ct4]{${\rm Ram}$}$^{{\rm HR}}_t$}
\end{notation}

With these preliminaries in mind let us proceed to introduce our,

\begin{setup}\thlabel{setup:ct1}
Let $f : X \to X$ be an open endomorphism of a Riemann surface and define a sequence of divisors by,
\nomenclature[N]{Finite approximations to the post}{critical set \hyperref[setup:ct1]{${\rm Ram}$}$^n$}
\begin{equation}
\label{ct5}
\begin{matrix}
\hfill {\rm Ram}^0 &:= &{\rm Ram}_f + ({\rm Ram}_f)_{\rm red} \hfill \\
\vert {\rm Ram}^{n+1} \vert &= &\vert {\rm Ram}^0 \vert \cup f (\vert {\rm Ram}\vert^n) \, , \quad n \in {\mathbb Z}_{\geq 0}
\end{matrix}
\end{equation}
wherein $\vert \ \vert$ denotes the support, and the multiplicity at $x \in \vert {\rm Ram}^{n+1}\vert$ is defined by,
\begin{equation}
\label{ct6}
{\rm ord}_x ({\rm Ram}^{n+1}) = \left\{\begin{matrix}
{\rm ord}_x ({\rm Ram}^0) \, , &x \in \vert {\rm Ram}^0 \vert \\
1 \, , &\mbox{otherwise.}
\end{matrix}\right.
\end{equation}
It follows that as divisors/sub-schemes of $X$,
\begin{equation}
\label{ct7}
f^* {\rm Ram}^{n+1} \geq {\rm Ram}^n \leq {\rm Ram}^{n+1}
\end{equation}
so that the rings of functions,
\begin{equation}
\label{ct8}
f^* {\mathcal O}_{{\rm Ram}^{n+1}} \longrightarrow {\mathcal O}_{{\rm Ram}^{n}} \longleftarrow {\mathcal O}_{{\rm Ram}^{n+1}}
\end{equation}
define an almost $f$-sheaf in the sense of \thref{cd:def1}. Better still the support of ${\rm Ram}^n$ decomposes according as to whether it belongs to a tame tail, a wild tail, or a finite orbit so just as in \eqref{ct3}, we have a disjoint decomposition,
\begin{equation}
\label{ct9}
{\rm Ram}^{n} = {\rm Ram}^{n}_b + {\rm Ram}^{n}_t + {\rm Ram}^{n}_w \, .
\end{equation}
In particular if ${\rm Ram}_f$ is finite then by \eqref{ct5}--\eqref{ct6}, ${\rm Ram}^{n}_b = {\rm Ram}^{n+1}_b$ for $n \gg 0$, and \eqref{ct8} defines an $f$-sheaf,
\begin{equation}
\label{ct10}
{\mathcal O}_{R_b} \, , \quad R_b = {\rm Ram}_b^n \, , \quad n \gg 0 \, .
\end{equation}
Similarly, and continuing to suppose ${\rm Ram}_f$ finite for ease of exposition, ${\rm Ram}_t^n$ admits a decomposition akin to that of \eqref{ct4}. In particular therefore, either:

\medskip

(a) For $\bullet \in \{ +, ++, 0+, 0-\}$ of \thref{not:ct1}, an accumulation point of $\underset{n \in {\mathbb Z}_{\geq 0}}{\bigcup} {\rm Ram}^n_{t,\bullet}$ is a periodic cycle $Z_{\bullet}$ of type $\bullet$, so, 
\nomenclature[N]{Weighted divisor $R_t$}{arising from tame ramification \hyperref[setup:ct1]{$R$}$_t$}
\begin{equation}
\label{ct11}
\varinjlim_n  {\rm Ram}^n_{t,\bullet}
\end{equation}
is a divisor, $R_t$, on the Riemann surface $X \backslash Z_{\bullet}$, and we have an $f$-sheaf,
\begin{equation}
\label{ct12}
i_* i^* {\mathcal O}_{\overline R_t}
\end{equation}
for $i : \overline R_t \hookrightarrow X$ the embedding of its closure. Equally if the set of parabolic cycles is non-empty then we can form the real blow up,
\begin{equation}
\label{ct13}
\widetilde\rho : \widetilde X \longrightarrow X
\end{equation}
in the same, and by hypothesis the accumulation points of \eqref{ct11} afford a finite cycle $\widetilde Z_{\bullet}$ in the exceptional divisor, $E$, of \eqref{ct13}, (the centres of the attracting petals in \thref{DRev:2}) and \eqref{ct11}--\eqref{ct12} hold for $i$ the inclusion of $\widetilde Z$ in $\widetilde X$. 

\medskip

(b) For $\bullet \in {\rm HR}$, resp. SD, and $n \gg 0$, ${\rm Ram}_{t,\bullet}^n$ eventually intersects the domain where an iterate of $f$ is conjugate to an irrational rotation. Thus subsequent iterates are dense in a finite number of circles, and we chose a closed annulus, resp. disc $i' : Z_{\bullet} \hookrightarrow X$, containing them. As such there is again, for $i$ the inclusion of the closure $\overline R_t$ an $f$-sheaf,
\begin{equation}
\label{ct14}
i_* \, i^* {\mathcal O}_X = i'_* \, i'^* {\mathcal O}_X \prod {\mathcal O}_{R_t^{\bullet} \backslash Z_{\bullet}} \, , \quad R_t^{\bullet} = \varinjlim_n {\rm Ram}_{t,\bullet}^n
\end{equation}
albeit the product decomposition in \eqref{ct14} is only as a sheaf, and never as an $f$-sheaf. Now if $X = {\mathbb P}^1$ then by \eqref{C2grp} any cycle of type $+,0+$, or $0-$ is the limit of points in ${\rm Ram}_t^n$ but otherwise this may fail. As such,

\medskip

(c) If at a cycle $Z_{\bullet}$ of type $\bullet = ++$, resp. SD, resp. a Herman ring, and $Z_{\bullet}$ has not already been defined and even if it has been defined at a SD then we define $Z_{\bullet}$ to be empty, resp. the cycle itself, resp. always add this cycle to $Z_{\bullet}$ at a Siegel disc, resp. the orbit of some circle $\vert z \vert = {\rm const}$ in the annulus of \thref{not:ct1} HR. In all these cases,
\begin{equation}
\label{ct15}
i_* \, i^* {\mathcal O}_{Z_{\bullet}}
\end{equation}
for $i$ the inclusion is an $f$-sheaf.

The one outstanding case therefore is,

\medskip

(d) If the cycle $Z_{\bullet}$ is Cremer, i.e. $\bullet = {\rm CR}$, of period $p$, then we have an $f$-sheaf,
\nomenclature[N]{Functions at a Cremer cycle to $2^{\rm nd}$}{order \hyperref[ct16]{${\mathcal O}$}$_{C_{\rm CR}}$}
\begin{equation}
\label{ct16}
{\mathcal O}_{C_{cr}} = \prod_{x \in Z_{\bullet}} {\mathcal O} / {\mathfrak m} (x)^2
\end{equation}
with corresponding $f$-ideal ${\mathcal O} (-C_{cr})$.

For the moment even on $X = {\mathbb P}^1$ we do not know there are finitely many cases of type (c) or (d), so choose a finite set of such with,
\begin{equation}
\label{ct17}
\widetilde j : \widetilde U = \widetilde X \backslash \bigcup_{\bullet \in {\rm (a), (b),} \atop {\rm finite \, choice \, of \, (c), (d)}} Z_{\bullet} \longhookrightarrow {\widetilde X} \longhookleftarrow \bigcup_{\bullet} Z_{\bullet} := Z : i = {\bigcup}_{\bullet} \, i_{\bullet}
\end{equation}
the resulting partition of the real blow up \eqref{ct13} into open and closed sets, then for any $n \geq 0$ we have according to the type of the attractor $\bullet$ of \thref{not:ct1} various $f$-quotients ${\mathcal Q}_{\bullet}$ of ${\mathcal O}_{\widetilde X}$ of the form,
\nomenclature[N]{Functions on a cycle and nearby}{ramification \hyperref[ct17.bis.bis]{${\mathcal Q}$}$_{\bullet}$}
\begin{equation}
\label{ct17.bis}
{\mathcal O} \longrightarrow j_! \, {\mathcal O}_{R_t^{\bullet}} \longrightarrow {\mathcal Q}_{\bullet} \longrightarrow i_* \, i^*_{\bullet} \, {\mathcal O}_{\widetilde X} \longrightarrow 0
\end{equation}
provided $\bullet$ is neither a Siegel disc nor a Herman ring, and
\begin{equation}
\label{ct17.bis.bis}
{\mathcal Q}_{\bullet} = {\mathcal O}_{R_t^{\bullet} \backslash Z_{\bullet}} \coprod \, i_* \, i^*_{\bullet} \, {\mathcal O}_{\widetilde X}
\end{equation}
otherwise. All of which fits into a short exact sequence,
\begin{equation}
\label{ct18}
0 \rightarrow j_! {\mathcal O}_{\widetilde U} (-R_b - R_t - {\rm Ram}^n_w - C_{cr} ) \rightarrow {\mathcal O}_{\widetilde X} \rightarrow {\mathcal O}_{C_{cr}} \prod  {\mathcal O}_{{\rm Ram}_w^n} \prod {\mathcal O}_{R_b} \prod_{\bullet} {\mathcal Q}_{\bullet} \rightarrow 0
\end{equation}
wherein, for ease of notation, we denote the kernel, resp. cokernel in \eqref{ct18} by,
\nomenclature[N]{Functions on all cycles,}{tame ramification, and finite approximation to wild \hyperref[ct18.bis]{$\widetilde{\mathcal Q}$}$_{\rm big}$}
\nomenclature[N]{Weighted divisor arising from}{tame ramification, finite approximation to wild, and Cremer points \hyperref[ct18.bis]{$\widetilde R$}}
\begin{equation}
\label{ct18.bis}
\widetilde j_! \, {\mathcal O}_{\widetilde U} (-\widetilde R) \, , \quad \mbox{resp.} \quad \widetilde {\mathcal Q}_{\rm big} \, .
\end{equation}
\end{setup}

Having got the notation out of the way our aim will be to study the long exact sequence obtained by applying ${\rm Ext}^{\bullet}_{\widetilde X / f} (\rho^* \Omega_X , \ )$ to \eqref{ct18}, so we begin with the middle term by way of,

\begin{claim}\thlabel{claim:ct1}
Let $f : {\mathbb P}^1 \to {\mathbb P}^1$ be a rational map of degree $d$ then, 
\begin{equation}
\label{ct19}
\dim_{\mathbb C} {\rm Ext}^q_{{\mathbb P}^1/f} (\Omega_{{\mathbb P}^1} , {\mathcal O}_{{\mathbb P}^1}) = \left\{\begin{matrix}
1 \, , &d=1 \, , \quad q = 0 \ \mbox{or} \ 1 \hfill \\
2d-2 \, , &d > 1 \, , \quad q=1 \hfill \\
0 &\mbox{otherwise.} \hfill
\end{matrix}\right.
\end{equation}
Thus, $\forall \, d$, ${\mathbb P}^1/f$ is smooth of the expected dimension.
\end{claim}

\begin{proof}
>From the spectral sequence of \thref{cd:fact3} applied to \eqref{ct19} we have,
\begin{equation}
\label{ct20}
{\rm E}_1^{p,q} : {\mathbb C}^3 \xrightarrow{ \, \sim \, } {\rm H}^0 ({\mathbb P}^1 , T_{{\mathbb P}^1}) \xymatrix{\ar[rr]^{d_1^{0,0}}_{f^* - \, df} &&} {\rm H}^0 ({\mathbb P}^1 , f^* T_{{\mathbb P}^1}) \xrightarrow{ \, \sim \, } {\mathbb C}^{2d+1} \, , 
\end{equation}
$$
\ p=0 \ \mbox{or} \ 1 \, , \ q=0 \, .
$$
As such \eqref{ct19} is equivalent to the assertion that $d_1^{0,0}$ has a non-trivial kernel of dimension 1 iff $d=1$. To check this let $\partial$ be a vector field on ${\mathbb P}^1$ with a zero of order $n$ (necessarily $0$, $1$ or $2$) at $y \in {\mathbb P}^1$ with $f(x) = y$ where locally $f^* y = x^e$, $e \geq 1$, and $\partial$ vanishes to order $m \geq 0$ then,
\begin{equation}
\label{ct21}
m = (n-1) \, e + 1
\end{equation}
so $n = 0 \Rightarrow e=1$. Consequently if $f$ is ramified at $x$, then $n=1$ or 2, but in the latter case $m$ is also $2$ and $e=1$, so $m=n=1$. If, therefore we normalise so that $\partial$ has zeroes at $0$ and $\infty$, then by \eqref{ct21} every point of $f^{-1} (0)$, resp. $f^{-1} (\infty)$, is also a zero of $\partial$, i.e.
\begin{equation}
\label{ct22}
f^{-1} (0) \cup f^{-1}(\infty) = \{0,\infty\}
\end{equation}
and all the ramification of $f$ is contained in the same, so, in an affine coordinate,
\begin{equation}
\label{ct23}
\partial = z \frac{\partial}{\partial z} \, , \ f(z) = \lambda z^d \ \mbox{or} \ \lambda z^{-d} \, , \ \lambda \in {\mathbb C}^{\times}
\end{equation}
from which $d_1^{0,0} (\partial) = (d-1) \partial$, and $d=1$. Alternatively $\partial$ has a double point, $\infty$, so we already know $d=1$, and $f$ fixes $\infty$ by \eqref{ct21}, i.e.
\begin{equation}
\label{ct24}
\partial = \partial / \partial z \, , \quad f(z) = z+1 \, .
\end{equation}

In either case if $d=1$ then ${\rm Ext}^1_{{\mathbb P}^1 / f} (\Omega^1 , {\mathcal O})$ has dimension 1, which is the dimension of PGL$_2$ modulo conjugation, while the if $d > 1$ the dimension of rational maps modulo PGL$_2$ is $2d-2$, so the deformation space is a smooth (champ) of the expected dimension.
\end{proof}

Such a satisfactory conclusion isn't true after a real blow up where a key role is played by,

\begin{lem}\thlabel{lem:ct1}
Let $i : Z \hookrightarrow {\mathbb P}^1$ be a parabolic cycle of period $p$ of a rational map, and let $\widehat{\mathcal O}_Z$ be the completion at $Z$ then if $r e + 1$ is the order of tangency of $f^p$ to a rational rotation of order $r$,
\begin{equation}
\label{ct24.bis}
\dim_{\mathbb C} {\rm Ext}_{{\mathbb P}^1/f}^q (\Omega_{{\mathbb P}^1} , i_* \, \widehat{\mathcal O}_Z) = \left\{\begin{matrix}
1 \, , &q=0 \hfill \\
e+1 \, , &q=1 \hfill \\
0 \, , &\mbox{otherwise.} \hfill \end{matrix} \right.
\end{equation}
\end{lem}

\begin{proof}
In this situation \thref{cd:fact3} becomes,
\begin{equation}
\label{ct25}
{\rm E}_1^{p,q} : T_{{\mathbb P}^1} \otimes \widehat{\mathcal O}_Z \xymatrix{\ar[rr]^{d_1^{0,0}}_{f^* - \, df} &&} (f^* T_{{\mathbb P}^1}) \otimes \widehat{\mathcal O}_Z \qquad p=0 \ \mbox{or} \ 1 \, , \ q = 0
\end{equation}
and plainly the kernel of \eqref{ct25} is equally that of $(f^p)^* -df^p$ computed at any point $z \in Z$. In the particular case that the multiplier of \eqref{Res3} is exactly 1, then this is the formal vector field, $\partial$, defined by,
\begin{equation}
\label{ct26}
\omega (\partial) = 1
\end{equation}
for $\omega$ the formal differential defined by \eqref{Res9}, while, in general if $f$ is tangent to a rational rotation of order $r$, then the field invariant by $f^r$ is also invariant by $f$, cf. \cite[post(5)]{adam}. As to the co-kernel of \eqref{ct25}, this is dual to the kernel of,
\begin{equation}
\label{ct26.bis}
{\rm H}_Z^1 ({\mathbb P}^1 , f^* \omega_{{\mathbb P}^1} \otimes \omega_{{\mathbb P}^1}) \underset{f_* - \, df^{\vee}}{-\!\!\!-\!\!\!-\!\!\!-\!\!\!\longrightarrow} {\rm H}^1_Z ({\mathbb P}^1 , f^* \omega_{{\mathbb P}^1} \otimes \omega_{{\mathbb P}^1})
\end{equation}
whose calculation is exactly the content of \cite[Lemma 1 (ii)]{adam}.
\end{proof}

Combining these we may, therefore, deduce,

\begin{cor}\thlabel{cor:ct1}
Let $\rho : \widetilde X \to X = {\mathbb P}^1$ be a real blow up in finitely many parabolic cycles, $Z_i$, invariant by a rational map $f$ then,
\begin{enumerate}
\item[(0)] ${\rm Ext}^{0}_{\widetilde X / f} (\rho^* \Omega_X , {\mathcal O}_{\widetilde X}) = {\rm Ext}^{0}_{X / f} (\Omega_X , {\mathcal O}_X)$ is given by \eqref{ct19}.
\item[(1)] If $q=1$, the Leray spectral sequence affords a short exact sequence,
$$
0 \to {\rm Ext}^1_{X/f} (\Omega_X , {\mathcal O}_X) \to {\rm Ext}^1_{\widetilde X / f} (\rho^* \Omega_{\widetilde X} , {\mathcal O}_{\widetilde X}) \to {\textstyle \underset{i}{\prod}} \, {\rm Ext}^{0}_{X / f} (\Omega_X , i_* \widehat{\mathcal O}_{Z_i}) \to 0
$$
wherein the kernel, resp. cokernel, are given by \eqref{ct19}, resp. \eqref{ct24.bis}.
\item[(2)] ${\rm Ext}^2_{\widetilde X / f} (\rho^* \Omega_{X} , {\mathcal O}_{\widetilde X}) = \underset{i}{\prod} \, {\rm Ext}^1_{{\mathbb P}^1 / f} (\Omega_{{\mathbb P}^1} , i_* {\mathcal O}_{Z_i})$ which is computed by \eqref{ct24.bis}.
\end{enumerate}
Irrespectively of the exact answer, therefore, ${\rm Ext}^{q}_{\widetilde X / f} (\rho^* \Omega_X , {\mathcal O}_{\widetilde X})$ is finite dimensional for all $q$.
\end{cor}

\begin{proof}
By the Leray spectral sequence is to be understood that afforded by the factorisation,
\begin{equation}
\label{ct27}
{\mathcal F} \longrightarrow \rho_* \, {\mathcal F} \longrightarrow {\rm Hom}_{X/f} (\Omega_X , \rho_* \, {\mathcal F})
\end{equation}
of ${\rm Hom}_{\widetilde X/f} (\rho^* \, \Omega_X , {\mathcal F})$ whether as $f$-sheaves or almost $f$-sheaves, since $\rho f = f\rho$. We therefore have,
\nomenclature[D]{Leray \hyperref[ct27]{spectral sequence}}{}
% for $f$-sheaves on a real blow up
\begin{equation}
\label{ct28}
{\rm E}_2^{p,q} = {\rm Ext}_{X/f}^p (\Omega_X , R^q \rho_* {\mathcal F}) \Rightarrow {\rm Ext}_{\widetilde X/f}^{p+q} (\rho^* \Omega_X , {\mathcal F})
\end{equation}
while $R^q \rho_*$ is given by \eqref{Dc34}, i.e. the $f$-sheaf,
\begin{equation}
\label{ct29}
R^1 \rho_* \, {\mathcal O}_{\widetilde X} \overset{\sim}{\longrightarrow} \prod_i i_* \,\widehat{\mathcal O}_{Z_i}
\end{equation}
since in cohomology, $\overline\partial \log \vert z \vert^2$ of op. cit. is invariant.
\end{proof}

To this we have various terms in the cokernel of \eqref{ct8} that we wish to analyse beginning with,

\begin{lem}\thlabel{lem:ct2}
Let everything be as in \thref{setup:ct1} with ${\rm Ram}_f$ finite, e.g. $X = {\mathbb P}^1$, then the vector spaces,
\begin{equation}
\label{ct30}
{\rm Ext}^q_{X/f} \left( \Omega_X , \left(\varprojlim_{n} {\mathcal O}_{{\rm Ram}^n}\right)_{\!\!\!c} \,\right) , \quad q = 0 \ \mbox{or} \ 1
\end{equation}
with the subscript $c$ as in \eqref{cd19}, are finite dimensional vector spaces of dimension the degree of ${\rm Ram}_f$, e.g. $2d-2$ if $f$ is a rational map of degree $d$.
\end{lem}

\begin{proof}
Observe that the $f$-sheaf in \eqref{ct30} is equally the direct limit
\begin{equation}
\label{ct31}
\varinjlim_n \, {\mathcal O}_{{\rm Ram}^n}
\end{equation}
where the map from $n$ to $n+1$ is extension by zero along the right most map in \eqref{ct7}, while the $f$-structure is,
\begin{equation}
\label{ct32}
{\mathcal O}_{{\rm Ram}^{n+1}} \overset{s}{\underset{= \, {\rm id}}{\longleftarrow\!\!\!-\!\!\!-}} {\mathcal O}_{{\rm Ram}^{n+1}} \underset{f^*}{\longrightarrow} {\mathcal O}_{{\rm Ram}^n} \ \overset{\rm extend}{\underset{{\rm by} \, 0}{-\!\!\!-\!\!\!-\!\!\!-\!\!\!\longrightarrow}} {\mathcal O}_{{\rm Ram}_{n+1}} \, .
\end{equation}

It therefore follows from \thref{cd:fact3} that \eqref{ct30} is the kernel and cokernel of the direct limit of $df-f^*$ where,
\begin{equation}
\label{ct33}
{\mathcal O}_{{\rm Ram}^{n+1}} \otimes f^* T_X \overset{df}{\longleftarrow} {\mathcal O}_{{\rm Ram}^{n+1}} \otimes T_X \overset{f^*}{\longrightarrow} {\mathcal O}_{{\rm Ram}^n} \otimes f^* T_X \overset{\rm extend}{\underset{{\rm by} \, 0}{-\!\!\!-\!\!\!-\!\!\!-\!\!\!\longrightarrow}} {\mathcal O}_{{\rm Ram}^{n+1}} \otimes f^* T_X \, .
\end{equation}
Better still for $n \gg 0$, the infinite tails define a finite subset $T_n \subseteq \vert {\rm Ram}^n\vert$, of cardinality the number of tails, such that,
\begin{equation}
\label{ct34}
{\rm Ram}^{n+1} = {\rm Ram}^n + f(T_n) \, , \ T_{n+1} = f(T_n) \, , \ f : T_n \to T_{n+1} \ \mbox{unramified.}
\end{equation}
As such we have for $n \gg 0$ an exact diagram,
\begin{equation}
\label{ct35}
\xymatrix{
T_X \otimes {\mathcal O}_{{\rm Ram}^n} \ar@{^{(}->}[d]_{\mbox{\footnotesize extend by $0$}} \dar[r]_{\!\!\!df}^{\!\!\!\mbox{\small $^{f^*}$}} &f^* T_X \otimes {\mathcal O}_{{\rm Ram}^n} \ar@{^{(}->}[d]^{\mbox{\footnotesize extend by $0$}} \ar[r] &C_n \ar[d] \ar[r] &0 \\
T_X \otimes {\mathcal O}_{{\rm Ram}^{n+1}} \ar[d] \dar[r]_{\!\!\!df}^{\!\!\!\mbox{\small $^{f^*}$}} &f^* T_X \otimes {\mathcal O}_{{\rm Ram}^{n+1}} \ar[d] \ar[r] &C_{n+1} \ar[r] &0 \\
T_X \otimes {\mathcal O}_{T^{n+1}} \ar[d] \dar[r]_{\!\!\!df}^{\!\!\!\mbox{\small $^{f^* = \, 0}$}} &f^* T_X \otimes {\mathcal O}_{T^{n+1}} \ar[d] \\
0 &0
}
\end{equation}
wherein the lowest horizontal is an isomorphism by \eqref{ct33}--\eqref{ct34}, so $C_n \to C_{n+1}$ is onto. As such the kernel and cokernel of the equaliser \eqref{ct33} stabalise, and have the same dimension. The dimension of the kernel, however is independent of $n$. Indeed if, more generally, $T_{n+1}$ of \eqref{ct34} were $\vert {\rm Ram}^{n+1}\vert \backslash \vert {\rm Ram}^n \vert$, then for $n \geq 1$, it is the image of a finite set $T_n$ such that $f : T_n \to T_{n+1}$ is unramified, so the diagram \eqref{ct35} equally yields an equality of kernels. Consequently we're reduced to calculating the kernel for $n=0$, $n+1 = 1$, which, by direct calculation is,
\begin{equation}
\label{ct36}
\prod_{x \in \vert {\rm Ram}^0\vert} {\mathfrak m} (x) \, T_X \xymatrix{ { \ } \ar@{^{(}->}[rr]^{\rm extend}_{\rm by \, zero} &&{ \ }} {\mathcal O}_{R_1} \otimes T_X \, .
\end{equation}
\end{proof}

Applying this to the cokernel terms in \eqref{ct18} we certainly have,

\begin{cor}\thlabel{cor:ct2}
Let everything be as in \thref{setup:ct1} with ${\rm Ram}_w^n$ the almost $f$-sheaf of wild tails then by \eqref{ct8}, for ${\rm Ram}_f$ finite, the dimension of,
\begin{equation}
\label{ct37}
{\rm Ext}_{X/f}^q (\Omega_X , {\mathcal O}_{{\rm Ram}_w^n}) \quad \mbox{for $q=0$, resp. 1}
\end{equation}
and $n \gg 0$ sufficiently large is at least,
\begin{equation}
\label{ct38}
\mbox{$\deg ({\rm Ram}_w)$, resp. $\deg ({\rm Ram}_w) - \#$ of wild tails.}
\end{equation}
\end{cor}

\begin{proof}
Observe that the structure of the almost $f$-sheaf defined by the pair ${\rm Ram}_w^n$, ${\rm Ram}_w^{n+1}$ is subtly different from \eqref{ct32}, i.e.
\begin{equation}
\label{ct39}
{\mathcal O}_{{\rm Ram}^{n}} \otimes f^* T_X \overset{\rm restrict}{\longleftarrow\!\!\!-\!\!\!-\!\!\!-} {\mathcal O}_{{\rm Ram}^{n+1}} \otimes f^*T_X \overset{df}{\longleftarrow} {\mathcal O}_{{\rm Ram}^{n+1}} \otimes T_X \overset{f^*}{\underset{= \, t}{-\!\!\!\longrightarrow}} {\mathcal O}_{{\rm Ram}^{n}} \otimes f^* T_X
\end{equation}
and, of course, by \thref{cd:fact3} we need to calculate the cokernel of \eqref{ct39}. To this end if their dimensions are ${\rm ext}^q$, $q=0$ or $1$, then by \eqref{ct34} and \eqref{ct39}
\begin{equation}
\label{ct40}
{\rm ext}^0 - {\rm ext}^1 = - \# \ \mbox{(of wild tales)}.
\end{equation}
Plainly, however, ${\rm Ext}^0$ contains at least the kernel of \eqref{ct33}, whose dimension is $\deg ({\rm Ram}_w)$ by a trivial variant of \thref{lem:ct2}.
\end{proof}

Continuing in the same vein we have,

\begin{cor}\thlabel{cor:ct3}
Let ${\mathcal Q}_{++}$ be the quotient of \eqref{ct17.bis} supported at a superattractor, and ${\rm Ram}_{++}$ the part of ${\rm Ram}_b + {\rm Ram}_t$ whose forward orbit limits on the said cycle then the dimension of the topological dual of,
\begin{equation}
\label{ct41}
{\rm Ext}^1_{X/f} (\Omega_X , {\mathcal Q}_{++})
\end{equation}
is the degree of ${\rm Ram}_{++}$.
\end{cor}

\begin{proof}
If no tail limits on the superattractor, this is clear from \thref{lem:ct2} and \thref{cd:fact6}. Otherwise we have the divisor $R_t^{++}$ of item (a) post \eqref{ct10}, which, by definition, doesn't include $Z_{++}$ in it's support, and the exact sequence \eqref{ct17.bis}. In particular if $i$ is the inclusion of $Z_{++}$ then,
\begin{equation}
\label{ct42}
{\rm Ext}^0_{X/f} (\Omega_X , i_* i^* {\mathcal O}_X) \subset \prod_{x \in Z_{++}} T_{X,x}
\end{equation}
and for any vector $(t_x)_{x \in Z_{++}}$ in the left hand group, $t_x$ is invariant by $f^p$, so its zero.

As such what we require to prove is that the topological dual of,
\begin{equation}
\label{ct43}
{\rm Ext}^1_{X/f} (\Omega_X , i_* i^* {\mathcal O}_X)
\end{equation}
has dimension the degree the part of ${\rm Ram}_f$ supported in $Z_{++}$. To this end let $1 \leq i \leq p$ be the points of the cycle with $z_i$ a coordinate around the same in which we can write cycle as,
\begin{equation}
\label{ct44}
f^* z_1 = z_p^{e_p} \, , \quad f^* z_{p-1} = z_{p-1}^{e_{p-1}} , \cdots , f^* z_2 = z_1^{e_1} \, , \quad e := e_1 \cdots e_p > 1 \, .
\end{equation}
Consequently the dual to $f^*$ is, cf. \thref{rmk:cd2},
\begin{equation}
\label{ct45}
f_* : f^* \Omega_X \otimes \omega_i \to \Omega_X \otimes \omega_{i+1} : w^i \left(\frac1{z_i}\right) f^* dz_{i+1} \otimes dz_i \mapsto w_{e_i}^i \left(\frac1{z_{i+1}}\right) dz_{i+1}^{\otimes 2} \, ,
\end{equation}
$$
1 \leq i \leq p
$$
where $w^i \in {\mathbb C} \, \langle 1/z_i \rangle$ is a power series with infinite radius of convergence vanishing at the origin, with $w_e$ the terms in the Taylor expansion at integers divisible by $e$. Similarly the dual to $df$ is,
\begin{equation}
\label{ct46}
(df)^{\vee} : f^* \Omega_X \otimes \omega_i \mapsto \Omega_X \otimes \omega_{i} : w^i \left(\frac1{z_i}\right) f^* dz_{i+1} \otimes dz_i \mapsto z_i^{e_i - 1} dz_i^{\otimes 2} \!\!\!\!\mod {\mathbb C} [z_i] \, .
\end{equation}

If, therefore, $w_n^i$, $n \geq 1$, is the coefficient of $z_i^n$ in $w^i$, then by \thref{cd:fact6} an element in the dual of \eqref{ct43} is a vector $w^i (1/z_i) \in {\mathbb C} \langle z_i \rangle \!\!\mod {\mathbb C} [z_i]$ such that,
\begin{equation}
\label{ct47}
w_{e_{i+1} + n}^{i+1} = e_{i+1}^{-1} \, w_{e_i + ne_i}^i \, , \quad n \geq 0 \, .
\end{equation}
Iterating \eqref{ct47}, in $i$, leads to,
\begin{equation}
\label{ct48}
w_{e_i+n}^i = e^{-1} w_{e_i + ne}^i \, , \quad n \geq 0 \, , \quad 1 \leq i \leq p
\end{equation}
so certainly $w_{e_i}^i = 0$, but iterating \eqref{ct48} in $n$ gives,
\begin{equation}
\label{ct49}
w_{e_i+n}^i = e^{-m} w^i_{e_i + ne^m} \, , \quad n,m \geq 0 \, , \quad 1 \leq i \leq p
\end{equation}
and $w_n^i$ is a convergent power series, so the right hand side of \eqref{ct49} goes to zero as $m \to \infty$. As such we even have the more precise statement that the dual has a basis,
\begin{equation}
\label{ct50}
z_i^{-j} \cdot f^* dz_{i+1} \otimes dz_i \, , \quad 1 \leq j \leq e_i - 1 \, , \quad 1 \leq i \leq p \, .
\end{equation}
\end{proof}

The next case to be addressed in the list of \thref{not:ct1} is, 

\begin{cor}\thlabel{cor:ct4}
Let ${\mathcal Q}_+$ be the quotient of \eqref{ct17.bis} supported at a simple attractor of type \thref{not:ct1} (+), then,
\begin{equation}
\label{ct51}
{\rm Ext}^1_{X/f} (\Omega_X , {\mathcal Q}_+)
\end{equation}
is finite dimensional of dimension at least the degree of that part ${\rm Ram}_+$ of ${\rm Ram}_b + {\rm Ram}_t$ which limits on the cycle $Z_+$.
\end{cor}

\begin{proof}
Again let $p$ be the period of the cycle, $i$ the the inclusion of $Z_+$ and $i_z$ the the inclusion of any point in its support then by \thref{cd:cor1.bis},
\begin{equation}
\label{ct52}
{\rm Ext}^q_{X/f} (\Omega_X , i_* \, i^* {\mathcal O}) = {\rm Ext}^q_{X/f^n} (\Omega_X , (i_z)_* \, i^*_z \, {\mathcal O}_X)
\end{equation}
and the latter are easily calculated in the coordinate $z$ of \eqref{Res5} to be the one dimensional space,
\begin{equation}
\label{ct52.bis}
{\mathbb C} \, z \, \frac{\partial}{\partial z} \, , \quad q = 0 \ \mbox{or} \ 1 \, .
\end{equation}
As such by the exact sequence \eqref{ct17.bis}, and \thref{lem:ct2}, \eqref{ct51} is finite dimensional while,
\begin{equation}
\label{ct53}
\dim_{\mathbb C} \, {\rm Ext}^1_{X/f} (\Omega_X , {\mathcal Q}_+) = -\chi ({\mathcal Q}_+) + \dim_{\mathbb C} \, {\rm Ext}_{X/f}^0 (\Omega_X , {\mathcal Q}_+) \geq \deg {\rm Ram}_+ \, .
\end{equation}
\end{proof}

In the remaining cases we need to be more carefull about the topology on ${\rm Ext}^1$. Indeed, already in \thref{cor:ct2} we didn't calculate the dimension of ${\rm Ext}^1$ but rather its maximal separated quotient. This was, however, out of laziness rather than necessity, and this is equally the case in,

\begin{cor}\thlabel{cor:ct5}
Let ${\mathcal Q}_0$ be the quotient of \eqref{ct17.bis} supported at a parabolic cycle $Z_0$ of type \thref{not:ct1} ($0+$) or  ($0-$) where $f$ is tangent to order $re+1$ to a rational rotation of order $r$ then for $\rho : \widetilde X \to X$ the real blow up in $Z_0$ the maximal separated quotient of,
\begin{equation}
\label{ct54}
{\rm Ext}_{\widetilde X / f}^1 (\rho^* \Omega_X , {\mathcal Q}_0)
\end{equation}
is finite dimensional of dimension at least,
\begin{equation}
\label{ct55}
\deg ({\rm Ram}_0) + e^2
\end{equation}
wherein ${\rm Ram}_0$ is the part of ${\rm Ram}_t + {\rm Ram}_b$ which limits on the cycle $Z_0$.
\end{cor}

\begin{proof}
Let $\widetilde Z \to Z_0$ be the cycle in $\widetilde X$ consisting of the centres of attracting petals on the exceptional divisor $E \hookrightarrow \widetilde X$, and $i : \widetilde Z \hookrightarrow \widetilde X$ its inclusion. In particular, exactly as in \thref{cor:ct3}, we can, by \thref{cd:cor1.bis}, suppose that $Z_0$ is a point, so $\widetilde Z$ is $e$-points on $E$. As ever, therefore, \thref{cd:fact3} tells us that the kernel and cokernel of,
\begin{equation}
\label{ct56}
i^*_Z \, \rho^* \, T_X \overset{f^*}{\underset{df}{\rightrightarrows}} i_Z^* \, f^* \rho^* \, T_X
\end{equation}
calculate ${\rm Ext}^q_{\widetilde X / f} (\rho^* \Omega_X , {\mathcal O}_{\widetilde X})$, $q=0$, resp. $1$, for ${\mathcal O}_{\widetilde X}$ holomorphic functions $C^{\infty}$ up to the boundary as per \thref{Dc:not1}. As such if $\widehat{\mathcal O}_{Z_0}$ is the completion of $X$ in $Z_0$, then the Borel-Ritt theorem and \thref{cd:fact6} tells us that the maximal separated quotients of these Ext groups are the kernel and cokernel of,
\begin{equation}
\label{ct57}
\xymatrix{
\left( T_X \otimes \widehat{\mathcal O}_{Z_0} \right)^{\!\!\oplus e} \dar[rr]^{\mbox{\footnotesize$^{(f^*)^{\oplus e}}$}}_{(df)^{\oplus e}} &&\left( f^* T_X \otimes {\mathcal O}_{\widehat Z_0} \right)^{\!\!\oplus e}
}
\end{equation}
where we have $e$-copies for the $e$-attracting petals. Consequently the groups in question have, by \thref{lem:ct1}, dimension $e$, resp. $e^2 + e$, and whence Euler characteristic $-e^2$. Exactly, therefore, as in \eqref{ct53}, we conclude to \eqref{ct54} by \eqref{ct17.bis} and \thref{lem:ct2}.
\end{proof}

In the outstanding case, however, we must be genuinely carefull about the distinction between Ext groups and their separated quotients should the multiplier be very well approximated by rationals, i.e.

\begin{cor}\thlabel{cor:ct6}
Let ${\mathcal Q}_{\rm HR}$ be the quotient of \eqref{ct17.bis} supported on a cycle of type \thref{not:ct1} (HR) then the dimension of the maximal separated quotient of,
\begin{equation}
\label{ct58}
{\rm Ext}^1_{X/f} (\Omega_X , {\mathcal Q}_{\rm HR})
\end{equation}
is at least,
\begin{equation}
\label{ct59}
\min \{1,\deg ({\rm Ram}_{\rm HR})\}
\end{equation}
where ${\rm Ram}_{\rm HR}$ is the part of ${\rm Ram}_t$ which limits on circles in a Herman ring.

There are 2-different structures at play here, which we may reasonably distinguished by first proving,
\end{cor}

\begin{lem}\thlabel{lem:ct3}
Let $i:Z \hookrightarrow X$ be the inclusion of an invariant closed annulus, disc, or circle, or point in the case of a disc, in a Herman ring or Siegel disc $f : X \to X$ of multiplier $\lambda$ then the dimension of,
\begin{equation}
\label{ct60}
{\rm Ext}^0 (\Omega_X , i_* \, i^* {\mathcal O}_X)
\end{equation}
is always $1$, as is the dimension of the maximal separated quotient of,
\begin{equation}
\label{ct61}
{\rm Ext}^1 (\Omega_X , i_* \, i^* {\mathcal O}_X) \, .
\end{equation}
The latter, however, will be infinite dimensional if $\lambda$ is very well approximated by rationals.
\end{lem}

\begin{proof}
By definition there is a coordinate $z$ on $X$ such that, $f(z) = \lambda z$ so \eqref{ct60} is exactly as in \eqref{ct52}--\eqref{ct52.bis}. Similarly, there is an obvious power series solution to
\begin{equation}
\label{ct62}
(f^* - df)(\tau) = \sigma \in i^* f^* T_X \, , \quad \sigma \ne z \, \partial / \partial z \, .
\end{equation}
However, whether or not this converges depends on the approximatibility of $f$ by rationals. Nevertheless what is true is that the functional defined by $(dz/z)^{\otimes z}$ vanishes on the image of $f^* -df$, and indeed it's the only such, so the dimension of the maximal separated quotient is $1$.
\end{proof}

We may, therefore, return to,

\begin{proof}[Proof of \thref{cor:ct5}]
Replacing everything by its maximal separated quotient, say ${\rm Ext}^1 \to \overline{\rm Ext}^1$, we still have \cite[Lemme 1]{ramis} a long exact sequence associated to \eqref{ct17.bis}, so, by \thref{lem:ct3} there's either nothing to do; or we conclude,
\begin{equation}
\label{ct63}
\dim_{\mathbb C} {\rm Ext}^1_{X/f} (\Omega_X , {\mathcal Q}_{\rm HR}) = \dim {\rm Hom}_{X/f} (\Omega_X , {\mathcal Q}_{\rm HR})  \geq \dim {\rm Hom}_{X/f} (\Omega_X , {\mathcal O}_{R \backslash Z}) \, .
\end{equation}
Again however both $f^*$ and $df$ always vanish on the left hand side of \eqref{ct36}, so the latter dimension is at least $\deg ({\rm Ram}_{\rm HR})$.
\end{proof}

The case of Siegel discs is subtly different since ${\rm Ram}_b$ might limit on the centre of the disc, which by construction, item (c) of \thref{setup:ct1} we always remove, i.e.

\begin{cor}\thlabel{cor:ct7}
Let ${\mathcal Q}_{\rm SD}$ be the quotient of \eqref{ct17.bis} limiting on a cycle of type \thref{not:ct1} (SD), then the dimension of the maximal separated quotient of,
\begin{equation}
\label{ct64}
{\rm Ext}_{X/f}^1 (\Omega_X , {\mathcal Q}_{\rm SD})
\end{equation}
is at least,
\begin{equation}
\label{ct65}
\deg ({\rm Ram}_{b,{\rm SD}}) + \min \{1,\deg ({\rm Ram}_{t,{\rm SD}})\}
\end{equation}
where ${\rm Ram}_{t,{\rm SD}} \, , \ \mbox{resp.} \ {\rm Ram}_{b,{\rm SD}}$ is the part of ${\rm Ram}_t$, resp. ${\rm Ram}_b$ which limits on circles, resp. the origin of the Siegel disc.
\end{cor}

\begin{proof}
We have cases according to whether the support of ${\mathcal Q}$ in the disc is closed discs, or the orbit of the point. The former case is, however, identical to the calculation of \thref{cor:ct5} for Herman rings, so, without loss of generality we may suppose the support is the latter, and, we let $i : Z \hookrightarrow X$ be the forward orbit of the origin, and consider the quotient of ${\mathcal Q}_{\rm SD}$ defined by,
\begin{equation}
\label{ct66}
0 \longrightarrow i_* \, {\mathfrak m} (Z)^2 \longrightarrow {\mathcal Q}_{\rm SD} \longrightarrow {\mathcal Q}_2 \longrightarrow 0
\end{equation}
whose ${\rm Ext}^1$ will certainly be that of the maximal separated quotient for the same reason as \thref{lem:ct3}. In particular by \thref{cd:cor1.bis}, we have the class of \eqref{ct52.bis},
\begin{equation}
\label{ct67}
\tau \left(= z \, \frac{\partial}{\partial z} \right) \in {\rm Ext}^0_{X/f} \left(\Omega , i_* \, \frac{{\mathfrak m}(z)}{{\mathfrak m}(z)^2}\right)
\end{equation}
and since $f^*$ is zero on every maximal ideal of ${\mathcal O}_R$ by construction, \eqref{ct67} extends by zero to the points of ${\rm Ram}_{\rm SD}$ which are not in $Z$. Thus we have a short exact sequence,
\begin{equation}
\label{ct68}
0 \longrightarrow {\rm Ext}_{X/f}^0 (\Omega , {\mathcal O}_{R \backslash Z}) \longrightarrow {\rm Ext}^0_{X/f} (\Omega , {\mathcal Q}_2) \longrightarrow {\mathbb C} \tau \longrightarrow 0
\end{equation}
while ${\rm Ext}^0_{X/f} (\Omega , {\mathcal Q}_2)$ has Euler characteristic zero, so we get \eqref{ct65} for the dimension of ${\rm Ext}^1_{X/f} (\Omega , {\mathcal Q}_2)$.
\end{proof}

By good luck we only ever propose, \eqref{ct16}, to work modulo the square of the maximal idea at Cremer points, so for exactly the same reason, 

\begin{bonus}\thlabel{cor:ct8}
Let ${\mathcal O}_{C_{cr}}$ be the part of the quotient of \eqref{ct17.bis} limiting on a cycle with return map a Cremer point, \thref{setup:ct1} (d), then the dimension of,
\begin{equation}
\label{ct69}
{\rm Ext}^1_{X/f} (\Omega_X , {\mathcal O}_{C_{cr}}) \, , \quad q = 0 \ {\rm or} \ 1
\end{equation}
is $1+\deg ({\rm Ram}_{\bullet})$, where ${\rm Ram}_{\bullet}$ is the part of ${\rm Ram}_b$ whose forward orbit includes the cycle.

All of which we may gather together in, 
\end{bonus}

\begin{csummary}\thlabel{sum:ct1}
For $\bullet$ a Herman ring, or Siegel disc, and ${\rm Ram}_{\bullet}$ as per \thref{cor:ct5} and \thref{cor:ct6} define,
\begin{equation}
\label{ct70}
\varepsilon_{\bullet} = \left\{\begin{matrix}
0 \, , &\mbox{if ${\rm Ram}_{t,{\bullet}}$ is empty} \\
1 \, , &\mbox{otherwise} \hfill \end{matrix} \right.
\end{equation}
then a convenient reformulation of the dimension estimate of op. cit. is,
\begin{equation}
\label{ct71}
1 + \deg ({\rm Ram}_{\bullet}) - \varepsilon_{\bullet} \, .
\end{equation}
Similarly for a parabolic cycle, i.e. $0+$ or $0-$ of \thref{defn:ct1}, let $e_{\bullet} r_{\bullet}$ be the order of tangency of the first return of the cycle to a rational rotation of order $r_{\bullet}$, then for $f : X \to X$ an endomorphism of a Riemann surface with finite ramification, the dimension of the maximal separated quotient,
\begin{equation}
\label{ct71.bis}
{\rm Ext}^1_{\widetilde X/f} (\rho^* \Omega , \widetilde {\mathcal Q}_{\rm big}) \twoheadrightarrow \overline{\rm Ext}^1 (\rho^* \Omega , \widetilde {\mathcal Q}_{\rm big})
\end{equation}
is, by the simple expedient of adding the estimates of \thref{cor:ct2}, \thref{cor:ct3}, \thref{cor:ct4}, \thref{cor:ct5}, \thref{cor:ct6}, \thref{cor:ct7}, and \thref{cor:ct8}, at least,
\begin{equation}
\label{ct72}
\deg ({\rm Ram}_f) + \sum_{\bullet \in {\rm par}} e_{\bullet}^2 + \sum_{\bullet \in {\rm HR \cup SD}} (1-\varepsilon_{\bullet}) + \# \, {\rm CR} - \# \, \{\mbox{wild tails}\} \, .
\end{equation}
\end{csummary}

%\end{document}

%%%%%%%%%%%%%%%

\vglue 1cm

\section{The Vanishing Theorem}\label{VT}

Our set up will follow \thref{setup:ct1} almost exactly with a view to computing the co-homology of the leftmost term in \eqref{ct18}. We should, however, for ease of exposition introduce,

\begin{furnot}\thlabel{not:v1}
Let $j : \widetilde U \hookrightarrow \widetilde X$ (or $\widetilde j$ if there is a danger of confusion) be as in \eqref{ct17} and $\widetilde R$ the divisor on $\widetilde U$ in the leftmost term of \eqref{ct18} so that, following \eqref{ct18.bis}, op. cit. is identically,
\begin{equation}
\label{v1}
\widetilde j_! \, {\mathcal O}_{\widetilde U} (-\widetilde R)
\end{equation}
while if $j : U \hookrightarrow X$ is the complement to $\rho_*$ of the cycle on the right of \eqref{ct17}, and $R$ the push forward of $\widetilde R$,
\nomenclature[N]{Push forward of the divisor $\widetilde R$}{\hyperref[not:v1]{$R$}}
\begin{equation}
\label{v2}
\rho_* \, \widetilde j_! \, {\mathcal O}_{\widetilde U} (-\widetilde R) = j_! \, {\mathcal O}_U (-R) \, ,
\end{equation}
albeit $\rho$ is an isomorphism around $\widetilde R$, so it's legitimate, and often convenient to confuse $R$ with $\widetilde R$. With this in mind we have a minor variant of Epstein's vanishing theorem, to wit:
\end{furnot}

\begin{thm}\thlabel{thm:v1}
Let $\rho : \widetilde X \to X = {\mathbb P}^1$ be the real blow up in the parabolic cycles of a rational map, $f$, of degree at least $2$, then, the image of the map of maximal separated quotients,
\begin{equation}
\label{v3}
\overline{\rm Ext}_{X/f}^2 (\Omega_X , j_! \, {\mathcal O}_U (-R)) \longrightarrow \overline{\rm Ext}^2_{\widetilde X/f} (\rho^* \Omega_X , \widetilde j_! \, {\mathcal O}_{\widetilde U} (-\widetilde R))
\end{equation}
is zero if $f$ is not a Latt\`es example, and dimension $1$ otherwise.

\noindent The map of \eqref{v3} comes from the Leray spectral sequence, and so fits into the larger diagram,
{\scriptsize
\begin{equation}
\label{v4}
\xymatrix{
{\rm Ext}_{X/f}^1 (\Omega_X, j_! {\mathcal O}_U (-R)) \ar[d] \ar[r] &{\rm Ext}_{\widetilde X/f}^1 (\rho^* \Omega_X, \widetilde j_! {\mathcal O}_U (-\widetilde R)) \ar[d] \ar[r] &{\rm Ext}_{X/f}^0 (\Omega_X, R^1 \rho_* \widetilde j_! {\mathcal O}_{\widetilde U} (-\widetilde R)) \ar[d] \\
{\rm Ext}_{X/f}^1 (\Omega_X,{\mathcal O}_X) \ar[d] \ar[r] &{\rm Ext}_{\widetilde X/f}^1 (\rho^* \Omega_X, {\mathcal O}_{\widetilde X}) \ar[d] \ar[r] &{\rm Ext}_{X/f}^0 (\Omega_X, R^1 \rho_* {\mathcal O}_{\widetilde X}) \ar[d]  \\
{\rm Ext}_{X/f}^1 (\Omega_X,{\mathcal Q}_{\rm big}) \ar[d] \ar[r] &{\rm Ext}_{\widetilde X/f}^1 (\rho^* \Omega_X, \widetilde{\mathcal Q}_{\rm big}) \ar[d] \ar[r] &{\rm Ext}_{X/f}^1 (\Omega_X,{\mathcal P}) \\
{\rm Ext}_{X/f}^2 (\Omega_X, j_! {\mathcal O}_U (-R)) \ar[r] &{\rm Ext}_{\widetilde X/f}^2 (\rho^* \Omega_X, \widetilde j_! {\mathcal O}_U (-\widetilde R))
}
\end{equation}
}
\noindent wherein ${\mathcal Q}_{\rm big}$, is the quotient of ${\mathcal O}_X$, by \eqref{v1}, cf. \eqref{ct18.bis}, while ${\mathcal P}$ is the sky scraper sheaf supported at parabolic cycles defined by either of the exact sequences,
\nomenclature[N]{As per $\widetilde{\mathcal Q}_{\rm big}$}{but without real blowing up \hyperref[v5]{${\mathcal Q}$}$_{\rm big}$} 
$$
0 \longrightarrow {\mathcal P} \longrightarrow R^1 \rho_* \, \widetilde j_! \, {\mathcal O}_{\widetilde X} (-\widetilde R) \longrightarrow R^1 \rho_* \, {\mathcal O}_{\widetilde X} \longrightarrow 0
$$
\begin{equation}
\label{v5}
0 \longrightarrow {\mathcal Q}_{\rm big} \longrightarrow \widetilde{\mathcal Q}_{\rm big} \longrightarrow {\mathcal P} \longrightarrow 0 \, .
\end{equation}
\nomenclature[N]{Sky scraper sheaf at parabolic cycles}{\hyperref[v5]{${\mathcal P}$}}
Notice, in particular, therefore that by \eqref{ct17.bis}, the second line in \eqref{v5} affords,
\begin{equation}
\label{v6}
0 \longrightarrow \prod_{\bullet} {\mathcal O}_{Z_{\bullet}} \longrightarrow \prod_{\bullet} {\mathcal O}_{\widetilde Z_{\bullet}} \longrightarrow {\mathcal P} \longrightarrow 0
\end{equation}
where the product is taken over parabolic cycles with the leftmost resp. middle term in \eqref{v6} the image in $X$, resp. the cycle post \eqref{ct13}.

Our principle preoccupation will be to understand the difference between the deformation space of $\widetilde X / f$ and $X/f$, beginning with,
\end{thm}

\begin{lem}\thlabel{lem:v1}
The map,
\begin{equation}
\label{v7}
{\rm Ext}^0_{X/f} (\Omega_X , R^1 \rho_* {\mathcal O}_{\widetilde X}) \longrightarrow {\rm Ext}_{X/f}^1 (\Omega_X , {\mathcal P})
\end{equation}
of the rightmost column of \eqref{v4} is zero.
\end{lem}

\begin{proof}
Both of the rightmost arguments in \eqref{v7} are skyscraper sheaves, so by \thref{cd:cor1.bis} we can suppse that $X$ is the germ of a parabolic fixed point. Better still by \thref{DSchol:1}, it will suffice to compute the image of \eqref{v7} modulo the image of flat functions in the right hand side of \eqref{v6}. To this end let,
\begin{equation}
\label{v8}
\tau = t(z) \, \frac\partial{\partial z}
\end{equation}
be a smooth vector field on $X$ whose asymptotic expansion is the formal invariant vector field $\partial$ of \eqref{ct26}, then by item (b) of \thref{cd:cor1} the left hand side of \eqref{v7} is the 1-dimensional complex vector space defined by the invariant, modulo the image of $\overline\partial$, class,
\begin{equation}
\label{v9}
\tau \otimes \frac{d \overline z}{\overline z} \, .
\end{equation}
On the other hand for a repelling petal $P$, cf. \eqref{Res7} et seq., there is a unique determination $\theta_P \in [0,2\pi)$ of the argument branched at the centre of the petal. Thus if the multiplier $\lambda$ is a primitive $r^{\rm th}$ root of unity, then the union of the branches of,
\nomenclature[N]{Argument branched at a petal}{\hyperref[v9.bis]{$\theta$}$_P$}
\begin{equation}
\label{v9.bis}
\theta := \frac1r \ \sum_{i=0}^{r-1} \theta_{(f^i)^* P}
\end{equation}
is invariant, so we have the convenient fact that for unit, $u \equiv 1 ({\rm mod} \, {\mathfrak m}^{re})$,
\begin{equation}
\label{v9.bis.bis}
f^* \theta - \theta = 1/2 \log u/\overline u
\end{equation}
while $\theta$ is smooth in a neighbourhood of the attracting petals. As such for,
\begin{equation}
\label{v10}
b_x : E \longrightarrow {\mathbb R}_{\geq 0} \, , \quad x \in \widetilde Z \subseteq E
\end{equation}
bump functions in a neighbourhood of the attracting directions, there is a flat section $\chi_0$ of $\Gamma (\widetilde X , {\mathcal A}^0_{\widetilde X} \otimes \rho^* T_X)$ such that,
\begin{equation}
\label{v11}
T := \frac12 \tau \otimes \frac{d \overline z}{\overline z} + \sum_{x \in \widetilde Z} \overline\partial (b_x (\theta) \theta \tau) - \overline\partial \chi_0 \in \Gamma (\widetilde X , j_! {\mathcal A}^0_{\widetilde U} \otimes \rho^* T_X) := 1/2 \tau \otimes \frac{d \overline z}{\overline z} + \overline\partial (T') \, .
\end{equation}
As ever, therefore, by \thref{cd:fact3} we require in the notation of op. cit. to calculate $d_1$ of \eqref{v11}, modulo $\overline\partial$ of forms supported off $\widetilde Z$. To this end observe,
\begin{enumerate}
\item[(a)] For the unit, $u \equiv 1 ({\rm mod} \, {\mathfrak m}^{re})$, of \eqref{v9.bis.bis} $e$ the order of tangency of $f$, to a rotation of order $r$,
\begin{equation}
\label{v12}
d_1 \left( \frac{dz}z \right) = \frac{du}u \, .
\end{equation}
\item[(b)] Modulo flat sections of $\Gamma (\widetilde X , {\mathcal A}_{\widetilde X}^{0,1} (\log E) \otimes \rho^* T_X)$,
\begin{equation}
\label{v13}
d_1 \left( \tau \otimes \frac{d\overline z}{\overline z} \right) = \overline\partial \, ((df)(\tau) \log \overline u)
\end{equation}
and in any case the left hand side is $\overline \partial A$ for some smooth function $A$.
\item[(c)] $d_1 (b_x)$ has support off $\widetilde Z$, so, again modulo flat sections,
\begin{equation}
\label{v14}
d_1 \overline\partial (b_x \, \theta \tau) = \overline\partial (d_1 (b_x) f^* \theta (df)(\tau)) + \overline\partial (b_x \, d_1 (\theta) (df)(\tau)) \, .
\end{equation}
\end{enumerate}
As such by (b)/general nonsense,  $d_1(T)$ is $\overline\partial h$ for a smooth function $h$ while the former is zero off $\widetilde Z$, thus $h$ is holomorphic close to $\widetilde Z$, so by (b), (c), and \eqref{v9.bis},
\begin{equation}
\label{v15}
h = \frac12 (\log u)(df)(\tau) \mod \mbox{(flat)}
\end{equation}
close to $\widetilde Z$. However by \eqref{v5}--\eqref{v6} it's exactly the class of \eqref{v15} in the right hand side of \eqref{v7} that we need to calculate, and this is zero because it's diagonal, i.e. in the kernel of \eqref{v6}, modulo flat.
\end{proof}

The key to proceeding further is to dualise \eqref{v3}, which entails a little more,

\begin{Notation}\thlabel{not:v2}
For ${\mathcal G}$ an almost $f$-sheaf, and an $f$-sheaf ${\mathcal F}$, let ${\mathbb E}{\rm xt}^{-m}_{X/f}$ ($m = p+q-n)$ 
\nomenclature[N]{Abutement of dual spectral sequence}{\hyperref[not:v2]{${\mathbb E}{\rm xt}$}$^{-m}_{X/f}$} 
be the abutment of the spectral sequence \eqref{cd24}, while if ${\mathcal G}$ is a locally free $f$-sheaf define,
\begin{equation}
\label{v16}
{\mathbb H}_{-m} (X/f , {\mathcal G}^{\vee} \otimes f^* {\mathcal F} \otimes \omega_X) := {\mathbb E}{\rm xt}^{-m}_{X/f} ({\mathcal G} , f^* {\mathcal F} \otimes \omega_X)
\end{equation}
\nomenclature[N]{Alternative notation to ${\mathbb E}{\rm xt}^{-m}_{X/f}$ for}{locally free sheaves \hyperref[not:v2]{${\mathbb H}$}$_{-m}(X/f, \ )$} 
wherein $X$, as per. op. cit., may, exceptionally, be a complex manifold or the real blow of a Riemann surface. In the particular case, therefore, of \eqref{v3}, which in the notation of \eqref{cd24bis} et seq. is, equally:
\begin{equation}
\label{v17}
\overline{\rm Ext}_2^{U/f} (\Omega_U , {\mathcal O}_U (-R)) \longrightarrow \overline{\rm Ext}_2^{\widetilde U/f} (\rho^* \Omega_U , {\mathcal O}_{\widetilde U} (-\widetilde R))
\end{equation}
it's dual, in the notation of \eqref{v16}, is,
\begin{equation}
\label{v18}
{\mathbb H}_0 (U/f , f^* \Omega (R) \otimes \omega_U) \longleftarrow {\mathbb H}_0 (\widetilde U / f , f^* \Omega (\widetilde R) \otimes \omega_{\widetilde U})
\end{equation}
where, for ease of notation, we drop $\rho^*$ when the context is clear. Irrespectively if we further define,
\begin{equation}
\label{v19}
{\mathbb H}_0 (\widetilde U / f , f^* \Omega (R) \otimes {\rm Tors} (\omega_{\widetilde U})) = {\rm Ker} \{ d_1^{\vee} \mid {\rm H}^0 (\widetilde U / f , f^* \Omega (R) \otimes \omega_{\widetilde U})\}
\end{equation}
for $d_1^{\vee}$ as in \eqref{cd24bis}, with the rightmost group in \eqref{v19} included in that of \eqref{v18} by way of \eqref{Dc59} then we have, by \eqref{Dc59}, and in the notation thereof, an exact sequence,
\begin{eqnarray}
\label{v20}
0 \to {\mathbb H}_0 (\widetilde U/f , f^* \Omega (R) \otimes {\rm Tors} (\omega_{\widetilde U})) \to {\mathbb H}_0 (\widetilde U/f , f^* \Omega (R) \otimes \omega_{\widetilde U}) \to \nonumber \\
\to {\mathbb H}_0 (\widetilde U/f , f^* \Omega (R) \otimes \varinjlim_{n \geq 0} \widetilde\omega (nE)) \hookrightarrow {\mathbb H}_0 (U/f , f^* \Omega (R) \otimes \omega_U) \, .
\end{eqnarray}
As such if we denote the image of the rightmost group in \eqref{v18} in the leftmost by ${\rm Obs}$ then the content of \thref{thm:v1} is, by \eqref{v20}, equivalently, either of:
\nomenclature[N]{Obstruction group mod torsion}{\hyperref[v21]{${\rm Obs}$}} 
\begin{equation}
\label{v21}
\begin{matrix}
(1) &{\rm Obs} = 0 \hfill \\
(2) &{\mathbb H}_0 (\widetilde U / f , f^* \Omega (R) \otimes \omega_{\widetilde U}) \ \mbox{is torison.} \hfill
\end{matrix}
\end{equation}
Irrespectively the obstruction group, ${\rm Obs}$, if it exists is far from arbitrary since inter alia it follows from \eqref{v4} that the composition,
\begin{equation}
\label{v22}
{\rm Obs} \hookrightarrow {\mathbb H}_0 (U/f , f^* \Omega (R) \otimes \omega_U) \to {\rm Ext}^0_{X/f} (\Omega , R^1 \rho_* \, \widetilde j_! \, {\mathcal O}_{\widetilde U} (-R))^{\vee}
\end{equation}
is zero.
\end{Notation}

With this in mind we have,

\begin{lem}\thlabel{lem:v2}
Let $\omega$ be a quadratic differential in ${\rm Obs}$, and $\widetilde Z \hookrightarrow \widetilde X \overset{\rho}{\longrightarrow} X$ lying over a cycle $Z$ whose first return map is tangent to order $e$ to a rational rotation of order $r$, then $\omega$ satisfies the conditions (for $f^r$) of \thref{Dcon:1}, \thref{Dhyp:1} and \thref{Dclaim:1}, so by op. cit., for every repelling petal, $P$, there is a constant $\lambda_p$ such that \eqref{Res37} holds, which is, in fact $0$.
\end{lem}

\begin{proof}
That $\omega$ satisfies the hypothesis in question is a mixture of the definition of ${\mathbb H}_0$, and the structure, \eqref{Dc59} or \eqref{v20} of $\omega_{\widetilde X}$ itself, modulo torsion, together with the fact that $R$ is concentrated in attracting petals.

Otherwise consider the calculation of the connecting homomorphism,
\begin{equation}
\label{v23}
{\rm Ext}_{\widetilde X / f}^0 (\Omega_X , R^1 \rho_* \, \widetilde j_! \, {\mathcal O}_{\widetilde U} (-\widetilde R)) \xrightarrow{ \ \delta \ } {\rm Ext}^2_{X/f} (\Omega_X , j_! \, {\mathcal O}_U (-R))
\end{equation}
by way of the double complex associated to \eqref{cd13}. Thus quite generally for $\Delta \hookleftarrow Z$ a small neighbourhood, we start with a $D = d_1 - \overline\partial$ closed class in,
\begin{eqnarray}
\label{v24}
C &:= &A \textstyle \coprod B \in \Gamma (\widetilde\Delta , {\mathcal A}_{\widetilde\Delta}^{0,1} (\log E) \otimes \rho^* T_X \otimes \widetilde j_! {\mathcal O}_{\widetilde U} (-\widetilde R)) \nonumber \\
&&\textstyle\coprod \Gamma (\widetilde\Delta , {\mathcal A}_{\widetilde\Delta}^{0,0} \otimes \rho^* f^* T_X \otimes \widetilde j_! {\mathcal O}_{\widetilde U} (-\widetilde R))
\end{eqnarray}
and choose a bump function, $b_Z$, about $Z$, so that \eqref{v23} is just,
\begin{equation}
\label{v25}
\delta : [C] \longmapsto [D(b_Z \, C)] \, .
\end{equation}
In particular not just for $\omega$ in the obstruction group, but even in the rightmost group of \eqref{v20}, $\omega (D(b_Z \, C))$ is absolutely integral, and the dual pairing, satisfies,
\begin{equation}
\label{v26}
\omega (\delta (0 {\textstyle \coprod} B)) = - \lim_{\varepsilon \to 0} \ \int_{\partial \Delta_Z (\varepsilon)} \omega (B)
\end{equation}
for $\Delta_Z (\varepsilon)$ a small disc of radius $\varepsilon$ around the cycle. Similarly,
\begin{equation}
\label{v27}
\omega (\delta (A {\textstyle \coprod} 0)) = \lim_{\varepsilon \to 0} \ \int_{X \backslash \Delta_Z (\varepsilon)} \omega (d_1 (\rho A)) \, .
\end{equation}
Now if $A$ were equal to $\overline \partial a$, then $\rho \, \overline\partial (a) - \overline\partial (\rho a)$ has no support close to $Z$ and $\omega$ is $d_1^{\vee}$ closed so, from $\overline\partial d_1 = d_1 \overline\partial$, \eqref{v27} is equally,
\begin{equation}
\label{v28}
\lim_{\varepsilon \to 0} \ \int_{\partial \Delta_Z (\varepsilon)} \omega (d_1 a)
\end{equation}
provided that the introduction of $a$ hasn't created any residues, or branches. In applying this to the case in point we can take $A=T$ of \eqref{v11}, so that, in the notation of op. cit.,
\begin{equation}
\label{v29}
B = d_1 (T') + \frac12 (df) (\tau) \log \pi / u \, .
\end{equation}
Similarly if we fix a repelling petal $P$, then we have a determination $\log_P z$ of a branch of the logarithm for $\theta_P$ as immediately before \eqref{v9.bis}, and,
\begin{equation}
\label{v30}
L(z) := \frac1r \ \sum_{i=0}^{r-1} \log_{(f^i)^* P} z
\end{equation}
affords a convenient candidate for $a$ of \eqref{v28}, i.e.
\begin{equation}
\label{v31}
a = \frac12 (\overline{L(z)} - L(z)) + T'
\end{equation}
since by \eqref{v9.bis}: \eqref{v26} and \eqref{v28} cancel, nor are these residues. This is, however, at the cost of introducing various branches $(f^i)(b)$ starting from the initial branch $b$ of $\theta_P$, so the pairing of $\omega$ with the class of \eqref{v11} et seq is for $b_{\varepsilon} = b \cap \Delta_Z (\varepsilon)$,
\begin{equation}
\label{v32}
- \frac{2\pi i}r \ \sum_{j=0}^{r-1} \ \int_{f^{j+1} (b_{\varepsilon}) \backslash f^j (b_{\varepsilon})} (f_* \, \omega)(\tau)
\end{equation}
which really does depend on the choice of a square root of $-1$. There is, however, no contradiction in this since the lifting of the class of \eqref{v8} to one in the left hand side of \eqref{v23} is in fact ambiguous up to a class in ${\rm Ext}^0_{X/f} (\Omega,{\mathcal P})$, e.g. up to some locally constant functions in the middle group of \eqref{v6}. Better still by the definition of ${\mathbb H}_0$, $\omega$ satisfies the estimate \eqref{Res35}, albeit with $re$ rather than $e$ in the current notation, so for $s$ the Fatou coordinate in any petal in the orbit of $P$,
\begin{equation}
\label{v33}
\omega (\tau) \sim \lambda_p \, ds
\end{equation}
with the asymptotics of \eqref{Res37}. Consequently the limit as $\varepsilon \to 0$ of \eqref{v32} is,
\begin{equation}
\label{v34}
-2\pi i \, \lambda_P \, .
\end{equation}
As such, by \eqref{v22}, if in addition $\omega$ belonged to ${\rm Obs}$ then $\lambda_P = 0$, where, of course, the implied ambiguities in \eqref{v34}, i.e. not just a choice of $\sqrt{-1}$ but also of $P$, all come from locally constant functions in ${\rm Ext}_{X/f}^0 (\rho^* \Omega , {\mathcal P})$, and work to our advantage since a priori the $\lambda_P$ were unrelated as the petal varies.
\end{proof}

Following the proof of \cite[Lemma 4]{adam} we can now proceed to give,

\begin{proof}[Proof of \thref{thm:v1}] 
The map of \eqref{Res23} is a non-linear natural transforms of $m$-forms to measures. On the other hand for $f : \widetilde X \to \widetilde Y$ proper, the dual of the pull-back of functions with compact support is push-forward of measures, which in the \thref{not:v2} is related to the maps of \thref{rmk:cd2} by way of,
\begin{equation}
\label{v35}
\xymatrix{
f^* \Omega_Y^{\otimes m-1} \otimes \omega_{\widetilde X} \ar[d]_{f_*} \ar[rr]^-{(df)^{\otimes m-1}} && \omega^m_{\widetilde X} \ar[r]^-{\Vert \ \Vert} &{\rm Measure}_{\widetilde X} \ar[d]^{f_*} \\
\Omega_Y^{\otimes (m-1)} \otimes \omega_{\widetilde Y} \ar@{=}[rr] &&\omega_{\widetilde Y}^m \ar[r]^-{\Vert \ \Vert} &{\rm Measure}_{\widetilde Y}
}
\end{equation}

$$
D_f (\omega) := f_* \Vert (df)^{\vee} (\omega) \Vert - \Vert f_* \, \omega \Vert \geq 0
$$
with equality iff $\omega$ is a pull-back of a section of the sheaf in the bottom left hand corner. Now in our case $\Vert \omega \Vert$ is absolutely integrable off any small neighbourhood, $N_{\varepsilon}$, around the cycles of \thref{not:ct1}, so:
\begin{equation}
\begin{split}
\label{v36}
\int_{\widetilde X \backslash N_{\varepsilon}} D_f(\omega) \leq &\int_{N_{\varepsilon} \backslash f^{-1} (N_{\varepsilon})} \Vert (df)^{\vee} (\omega)\Vert - \int_{f^{-1} (N_{\varepsilon}) \backslash N_{\varepsilon}} \Vert (df)^{\vee} (\omega) \Vert \\
\to & \ {\rm Res}_f \mid \Vert (df)^{\vee} (\omega)\Vert
\end{split}
\end{equation}
where the residue is Epstein's dynamical residue of \thref{DRev:1} for $\infty$ of op. cit. the non-repelling cycles and any circles in \thref{setup:ct1} (b). It is, however, essentially trivial not only that the dynamical residue is the sum over those calculated one cycle or circle at a time but that for anything other than a parabolic cycle or a Cremer point, there are arbitrarily small neighbourhoods about the same such that $f(V) \subseteq V$. As such the only cases we need to calculate are Cremer points, where it's zero by \eqref{Res4}, or parabolic points, where it's non-positive for $\omega \in {\rm Obs}$ by \thref{lem:v2} and \thref{Dc:ClaimExtra3}. Thus by the connectivity of $U$, \eqref{v35} and \eqref{v36}, any $\omega \in {\rm Obs}$ is of the form,
\begin{equation}
\label{v37}
\omega = f^* q = \deg (f) \cdot q \, .
\end{equation}
Now, for $\deg(f) > 1$, this is only possible if $f$ is a Latt\`es example, \cite[1]{DH}, i.e. multiplication by an integer on an elliptic curve, $E$ modulo $\pm 1$, 
\nomenclature[D]{Latt\`es \hyperref[v38]{example}}{} 
\begin{equation}
\label{v38}
\xymatrix{
E \ar[d]_n \ar[r]^-{\varepsilon} &E/\pm = {\mathbb P}' \ar[d]^f \\
E \ar[r]_{\varepsilon} &{\mathbb P}'
}
\end{equation}
In particular the branch locus of $\varepsilon$ is certainly contained in $R$, so the invariant differentials for $f$, are exactly the holomorphic differential on $E$ which are invariant by $n$, i.e. in this case all of the groups in \eqref{v3}, wherein ${\rm Ext}^2$ is already its own separated quotient, have dimension $1$.
\end{proof}

In the disconnected case, we still know that $D_f \, \omega$ of \eqref{v35} is zero, but we cannot conclude to \eqref{v37}. This case, however, only occurs if we have extracted in item (b) of \thref{setup:ct1} a closed annuli or circle, $C_i$, from some non empty set of Herman rings and/or annuli. Now if $C$ is the union of the $C_i$ and $U_{\alpha}$, resp. $V_{\beta}$ the connected components of $X \backslash C$, resp. $X \backslash f^{-1} (C)$ then whenever $f$ maps $V_{\beta}$ to $U_{\alpha}$ it is proper flat and surjective. As such the conclusion to be drawn from $D_f \, \omega = 0$ is: $\exists \, q_{\alpha} \in {\rm H}^0 (U_{\alpha} , \Omega (R) \otimes \widetilde \omega)$ such that if $f : V_{\beta} \to U_{\alpha}$ then $\exists \, C_{\alpha\beta} \in {\mathbb R}_{\geq 0}$ for which 
\begin{equation}
\label{v40}
(df)^{\vee} \omega \mid_{V_{\beta}} = C_{\alpha\beta} f^* q_{\mid\alpha} \, .
\end{equation}
In particular if we fix a curve $C_{\alpha}$ in the boundary of a given $U_{\alpha}$, then it determines a unique $V_{\gamma}$ which contains the same piece of the Siegel disc or Herman ring close to $C$ as $U_{\alpha}$. In particular if $\omega \mid_{U_{\alpha}} \ne 0$ then,
\begin{equation}
\label{v41}
0 \ne (df)^{\vee} \omega \mid_{V_{\alpha}} = (f^* q_{\alpha}) \, C_{\alpha \gamma}
\end{equation}
so without loss of generality $C_{\alpha \gamma} = 1$, and, 
\begin{equation}
\label{v42}
f^* q_{\alpha} = (f_* \, \omega) \mid_{U_{\alpha}} = \left( \sum_{\beta} C_{\alpha\beta} \right) q_{\alpha} \, .
\end{equation}
However close to $C_{\alpha}$ any eigen tensor has modulus 1, so,
\begin{equation}
\label{v43}
C_{\alpha\beta} = 0 \ \mbox{for} \ \beta \ne \gamma \, , \quad \omega_{U_{\alpha}} = q_{\alpha} \, , \quad f^* \omega_{\alpha} = \omega_{\alpha} \, .
\end{equation}
By construction, however, $U_{\alpha}$ contains a boundary component of the Siegel disc or Herman ring and we assert,

\begin{claim}\thlabel{claim:v1}
If $f:D \to D$ is a Siegel disc or Herman ring for an endomorphism of a Riemann surface then there is no open neighbourhood $U$ of a boundary component, $\partial$, of $D$ such that there exists $m \in {\mathbb N}$ for which the invariant differential $\left( \frac{dz}z \right)^{\otimes m}$ extends regularly from $D$ to $U$.
\end{claim}

\begin{proof}[Sub-proof of \thref{claim:v1}]
Suppose there were such an extension, $q$, say, and let $\pi : \widetilde U \to U$ be a connected component of $\sqrt[m]{q}$ with say, $\pi^* q=Q^m$, and $\widetilde\partial = \pi^{-1}(\partial)$. Now fix $x \in \widetilde\partial$, with $d\pi(x) \ne 0$ and define for $y$ close to $x$,
\begin{equation}
\label{v44}
\widetilde L (y) = \int_x^y {\rm Re} (Q) \, .
\end{equation}
Equally, however, $\log \vert Z \vert^2$ is a constant on the boundary, i.e. for a fixed base point $*$ in $\pi^{-1} (D) \cap \widetilde U$ there is a constant, $C$, such that
\begin{equation}
\label{v45}
L(y_n) := \int_*^{y_n} {\rm Re}(Q) + C \longrightarrow 0 \, , \quad \pi^{-1} (D) \ni y_n \to y \in \partial \, .
\end{equation}
Consequently if we choose sequences $x_n , y_n$ in $\pi^{-1} (D)$ converging to $x$ and $y$ respectively, then,
\begin{equation}
\label{v46}
\widetilde L (y) = \lim_n \widetilde L (x_n) + \int_{x_n}^{y_n} {\rm Re} (Q) = \lim_n (L(y_n) - L(x_n)) = 0 \, .
\end{equation}
Similarly if $y \in \pi^{-1} (D)$ then,
\begin{equation}
\label{v47}
L(y) = L(x_n) + \int_{x_n}^y {\rm Re} (Q) \longrightarrow \widetilde L(y) \, , \quad n \to \infty
\end{equation}
so close to $x , \pi^{-1} (D)$ is exactly $\widetilde L < 0$. Further, either a branch point or a singularity of the boundary lies over a zero of $q$, which by \eqref{v43} defines a finite $f$-invariant cycle. However, apart from notations, no germ of a fixed point admits a regular $m$-differential, so $\pi$ is everywhere \'etale around $\partial$, and $\partial$ is smooth. This shows, however, that $\sqrt[m]{q}$ has no monodromy around $\partial$, so, without loss of generality $m=1$, and $L_{<\varepsilon}$, for some $\varepsilon > 0$, is an invariant neighbourhood on which,
\begin{equation}
\label{v48}
z \longmapsto \exp \left( \int_*^z Q \right)
\end{equation}
extends $D$ to a larger annulus or disc over which $f$ is conjugate to a rotation, contrary to the definition of a Herman ring or Siegel disc. It therefore follows that $\omega \mid_{U_{\alpha}}$ is zero.
\end{proof}

This also covers what happens with higher order differentials, i.e.

\begin{var}\thlabel{var:v2}
Let everything be as in \thref{thm:v1} then for any $m \geq 2$, the image of,
\begin{equation}
\label{v49}
{\mathbb H}_0 (\widetilde U / f , f^* \Omega^{\otimes m} (mR) \otimes \omega_{\widetilde U}) \longrightarrow {\mathbb H}_0 ( U / f , f^* \Omega^{\otimes m} (mR) \otimes \omega_{U})
\end{equation}
is zero, i.e. there is no exceptional Latt\`es case.
\end{var}

\begin{proof}
We've already done the disconnected case in the proof of \thref{thm:v1} i.e. apply \thref{claim:v1} as in the end of the proof of \thref{thm:v1}. Otherwise the proof of \thref{thm:v1} goes through up to and including \eqref{v37}. Now $m$-forms which are eigenforms for $f^*$ either come from descending invariant differentials on elliptic curves modulo a finite group $G$ with $f$ the quotient of an endomorphism $F$ on the same, or ${\mathbb G}_m$. As such the possible cases, \cite{}[DH], are $G = \pm \mu_2 , \mu_3, \mu_4$ or $\mu_6$ with $F$ either an integer, $a$, resp. multiplication by $\alpha$ in some quadratic ring, so by \eqref{v37} $a^{m-2} = 1$ or $\alpha^{m-1} = \overline\alpha$, both of which are impossible for $\deg (f) > 1$. Similarly only $m=0$ is possible for ${\mathbb G}_m$.
\end{proof}

Equally while the choice of $R$ in \eqref{not:v1} is the good one for the Fatou-Shishikura inequality we've used no particular property of it except to guarantee the finiteness of either term on the right of \eqref{v36}, thus we have,

\begin{var}\thlabel{var:v3}
Let everything be as in \thref{thm:v1}, resp. \thref{var:v2}, and $D$ a reduced almost $f$-divisor on $\widetilde U$ (i.e. a pair of reduced divisors $D_1 , D_0$ on $\widetilde U$ such that $f(D_1) \subset D_0$, so, infinite is allowed provided it accumulates in $\widetilde X \backslash \widetilde U$) disjoint from $R$ then all of \eqref{v3}, resp. \eqref{v49} hold with $R+D$ rather than just $R$.
\nomenclature[D]{Reduced \hyperref[var:v3]{almost $f$ divisor}}{}
\end{var}

In order to describe the consequences of the vanishing theorem we first need to describe some redundancy occasioned by real blowing up by way of,

\begin{rmk}\thlabel{rmk:v3}
Let $\rho : \widetilde\Delta \to \Delta$ be a real blow up of a small disc about a parabolic fixed point, with $\widetilde Z \subset E = \rho^{-1} (0)$ the complementary cycle of \eqref{ct17} then if ${\mathcal O}_{\widetilde z}$ is the germ of ${\mathcal O}_{\widetilde\Delta}$ about some choice of $\widetilde z \in E$ there is a unique diagonal embedding, $\widetilde \delta_{\widetilde z}$, rendering commutative the diagram,
\begin{equation}
\label{v54}
\xymatrix{
{\mathcal O}_{\Delta,0} \ar[dd] \ar[rd]^{\rho^*} \\
&\underset{x \in \widetilde z}{\prod} {\mathcal O}_{\widetilde\Delta , x} \\
\rho_* {\mathcal O}_{\widetilde z} \ar[ru]_{\widetilde\delta_{\widetilde z}}
}
\end{equation}

As such if $\widetilde {\mathcal Q}_{\rm big}$ is the total quotient in \eqref{ct18} then we have a fiber square of rings,
\nomenclature[N]{Diagonal subring of $\widetilde {\mathcal Q}_{\rm big}$}{\hyperref[v54]{$Q$}$_{\rm diag}$} 
\begin{equation}
\label{v52}
\xymatrix{
{\mathcal Q}_{\rm diag} \ar@{^{(}->}[d] \ar[r] &(i_{\widetilde z})_* (\widetilde\delta_{\widetilde z} \, \rho_* \, {\mathcal O}_{\widetilde z})_{ \ _{ \ } } \ar@{^{(}->}[d] \\
\widetilde{\mathcal Q}_{\rm big} \ar[r] &(i_{\widetilde z})_* \left( \underset{x \in \widetilde z}{\prod} {\mathcal O}_{\widetilde\Delta , x} \right)
}
\end{equation}
Further \eqref{v54} is a retract, and \thref{cd:cor1.bis} applies, so for any $m,q \in {\mathbb Z}_{\geq 0}$ we get retracted subspaces,
\begin{equation}
\label{v55}
{\rm Ext}^q_{\widetilde X / f} (\rho^* \Omega_X^{\otimes m} , {\mathcal Q}_{\rm diag}) \hookrightarrow {\rm Ext}^q_{\widetilde X / f} (\rho^* \Omega_X^{\otimes m} , \widetilde{\mathcal Q}_{\rm big}) 
\end{equation}
which are independent of the choice of $\widetilde z$. Indeed they're the subgroup annihilated by,
\begin{equation}
\label{v56}
{\mathbb H}_0 (\widetilde U , \rho^* \Omega_X^{\otimes m} (R) \otimes {\rm Tors} \, \omega_{\widetilde X})
\end{equation}
on even $R+D$ if we add an almost $f$-divisor to the construction as per \thref{var:v3}.
\end{rmk}

Putting all of this together we have,

\begin{thm}\thlabel{thm:v2}
Let everything be as in \thref{not:v1}, with $D$ a reduced almost $f$-divisor in $\widetilde U$ disjoint from $R$, and $Q^{\rm diag}$ as in \eqref{v52}, or, more generally, ${\mathcal Q}^{\rm diag} (D)$ if we replace ${\mathcal Q}_{\rm big}$ by the quotient of ${\mathcal O}_{\widetilde X}$ by $\widetilde j_! {\mathcal O}_{\widetilde X} (-R-D)$, then the map,
\begin{equation}
\label{v55.bis}
{\rm Ext}^1_{\widetilde X / f} (\rho^* \Omega^{\otimes m} , {\mathcal O}_{\widetilde X}) \longrightarrow \overline{\rm Ext}^1_{\widetilde X / f} (\rho^* \Omega_X^{\otimes m} , {\mathcal Q}_{\rm diag}(D)) 
\end{equation}
to the maximal separated quotient of the implied ${\rm Ext}^1$ is surjective, unless $m=1$ and $f$ is a Latt\`es example, in which case it has codimension $1$.
\end{thm}

\begin{proof}
By \thref{cd:fact7}, and \eqref{v19} et seq. we have an exact sequence,
\begin{eqnarray}
\label{v56.bis}
{\rm Ext}^1_{\widetilde X / f} (\rho^* \Omega_X^{\otimes m} , {\mathcal O}_{\widetilde X}) &\to &\overline{\rm Ext}^1 (\rho^* \Omega_X^{\otimes m} , {\mathcal Q}_{\rm diag}(D)) \nonumber \\ 
&\to &{\mathbb H}_0 (\widetilde U / f , \rho^* \Omega^{\otimes m} (+R+D) \otimes \omega_{\widetilde X})^{\vee} \, .
\end{eqnarray}
Equally by \eqref{v20} we have an exact sequence,
\begin{eqnarray}
\label{v57}
0 &\to &{\mathbb H}_0 (\widetilde U / f , \rho^* \Omega^{\otimes m} (R+D) \otimes {\rm Tors} \,  \omega_{\widetilde X}) \nonumber \\
&\to &{\mathbb H}_0 (\widetilde U / f , \rho^* \Omega^{\otimes m} (R+D) \otimes {\mathcal O}_{\widetilde X}) \to {\rm Obs} \to 0
\end{eqnarray}
and, irrespectively of $D$, we know by \thref{thm:v1} and \thref{var:v3} that the obstruction group is $0$. Consequently, the image of the leftmost term in \eqref{v56} is the subgroup of the middle group annihilated by the torsion, which is exactly the right hand side of \eqref{v55} by \thref{rmk:v3}.
\end{proof}

In general it is false that the smaller subgroup,
\begin{equation}
\label{v58}
{\rm Ext}_{X/f}^1 (\Omega^{\otimes m} , {\mathcal O}_X) \hookrightarrow {\rm Ext}_{\widetilde X/f}^1 (\rho^* \Omega^{\otimes m} , {\mathcal O}_{\widetilde X})
\end{equation}
is surjective onto the right hand side of \eqref{v55}, e.g. quadratic polynomials with a parabolic fixed point, and $m=1$, have dimension 2 on the left of \eqref{v58}, but, \thref{cor:ct2} and \thref{cor:ct4}, dimension 3 on the right of \eqref{v55}. There, are, however circumstances where this may be achieved, i.e.

\begin{altsetup}\thlabel{altsetup:v1}
Following \cite[\S1]{adam}, let $\sigma : \overline X \to X$ be the real blow up of ${\mathbb P}^1_{\mathbb C}$ in the set of parabolic cycles of  a rational map $f$ which are parabolic repelling, i.e. case $0-$ of \thref{defn:ct1}.

Now introduce a variant of the decomposition of the ramification in \eqref{ct3}--\eqref{ct4}, to wit:
\nomenclature[N]{Ramification, wild,}{without real blowing up at non repelling parabolics \hyperref[v59]{${\rm Ram}$}$_{\overline w}$} 
\nomenclature[N]{Ramification, tame,}{without real blowing up at non-repelling parabolics \hyperref[v59]{${\rm Ram}$}$_{\overline t}$} 
\begin{equation}
\label{v59}
{\rm Ram}_{\overline w} = {\rm Ram}_w + {\rm Ram}_t^{0,+} \, , \ {\rm Ram}_{\overline t} = {\rm Ram}_t  - {\rm Ram}_t^{0,+} \, .
\end{equation}
Thus following \eqref{ct9}--\eqref{ct11} we have for any $n > 0$, divisors,
\nomenclature[N]{Ramification, finite approximation to}{${\rm Ram}_{\overline w}$ \hyperref[v60]{${\rm Ram}$}$_{\overline w}^{n}$} 
\nomenclature[N]{Ramification, finite approximation to}{${\rm Ram}_{\overline t}$ \hyperref[v59]{${\rm Ram}$}$_{\overline t}^{n}$} 
\begin{equation}
\label{v60}
{\rm Ram}_{\overline w}^n \, , \ {\rm Ram}_{\overline t}^n \, , \ \overline R_{\overline t} = \varinjlim_n {\rm Ram}_{\overline t}^n
\end{equation}
on $\overline X$, i.e. we truncate (at ``$n$''), not just the wild ramification, but any tame ramification accumulating at non-repelling parabolic cycles. Similarly rather than \eqref{ct18}, we have an exact sequence,
\begin{equation}
\label{v61}
0 \to \widetilde j_! {\mathcal O}_{\widetilde U} (-R_b - R_{\overline t} - {\rm Ram}_{\overline w}^n -C_{cr}) \to {\mathcal O}_{\widetilde X} \to {\mathcal O}_{C_{cr}} {\textstyle \prod} {\mathcal O}_{{\rm Ram}_{\overline w}^n} {\textstyle \prod} {\mathcal O}_{R_b} \underset{\bullet}{\textstyle \prod} {\mathcal Q}_{\bullet} \to 0
\end{equation}
wherein, following the short hand of \eqref{v1} and \eqref{v52}, we denote the kernel, resp. co-kernel, in \eqref{v61} by,
\nomenclature[N]{Functions on all cycles,}{tame ramification, and finite approximation to wild, no real blow up at non-repelling parabolics \hyperref[v62]{${\mathcal Q}$}$_{\overline {\rm big}}$} 
\nomenclature[N]{Weighted divisor arising from tame ramification,}{finite approximation to wild, and Cremer points, no real blow up at non-repelling parabolics \hyperref[v62]{$\overline R$}} 
\begin{equation}
\label{v62}
\widetilde j_! \, {\mathcal O}_{\widetilde U} (-\overline R) \, , \qquad \mbox{resp.} \ {\mathcal Q}_{\overline{\rm big}} \, .
\end{equation}
Finally, therefore, following \thref{rmk:v3}, and \thref{thm:v2} we have sheaves, 
\nomenclature[N]{Diagonal subring of ${\mathcal Q}_{\overline{\rm big}}$}{\hyperref[v63]{$Q$}$_{\overline {\rm diag}}$}
\begin{equation}
\label{v63}
{\mathcal Q}_{\overline{\rm diag}} \, , \ {\mathcal Q}_{\overline{\rm diag}} (\overline D)
\end{equation}
for $\overline D$ the restriction to $\widetilde X$ of a reduced almost $f$-divisor on $\widetilde U$ without accumulation points on non-repelling parabolics. In particular, we do not change the inclusion $\widetilde j$ of \eqref{ct17} and \thref{not:v1}, but only do what is necessary, to change the divisor $\widetilde R$ of \eqref{v1} on $\widetilde U$ to a divisor which extends to $\overline X$ about the non-repelling parabolic cycles, i.e. $\overline R$ of \eqref{v62}.
\end{altsetup}

Having set up the notation, we arrive to, 

\begin{var}\thlabel{var:v4}
Let everything be as in \thref{altsetup:v1} then for any $m > 0$ the map,
\begin{equation}
\label{v64}
{\rm Ext}^1_{\overline X / f} (\sigma^* \Omega^{\otimes m} , {\mathcal O}_{\overline X}) \longrightarrow \overline{\rm Ext}^1_{\widetilde X / f} (\rho^* \Omega^{\otimes m}_X , {\mathcal Q}_{\overline{\rm diag}} (\overline D))
\end{equation}
to the maximal separated quotient of the implied ${\rm Ext}^1$ is surjective, unless $m=1$ and $f$ is a Latt\`es example, in which case it has codimension 1.
\end{var}

\begin{proof}
By \thref{thm:v2} it will suffice to prove that the left hand sides of \eqref{v55} and \eqref{v64} have the same image in the right hand side of \eqref{v64} since the latter is a quotient of the right hand side of \eqref{v55}. As such the relevant diagram is \eqref{v4} albeit with $\overline X$, $\overline U$, etc. in the rightmost column. In particular for each non repelling parabolic cycle we have by \thref{lem:ct1} a unique class in the 2$^{\rm nd}$ row of the rightmost column, which by op. cit. and \thref{lem:v1} can be lifted to a class in the penultimate row of the leftmost column, whose value in the bottom leftmost entry, after passing to maximal separated quotients by \thref{cd:fact7}, is the a priori obstruction to deducing \thref{var:v4}.

On the other hand by \thref{cor:ct1} for any $N \gg 0$ (in fact $m(re+2)$ in the notation of op. cit.) if $\overline Z$ is the set of attracting directions at the non-repelling parabolics, $\overline\rho : \widetilde X \to \overline X$, we have open inclusions,
\begin{equation}
\label{v65}
\begin{matrix}
\overline j &: &\overline V := \overline U \cup \overline \rho (\overline Z) \longhookrightarrow \overline X \\
\widetilde{\bar j} &: &\widetilde V = \widetilde U \cup \overline Z \longhookrightarrow \widetilde X \hfill
\end{matrix}
\end{equation}
affording, an inclusion of the kernel in \eqref{v61}, resp. its direct image under $\overline \rho$ into,
\begin{equation}
\label{v66}
\widetilde{\bar j}_! \, {\mathcal O}_{\widetilde V} (-\overline R) \otimes \prod_{x \in \overline Z} {\mathfrak m} (x)^N \, , \quad \mbox{resp.} \ \overline j_! \, {\mathcal O}_{\overline V} (-\overline R - N \, \overline\rho \, (\overline Z))
\end{equation}
with resulting quotients $\widetilde {\mathcal Q}_N$, resp. $\overline {\mathcal Q}_N$ such that in,
\begin{equation}
\label{v67}
\xymatrix{
\overline{\rm Ext}^1_{\overline X / f} (\sigma^* \Omega^{\otimes m} , \overline {\mathcal Q}_{\rm big}) \ar[d]^{\wr} \ar[r] &\overline{\rm Ext}^1_{\widetilde X / f} (\rho^* \Omega^{\otimes m} , \widetilde {\mathcal Q}_{\rm big}) \ar[d]^{\wr} \\
\overline{\rm Ext}^1_{\overline X / f} (\sigma^* \Omega^{\otimes m} , \overline {\mathcal Q}_{N}) \ar[r] &\overline{\rm Ext}^1_{\widetilde X / f} (\rho^* \Omega^{\otimes m} , \widetilde {\mathcal Q}_N)
}
\end{equation}
the vertical arrows are isomorphisms. As such if we form the diagram of \eqref{v4} with the kernel of \eqref{v61} replaced by the first term in \eqref{v66}, and pass to maximal separated quotients, then we can reduce the obstruction group for our assertion to,
\begin{equation}
\label{v68}
\overline{\rm Ext}^2_{\overline X / f} (\sigma^* \Omega^{\otimes m} , \overline j_! \, {\mathcal O}_{\overline V} (- \overline R - N \, \overline\rho \, (\overline Z))
\end{equation}
whose dual modulo torsion is,
\begin{equation}
\label{v69}
{\mathbb H}_0 (\overline V , \sigma^* \Omega^{\otimes m} (\overline R + N \, \overline\rho \, (\overline Z)) \otimes \omega_{\overline X}) \, .
\end{equation}
As such by \thref{lem:ct1} any differential in \eqref{v69} is, after completion, around $\rho (\overline Z)$ a power of \eqref{Res9}. Consequently if we repeat the steps of the proof of \thref{thm:v1} then in the key estimate of \eqref{v36}, the dynamical residue around $\overline\rho ( \overline Z )$ is given by \eqref{Res13} which is non-positive by construction. Thus \eqref{v69} is zero. Equally the classes in \eqref{v68} may more easily, and equivalently, be obtained by the simpler expedient of lifting along the top rightmost vertical in \eqref{v4} and pushing down (to the missing group) along the (missing) top rightmost horizontal, then mapping to \eqref{v68}. The advantage of this latter formulation is that we know from \eqref{v24}--\eqref{v25} that such classes are obtained by multiplication by a bump function with support close to $\rho (\overline Z)$ and extending by zero. Thus, for trivial support reasons, their product whether with the torsion classes, which under \thref{altsetup:v1} are concentrated at parabolic repelling cycles, is zero.
\end{proof}

Irrespectively there is even a case where we can guaratee elements in the kernel whether of \eqref{v55} or \eqref{v64}, to wit:

\begin{bonus}\thlabel{bonus:v1}
Let everything be as in \thref{thm:v1}, with $A_i$, $0 \leq i \leq N$, and $C_i \hookrightarrow A_i$, resp. $I_i : C_i \to X$ the closed annuli or circles that were extracted in item~(b) of \thref{setup:ct1}, then there is an exact sequence,
\begin{eqnarray}
\label{v70}
0 \to \prod_i {\rm Ext}^0_{X/f} (\Omega^{\otimes m} , (I_i)_* I_i^* {\mathcal O}_X) &\to &\coprod_i {\rm Ext}_1^{A_i \backslash C_i/f} (\Omega^{\otimes m} , {\mathcal O}_{A_i}) \nonumber \\
&\to &{\rm Ext}^1_{X/f} (\Omega^{\otimes m} , {\mathcal O}_X)
\end{eqnarray}
whose image in the rightmost group of \eqref{v70} is a subgroup of the kernel whether of \eqref{v55} or \eqref{v64}.
\end{bonus}

\begin{proof}
Let $z : A \hookrightarrow {\mathbb C}$ be an embedding of an annulus onto a standard one, and define,
\begin{equation}
\label{v71}
{\mathbb R} \supseteq \vert A \vert := \vert z\vert (A) \supseteq \vert C_i \vert := \vert z\vert (C_i) \, .
\end{equation}
Then by \thref{cor:ct1} we have a Dolbeault type isomorphism,
\begin{equation}
\label{v72}
{\rm H}_c^1 (\vert A \vert , {\mathbb C}) \xrightarrow{ \ * \ } {\rm Ext}_1^{A/f} (\Omega^{\otimes m} , A) : \beta \mapsto \beta^* := (\vert z \vert^* \beta)^{0,1} \otimes \left( z \frac\partial{\partial z} \right)^{\!\!\otimes m}
\end{equation}
while integration affords a canonical isomorphism,
\begin{equation}
\label{v73}
\int : {\rm H}_c^1 (\vert A \vert , {\mathbb C}) \longrightarrow {\mathbb C} \, .
\end{equation}
Thus if $\beta_{i0}^* , \beta_{i\infty}^*$ are classes on the components $A_{i0}, A_{i\infty}$ associated to $\beta_{i0} , \beta_{i\infty}$ in either component of $\vert A_i \vert \backslash \vert C_i \vert$ then the kernel in \eqref{v70} defines the relation,
\begin{equation}
\label{v74}
\beta^*_{i0} = \beta^*_{i\infty} \quad \mbox{iff} \quad \int \beta_{i0} = \int \beta_{i\infty} \, , \quad 0 \leq i \leq N \, .
\end{equation}
As such if $b_i$ is a radial bump function in a small neighbourhood of $C_{i0}$ then by \thref{cd:fact3} the image in the final group of \eqref{v70} is generated by the classes,
\begin{equation}
\label{v75}
B_i := \overline\partial (b_i) \left( z_i \frac{\partial}{\partial z_i} \right)^{\!\otimes m} \Bigl\vert_{A_{i0}} \in {\rm Ext}_1^{A_{i0}/f} (\Omega^{\otimes m} , {\mathcal O}_{A_i}) \, .
\end{equation}
Now suppose, employing the summation convention, that some combination of these is zero in the rightmost group of \eqref{v70} then,
\begin{equation}
\label{v76}
\alpha_i \, B_i = \overline\partial (t) \, , \ d_1(t)=0 \, , \ t \in \Gamma (X , {\mathcal A}_X^0 \otimes T_X^{\otimes m}) \, , \ \alpha_i \in {\mathbb C}
\end{equation}
for $d_1$ the differential of \eqref{cd10}. In particular on a given annulus $A_i$ we have the following invariant holomorphic tensors,
\begin{equation}
\label{v77}
t \ \mbox{on} \ A_{i,0} \backslash {\rm supp} (b_i) \, , \ t-\alpha_i b_i \left( z_i \, \frac\partial{\partial z_i} \right)^{\!\!\otimes m} \mbox{on} \ {\rm supp} (b_i) \cap A_{i0} \, , \ t \ \mbox{on} \ A_i \backslash \overline{A_{i0}} \, .
\end{equation}
Plainly, however, to an invariant class, $C$, say, in $T^{\otimes m}$, there is associated a unique invariant $C^{\vee}$ in $\Omega^{\otimes m}$ satisfying,
\begin{equation}
\label{v74.bis}
C^{\vee} (C)=1 \, .
\end{equation}
So by \thref{claim:v1} the initial and final invariant tensors in \eqref{v77} are zero. Equally, however, even the middle one must satisfy,
\begin{equation}
\label{v75.bis}
t-\alpha_i \, b_i \left( z_i \, \frac\partial{\partial z_i} \right)^{\!\!\otimes m} = \alpha'_i \left( z_i \, \frac\partial{\partial z_i} \right)^{\!\!m} \ \mbox{on} \ {\rm supp} (b_i) \cap A_{i0} \, , \ \alpha'_i \in {\mathbb C}
\end{equation}
while in the initial, resp. final, extreme of $A_{i0}$, $b_i \to 0$, resp. $1$, so:
\begin{equation}
\label{v76.bis}
\alpha'_i = 0 \quad \mbox{and} \quad \alpha_i + \alpha'_i = 0 \, .
\end{equation}
Consequently \eqref{v70} is indeed exact. Finally, it's plain, for trivial support reasons, that the image of \eqref{v70} is zero in any part of the right hand side of \eqref{v55} or \eqref{v64} that isn't supported on the annuli themselves, while on the annulus, $A_i$, the long exact sequence of invariant Ext with compact support associated to,
\begin{equation}
\label{v77.bis}
0 \longrightarrow j_! \, {\mathcal O}_{A_i \backslash C_i} \longrightarrow {\mathcal O}_{A_i} \longrightarrow (I_i)_* \,  I_i^* {\mathcal O}_{C_i} \longrightarrow 0
\end{equation}
shows that here the image is zero too. 
\end{proof}

An initial corollary of which is,

\begin{cor}\thlabel{cor:v1}
Let everything be as in \thref{not:v1}, and at a parabolic non-repelling cycle, i.e. $\bullet = 0-$ of \thref{defn:ct1}, where the order of tangency of the first return map to a rational rotation of order $r_{\bullet}$ is $r_{\bullet} e_{\bullet}$, define,
\nomenclature[N]{Extra ramification at non-repelling}{parabolics \hyperref[v78]{$\delta$}} 
\begin{equation}
\label{v78}
\delta_{\bullet} = \left\{\begin{matrix}
1 \, , &\mbox{if exactly $e_{\bullet} > 0$ forward orbits of ${\rm Ram}_t$ limit on $\bullet$,} \hfill \\
0 \, , &\mbox{otherwise} \hfill
\end{matrix} \right.
\end{equation}
with $\delta$ the sum of \eqref{v78} over all such cycles, then, for $\varepsilon_{\bullet}$ as in \eqref{ct71},
\begin{equation}
\label{v79}
2 \, \# \, ({\rm HR}) + \# \, ({\rm SD}) + \# \, ({\rm CR}) + \delta \leq \sum_{\bullet \, \in \, {\rm HR} \, \cup \, {\rm SD}} \varepsilon_{\bullet} + \# \, (\mbox{wild tails}) \, .
\end{equation}
\end{cor}

\begin{proof}
The statement is trivial for Latt\`es examples, so we may suppose \thref{thm:v2} or \thref{var:v4}, with $D=0$, according as $\delta = 0$, resp. $1$. Now $\widetilde{\mathcal Q}_{\rm diag} \hookrightarrow \widetilde {\mathcal Q}_{\rm big}$, has, in the notation of \thref{sum:ct1} codimension,
\begin{equation}
\label{v80}
\sum_{\bullet \, \in \, {\rm par}} (e_{\bullet}^2 - 1) \, .
\end{equation}
As such, by \eqref{ct72}, the right hand side whether in \eqref{v55} or \eqref{v64} is,
\begin{equation}
\label{v82}
\deg ({\rm Ram}_f) + \# \, ({\rm par}) + \# \, ({\rm SD}) + \# \, ({\rm CR})  + \# \, ({\rm HR}) - \# \, \{\mbox{wild tails}\} \, .
\end{equation}
On the other hand by \thref{claim:ct1} and \thref{cor:ct1} the dimension of the left hand side of \eqref{v55}, resp. \eqref{v64} is,
\begin{equation}
\label{v83}
\deg ({\rm Ram}_f) + \# \, ({\rm par}) \quad \mbox{(resp. $-\delta$)}
\end{equation}
so taking into account \thref{bonus:v1} we get \eqref{v79}.
\end{proof}

Now following \cite[Theorem 2]{mitsu} this can be improved to,

\begin{cor}\thlabel{cor:v2}
Let $f$ be a rational map of ${\mathbb P}^1$ with $\delta$ as in \eqref{v78} et seq. then,
\begin{equation}
\label{v84}
2 \, \# \, ({\rm HR}) + \# \, {\rm SD} + \# \, {\rm CR}  + \delta \leq \# \,  \{\mbox{wild tails}\} + \sum_{\bullet \, \in \, {\rm HR}} \varepsilon_{\bullet} \, .
\end{equation}
\end{cor}

\begin{proof}
Following op. cit.  \S5, we eliminate the contribution of $\varepsilon_{\rm SD}$ from the right hand side of \eqref{v19} by applying \thref{thm:v2} to find a deformation $f_s$, $s$ in a small disc, of $f$ which leaves fixed a closed annulus in each Herman ring and the cardinality of any tame ramification therein; the multipliers at any cycles other than Siegel discs in the list of \thref{not:ct1} and the cardinality of any tame or bounded ramification therein; but all Siegel discs become simple attractors, i.e. $+$ of op. cit., while conserving at least the contribution of the tame and bounded ramification, $R_{\rm SD}$. Now all such ramification points are contained in pointed regions which map to the disc pointed in the origin, so the same is true of the deformations $f_s$ at the attractors which we have just created. In particular therefore the old ramification coming from $R_{\rm SD}$ cannot account for the critical point in the immediate basin of attraction guaranteed by \thref{F3grp}, i.e.
\begin{equation}
\label{v85}
\# \,  \{\mbox{new wild tails}\} \leq \# \,  \{\mbox{old wild tails}\} - \# \, {\rm SD}
\end{equation}
and since $\varepsilon_{\rm SD}$ is zero for $f_s$, we've achieved our initial goal.
\end{proof}

We can, of course, use the invariant compactly supported Beltrami differential of \thref{bonus:v1} to pinch an invariant circle in a Herman ring and so make it into a Siegel disc. Unfortunately, however, the limiting procedure necessarily increases the number of components, and, in so doing, may increase the number of wild tails even though the wild ramification itself is constant. Thus we obtain the not wholly satisfactory,

\begin{cor}\thlabel{cor:v3}
Let everything be as in \thref{cor:v2} then the left hand side of \eqref{v84} admits the alternative upper bound,
\begin{equation}
\label{v86}
\# \, {\rm Ram}_w
\end{equation}
i.e. the number of wild ramification points counted without multiplicity.
\end{cor}

\begin{proof}
We could simply quote \cite[Proposition 2]{mitsu}, but we can usefully offer a little simplification and better understand why we don't just have wild tails in \eqref{v86} by recalling what is relevant. Specifically choose a generic invariant circle in each connected component of each cycle whose first return map is a Herman ring, and let $C$ be the union of these so that we have fibre squares of closed embeddings and their complements,
\begin{equation}
\label{v87}
\xymatrix{
\Gamma \, \ar[d]_f \ar@{^{(}->}[r] &{\mathbb P}^1 \ar[d]^f &V \ar[d]^f \ar@{_{(}->}[l] \\
C \, \ar@{^{(}->}[r] &{\mathbb P}^1 &U \ar@{_{(}->}[l]
}
\end{equation}
Now write $U = \underset{i}{\coprod} U_i$ for the connected components of $U$, and $V_i = \underset{ij}{\coprod} V_{ij}$ for those of $V$ lying over $U_i$, then all of these are discs minus a (possibly empty) set of circles so collapsing these to points makes any $U_i$ or $V_{ij}$ into a $S^2$. Better still this operation is identical with real blowing down so if we put the unique conformal structure on the resulting $\overline U_i$ then over any circle in $\Gamma$ of \eqref{v87} where $f$ is a covering of degree $\gamma$, $\overline V_{ij} \to \overline U_i$ is a ramified cover of degree $\gamma$, so we have rational maps,
\begin{equation}
\label{v88}
f_{ij} : \overline V_{ij} \longrightarrow \overline U_i
\end{equation}
of degrees $d_{ij}$ such that,
\begin{equation}
\label{v89}
\sum_j d_{ij} = \deg (f) \, , \quad \forall \, i \, .
\end{equation}
Observe, however, following \cite[\S6]{mitsu}, that cycles of $f$ do not automatically yield cycles of some $f_{ij}$, but, rather, a cycle $Z$ determines a set of indices, resp. possibly 2 such if it were a Herman ring,
\begin{equation}
\label{v90}
K(Z) = \{ij \in K(Z)\} \subseteq \pi_0 (V)
\end{equation}
and a cyclic permutation $\sigma$ of $K$, resp. possibly 2 such, such that if,
\begin{equation}
\label{v91}
F_Z := \coprod_{k \in K(Z)} {\mathbb P}^1_k \longrightarrow \coprod_{k \in K(Z)} {\mathbb P}^1_k : (k,z) \longmapsto (k+1 , f_k(z))
\end{equation}
then our cycle of interest becomes a cycle of $F_Z$. Specifically, multipliers and the r\'esidu iteratif are diffeomorphism invariants, so the only way that a Cremer point, or Siegel disc wouldn't occur for some $F_Z$ would be if we turned the Siegel disc into a rotation of ${\mathbb C}$, but this is impossible since off the boundaries that we have collapsed in \eqref{v88} everything else is quasi conformal to what we started with. Similarly even with the Herman ring, neighbourhoods of either boundary component are also quasi conformal to our initial situation, so their moduli cannot become infinite, i.e. a Herman ring becomes 2 Siegel discs, of possibly different $F_Z$, and not a rotation of ${\mathbb C}$.

As such we require to know \thref{cor:v2} not just for rational maps, but in the slightly greater generality of endomorphisms of the form \eqref{v91}. This is, however, an irrelevant change to the local considerations of \thref{sum:ct1}, \thref{rmk:v3} and \thref{altsetup:v1}, while the lack of connectedness of the domain of \eqref{v91} doesn't occasion any problem to the vanishing theorem either, cf. \eqref{v41}--\eqref{v42}, so we have \eqref{v79} for endomorphisms of the form \eqref{v91}. On the other hand to get the left hand side of \eqref{v84} for our initial $f$, we need to sum over all $F_Z$ that may contain Siegel discs, Cremer points or suitable non-repelling parabolic cycles, and this may split a single wild tail into the same but for possibly more than one $F_Z$. As such the resulting bound,
\begin{equation}
\label{v92}
\sum_z \# \ \{\mbox{wild tails of} \ F_Z\}
\end{equation}
may be no better than \eqref{v86}.
\end{proof}

To which we can usefully add,

\begin{rmk}\thlabel{rmk:v4}
Notice following \cite[Theorem 3]{mitsu} or \cite[Theorem A.1]{milnor2} that if either $f$ or a map of the form in \eqref{v91} had only 2 (counted without multiplicity) critical points then it cannot have a Herman ring since there must result 2 maps of the form \eqref{v91} having a Siegel disc. Starting from which, by induction on max degree of cyclic maps restricted to a connected component in the direct sum, the limiting case of $\# \, {\rm HR} = \deg (f)-1$ cannot occur in \eqref{v84} since (for parity reasons) a limiting case with 1 Siegel disc doesn't occur either.
\end{rmk}

%\end{document}

%%%%%%%%%%%%%%%

\newpage

\appendix
\section{Duality on cylinders}\label{Dc}

Let us begin by introducing our,

\begin{setupnot}\thlabel{Dc:not1}
Let $* \hookrightarrow \Delta$ be a pointed disc (with the point identified to the origin) and $\rho : \widetilde\Delta \to \Delta$ 
\nomenclature[N]{Real blow up}{\hyperref[Dc:not1]{$\rho : \widetilde\Delta \to \Delta$}} 
the real blow up in the same, then ${\mathcal A}^0_{\widetilde\Delta}$ 
\nomenclature[N]{Real blow up, smooth functions}{\hyperref[Dc:not1]{${\mathcal A}$}$^0$} 
\nomenclature[N]{Taylor series, projectors}{\hyperref[Dc10]{$\Pi$}$_1$} 
denotes the sheaf of complex valued $C^{\infty}$ functions smooth up to the boundary $i : E := \rho^{-1} (*) \hookrightarrow \widetilde\Delta$, 
\nomenclature[N]{Real blow up, exceptional divisor}{\hyperref[Dc:not1]{$E$}}
and ${\mathcal O}_{\widetilde\Delta} \hookrightarrow {\mathcal A}^0_{\widetilde\Delta}$ 
\nomenclature[N]{Real blow up, holomorphic functions}{\hyperref[Dc:not1]{${\mathcal O}$}}
\nomenclature[D]{Holomorphic \hyperref[Dc:not1]{functions on a real}}{blow up}
the sub-sheaf of those holomorphic on the complement $\widetilde\Delta \backslash E$. Similarly the $\widetilde\Delta$ smooth sections of the anti-holomorphic line bundle $\varphi^* {\mathbb V} (\overline\Omega (\log *))$ (EGA notation) define a locally free ${\mathcal A}^0_{\widetilde\Delta}$ module, ${\mathcal A}^{0,1}_{\widetilde\Delta} (\log E)$,
\nomenclature[N]{Real blow up, $(0,1)$ forms}{\hyperref[Dc:not1]{${\mathcal A}$}$^{0,1} (\log E)$}
and the $\overline\partial$ operator on $\Delta$ extends to a map,
\begin{equation}
\label{Dc1}
{\mathcal A}^0_{\widetilde\Delta} \xrightarrow{ \ \overline\partial \ } {\mathcal A}^{0,1}_{\widetilde\Delta} (\log E) \, .
\end{equation}
Explicitly, and to fix notation, if we choose an embedding of $\Delta$ in ${\mathbb C}$, i.e. a coordinate $z$, then we have smooth functions on $\widetilde\Delta$,
\nomenclature[N]{Real blow up, polar coordinates}{\hyperref[Dc2]{$r,\theta$}}
\begin{equation}
\label{Dc2}
r = \vert z \vert : \widetilde\Delta \longrightarrow {\mathbb R}_{\geq 0} \, , \quad \exp (\theta) : \widetilde\Delta \longrightarrow E \, , \quad \theta \in {\mathbb R} (1)
\end{equation}
along with differential operators,
\begin{equation}
\label{Dc3}
 \frac{\partial}{\partial r} := \frac zr \frac{\partial}{\partial z} + \frac{\overline z}r \frac{\partial}{\partial \overline z} \, , \quad \frac{\partial}{\partial \theta} = z \frac{\partial}{\partial z} - \overline z \frac{\partial}{\partial \overline z} \, .
\end{equation}
As such an explicit formula for \eqref{Dc1} is,
\nomenclature[N]{Real blow up, anti-holomorphic derivative}{\hyperref[Dc4]{$\overline D$}}
\begin{equation}
\label{Dc4}
\overline\partial (\varphi) = \overline D (\varphi) \frac{d\overline z}{\overline z} \, , \quad \overline D := 1/2 \left(r \frac{\partial}{\partial r} - \frac{\partial}{\partial \theta} \right) = \overline z \, \frac{\partial}{\partial \overline z}
\end{equation}
and, manifestly, ${\mathcal O}_{\widetilde\Delta}$ is identically ${\rm Ker} (\overline\partial) = {\rm Ker} (\overline D)$.
\end{setupnot}

Unsurprisingly, therefore, our goal is to prove,

\begin{fact}\thlabel{Dc:fact1}
The map \eqref{Dc1} is a resolution of ${\mathcal O}_{\widetilde\Delta}$, i.e. we have an exact sequence of sheaves,
\begin{equation}
\label{Dc5}
0 \longrightarrow {\mathcal O}_{\widetilde\Delta} \longrightarrow {\mathcal A}^0_{\widetilde\Delta} \xrightarrow{ \ \overline\partial \ } {\mathcal A}^{01}_{\widetilde\Delta} (\log E) \longrightarrow 0 \, .
\end{equation}
\end{fact}

The proof will be divided into a series of assertions beginning with,

\begin{claim}\thlabel{Dc:claim1}
Let $K \subsetneqq E$ be compact (the important cases is $K$ a point) and $K \subset U \subset \widetilde\Delta$ an open neighbourhood with $\Psi = \psi \frac{d\overline z}{\overline z}$ where $\psi$ is in $\Gamma (U,{\mathcal A}^0_{\widetilde\Delta})$. Then, for every $k \geq 0$, there is a $C^k$ function $\varphi$ on a neighbourhood $U \supset\supset V$ of $K$ such that,
\begin{equation}
\label{Dc6}
\overline\partial \varphi = \Psi \quad \mbox{in} \quad V \, .
\end{equation}

As the statement suggests the proof involves Taylor series, and since their use is endemic to what follows we may usefully introduce some further,
\end{claim}

\begin{Notation}\thlabel{Dc:not2}
Let $i : E \hookrightarrow \widetilde\Delta$, and $\exp (\theta) : \widetilde\Delta \to E$ be as in \thref{Dc:not1} then we have a projector,
\begin{equation}
\label{Dc7}
\Pi : \exp (\theta)^* i^* : {\mathcal A}^0_{\widetilde\Delta} \longrightarrow {\mathcal A}^0_{\widetilde\Delta}
\end{equation}
which in the presence of the choice of embedding $z : \Delta \hookrightarrow {\mathbb C}$ of op. cit. We extend to ${\mathcal A}^{0,1}_{\widetilde\Delta} (\log E)$ by the formula,
\begin{equation}
\label{Dc8}
\Pi \left(\psi \frac{d \overline z}{\overline z} \right) := \Pi (\psi) \frac{d \overline z}{\overline z} \, .
\end{equation}
In particular, therefore, by \eqref{Dc4} and \eqref{Dc8},
\begin{equation}
\label{Dc9}
[\Pi , \overline D] = 0 \quad \mbox{and} \quad [\Pi,\overline\partial] = 0
\end{equation}
and we can conveniently formulate Taylor's theorem by the following induction on $k \geq 0$,
\nomenclature[N]{Taylor series, projectors}{\hyperref[Dc10]{$\Pi$}$_1$}
\begin{equation}
\label{Dc10}
\Pi_0 = \Pi \, , \quad R_k := 1 - \sum^k_{i=0} \Pi_i \quad \mbox{is a projector onto} \quad z^{k+1} \widetilde A^0 \, ,
\end{equation}
$$
\Pi_{k+1} = z^{k+1} \, \Pi_0 \, \frac1{z^{k+1}} R_k \, , \quad [\Pi_i , \Pi_j] = \delta_{ij} \, \Pi_j \, .
$$
Thus if we again extend the $\Pi_i$ from functions to forms by their action on the coefficient as in \eqref{Dc8} then, by induction in $k$, we continue to have the commutation relations,
\begin{equation}
\label{Dc11}
[\Pi_i , \overline D] = 0 \, , \quad [\Pi_i , \overline\partial] = 0 \, , \quad i \geq 0 
\end{equation}
while we may equally write the remainder as,
\nomenclature[N]{Taylor series, remainder}{\hyperref[Dc10]{$R$}$_k$} 
\begin{equation}
\label{Dc12}
R_k = \frac{\vert z \vert^k}{k!} \int_0^1 t^k \left( \frac{\partial^{k+1} f}{\partial r^{k+1}} \right) (z (1-t)) dt \, , \ \mbox{where} \ \vert R_k \vert \leq \frac{\vert z \vert^{k+1}}{(k+1)!} \, \Vert f \Vert_{k+1}
\end{equation}
for a suitable $C^{k+1}$-norm, $\Vert \ \Vert_{k+1}$.
\end{Notation}

To employ this formalism in the proof of \thref{Dc:claim1}, we first require

\begin{lem}\thlabel{Dc:lem1}
Suppose a compactly supported function $\psi$ is a section of the ideal ${\mathcal A}^0_{\widetilde\Delta} z$ in a neighbourhood of a compact subset $K$ of $E$ (possibly all of it) then there is an open neighbourhood $U \supset K$ such that,
\begin{equation}
\label{Dc13}
K\psi : s \longmapsto K_s \left( \psi \frac{d\overline z}{\overline z} \right) := \int_{\widetilde\Delta} \psi \frac{d\overline z \, dz}{\overline z(z-s)} \, , \quad s \in V
\end{equation}
is a $C^0$ solution of $\overline\partial (K\psi) = \psi \frac{d \overline z}{\overline z}$, $\forall \, \alpha <1$.
\end{lem}

\begin{proof}
By hypothesis $\psi = z \psi_0$ for a $\widetilde\Delta$ smooth function $\psi_0$, so \eqref{Dc13} is absolutely integrable, and $K\psi$ is continuous up to the boundary by the dominated convergence theorem. 
\end{proof}

We are now in a position to give,

\begin{proof}[Proof of \thref{Dc:claim1}]
In the notation of \thref{Dc:not2}, we have,
\begin{equation}
\label{Dc14}
\Psi = \psi \frac{d\overline z}{\overline z} = \left( \sum^k_{i=0} \Pi_i (\psi) + R_k \, \psi \right) \frac{d\overline z}{\overline z}
\end{equation}
while a convenient way to write \eqref{Dc13} is,
\begin{equation}
\label{Dc15}
K_s (\gamma \, d\overline z) = \int_{\widetilde\Delta} \gamma (w+s) \frac{d \overline w \, dw}w
\end{equation}
albeit that one must be carefull that $\widetilde\Delta$-smooth, etc, is a statement not about $\gamma$ but $\overline z \gamma$. In particular by \eqref{Dc12} and \eqref{Dc14},
\begin{equation}
\label{Dc16}
\frac{R_k \psi}{\overline z} = z^k \cdot \{\widetilde\Delta \ \mbox{continuous}\} \, .
\end{equation}
As such we can safely differentiate $k$-times in $s$ under the integral sign of \eqref{Dc15}, while applying \thref{Dc:lem1} to conclude that,
\begin{equation}
\label{Dc17}
K \left( R_k \, \psi \, \frac{d\overline z}{\overline z} \right)\ \mbox{is uniformly $C^k$ in $\widetilde\Delta$ up to the boundary.}
\end{equation}
Better still the other terms in \eqref{Dc14} have, in the notation of \eqref{Dc2}, the form,
\begin{equation}
\label{Dc18}
\Pi_i (\psi) \frac{d\overline z}{\overline z} = \left(z^i \psi_i (\theta) \right) \frac{d\overline z}{\overline z} \, , \quad \psi_i \in C^{\infty}(E) \,.
\end{equation}
As such if we choose a point, $\bullet \in E$, with arguably a choice of $\bullet \notin K$ being preferable, then,
\begin{eqnarray}
\label{Dc19}
\overline\partial \left( z^i \int_{\bullet}^{\theta} \psi_i (\theta) \, d\theta \right) &= &z^i \, \overline D \left( \int_{\bullet}^{\theta} \psi_i (\theta) \, d\theta \right) \frac{d\overline z}{\overline z} \nonumber \\
&= &-\Pi_i (\psi) \frac{d\overline z}{\overline z}
\end{eqnarray}
and whence \eqref{Dc6} from \eqref{Dc17} and \eqref{Dc19}. 
\end{proof}

The next thing we need to do is to modify the $C^k$-solutions of \eqref{Dc6} so that they fit together into a $C^{\infty}$ solution, which, requires a further,

\begin{claim}\thlabel{Dc:claim2}
Let $K \subset E$ be a closed interval about a point $e \in E$, and $h$ a holomorphic function on a neighbourhood of $\theta \in K$, $r \leq R$ which is $C^{k+2}$, $k \geq 0$, up to the boundary, then for every strictly smaller $K \supset K' \ni e$, $R' < R$, and $\varepsilon > 0$ there is a section, $g$, of ${\mathcal O}_{\widetilde\Delta}$ on a neighbourhood of $\theta \in K'$, $r \leq R'$ such that,
\begin{equation}
\label{Dc20}
\sup_{\theta \in K' , r \leq R'} \Vert h-g \Vert_k \leq \varepsilon
\end{equation}
where $\Vert \ \Vert_k$ is the maximum of any partial derivative $\partial^{i+j} / \partial r^i \, \partial \theta^j$ of order $i+j \leq k$.
\end{claim}

\begin{proof}
By \eqref{Dc4} the restriction of a holomorphic function to $E$ is constant so there is a polynomial, $z \mapsto p_k(z)$, of degree $k+1$ such that close to the boundary,
\begin{equation}
\label{Dc21}
h = p_k (z) + O(\vert z \vert^{k+2}) \, .
\end{equation}

Furthermore $p_k (z)$ is everywhere a section of ${\mathcal O}_{\widetilde\Delta}$, so without loss of generality we may suppose that it is zero. In particular for $\gamma$ a smooth function on the boundary of $K \times [0,R]$ which vanishes to order $k+1$ on $\partial K \times [0,R]$ close to zero, consider,
\begin{equation}
\label{Dc22}
g(s) := \int_{\partial (K \times [0,R])} \frac{\gamma (z) \, dz}{z-s} \, , \quad s \in {\rm int} (K \times [0,R]) \, .
\end{equation}

In particular the vanishing of $\gamma$ to order $k+1$ at the boundary not only guarantees that \eqref{Dc22} is well defined, but so are all of the derivatives,
\begin{equation}
\label{Dc23}
g^i (s) = \int_{\partial (K \times [0,R])} \frac{\gamma (z) \, dz}{(z-s)^{i+1}} \, , \quad 0 \leq i \leq k+1 \, .
\end{equation}

Better still if around $0$, $\gamma$ is flat on $\partial K \times [0,R]$, then \eqref{Dc22} is, in fact, well defined for all $i \geq 0$, so, it will suffice to find a flat $\gamma$ such that the moduli of all of,
\begin{equation}
\label{Dc24}
h^i (s) - g^i (s) = \int_{\partial (K \times [0,R])} \frac{(h-\gamma)(z)dz}{(z-s)^{i+1}} \, , \ s \in K' {\times} [0,R'] \, , \ 0 \leq i \leq k+1
\end{equation}
are uniformly bounded. To this end fix $\delta > 0$, to be decided, then, the distance in $\Delta$ between $\partial (K \times [0,R])$ and $K' \times [\delta , R']$ is bounded below, so for $s$ in the latter region this is a non-issue. Equally for $\delta$ sufficiently small,
\begin{equation}
\label{Dc25}
\vert h^i (s) \vert \leq O(\delta^{k+2-i}) \, , \quad s \in K \times [0,\delta] \, , \quad 0 \leq i \leq k+1 \, .
\end{equation}

So for $\delta < O(\varepsilon)$ it will suffice to bound the moduli of $g^i (s)$ by $\varepsilon$ for $s \in K' \times [0,\delta]$, or, equivalently of $g^i (0)$ and $g^i (s) - g^i (0)$. To this end, observe that if $\kappa$ is the distance between $K'$ and $K$, then for $z \in \partial K \times [0,R]$, the smallest point on a line from $z-s$ to $z$ is of order at worst,
\begin{equation}
\label{Dc26}
O(\kappa \vert z \vert) \, , \quad s \in K' \times [0,R] \, .
\end{equation}
Consequently we get the estimate,
\begin{equation}
\label{Dc27}
\vert g^i(s) - g^i(0)\vert \ll \frac{\vert s \vert}{\kappa^{i+2}} \int_{\partial (K \times [0,R])} \left\vert \frac{\gamma \, dz}{z^{i+2}} \right\vert \, , \quad 0 \leq i \leq k
\end{equation}
while the bound for $\vert h^i (0) - g^i (0)\vert$ is easier by \eqref{Dc24} and is even valid up to $k+1$. As such,  and bearing in mind that \eqref{Dc21} et seq. allows us to write,
\begin{equation}
\label{Dc28}
h = z^{k+2} h_0 \, , \quad \gamma = z^{k+2} \gamma_0 \, , \quad h_0 , \gamma_0 \in C^0 (\widetilde\Delta) 
\end{equation}
any sufficiently fine approximation of $h_0 \vert_{\partial (K \times [0,R])}$ by $\gamma_0 \vert_{\partial (K \times [0,R])}$ in $\ell_1$, with respect to $\vert dz \vert$, norm on the boundary will yield \eqref{Dc20} for $g$ defined by \eqref{Dc22}, and since this is plainly possible for $\gamma_0$ flat we get that $g$ is a section of ${\mathcal O}_{\widetilde\Delta}$ by \eqref{Dc23} et seq. 
\end{proof}

We can now, therefore, put all this together to give,

\begin{proof} 

[Proof of \eqref{Dc:fact1}]. Fix $e \in E$ and let a section $\Phi$ of ${\mathcal A}_{\widetilde\Delta}^{0,1} (\log E)$ be given over an open $U \ni e$, with $V \subset U$ a neighbourhood of $e$ such that for every $k$ we have $C^k$ solutions of \eqref{Dc6} with $W \subset\subset V$ a neighbourhood such that in the notation of \thref{Dc:claim2}
\begin{equation}
\label{Dc29}
W \subset K' \times [0,R'] \subset K \times [0,R] \subset V \, .
\end{equation}
This allows us to define, inductively, a sequence of $C^{k+2}$ solutions, $\widetilde\varphi_k$, $k \geq 0$, of the $\overline{\partial}$ equation defined on neighbourhoods of a strictly decreasing sequence of compacts $K_k \times [0,R_k] \supsetneqq K' \times [0,R']$ such that in the notation of \eqref{Dc20},
\begin{equation}
\label{Dc30}
\sup_{K_k \times [0,R_k]} \Vert \widetilde\varphi_{k+1} - \widetilde\varphi_k \Vert_k \leq 2^{-k} \, , \quad k \geq 0 \, .
\end{equation}
Specifically $\widetilde\varphi_0$ is any $C^2$ solution, and preceding by induction from $k$ to $k+1$, $k \geq 0$, let $\varphi_{k+1}$ be a $C^{k+3}$ solution of \eqref{Dc6} on a neighbourhood $V_k$ of $K_k \times [0,R_k]$.

\smallskip

As such $h_k := \varphi_{k+1} - \widetilde\varphi_k$ is holomorphic on $V_k$ and $C^{k+2}$ up to the boundary. Consequently if we choose any,
\begin{equation}
\label{Dc31}
K' \times [0,R'] \subsetneqq K_{k+1} \times [0,R_{k+1}] \subsetneqq K_k \times [0 , R_k]
\end{equation}
then we can apply \thref{Dc:claim2} to the latter inclusion of \eqref{Dc31} to find a holomorphic function $g_k$, $C^{\infty}$ up to the boundary, such that,
\begin{equation}
\label{Dc32}
\sup_{K_{k+1} \times [0,R_{k+1}]} \Vert h_k - g_k \Vert_k \leq 2^{-k}
\end{equation}
and whence $\widetilde\varphi_{k+1} := \varphi_{k+1} - g_k$ satisfies \eqref{Dc20}. Plainly, however, such a sequence $\widetilde\varphi_k$ converges to a $C^{\infty}$ function on $W$, from which \thref{Dc:fact1}. 
\end{proof}

\bigskip

{\it En passant} we can pick up a rather considerable,

\begin{bonus}
\thlabel{Dc:bonus1}

Let $K \subset E$ be compact, $i_K$ it's inclusion in $\widetilde\Delta$ then the cohomology in degree 1 of,
\begin{equation}
\label{Dc33}
(i_K)_* \, i_K^* \, {\mathcal A}^0_{\widetilde\Delta} \xrightarrow{ \ \overline\partial \ } (i_K)_* \, i_K^* \, {\mathcal A}^{0,1}_{\widetilde\Delta} (\log E)
\end{equation}
is equal to zero if $K \ne E$, and otherwise is (in the induced topology) isomorphic via the integrals over $E$ of the projectors $\Pi_i$ of \eqref{Dc10} to,
\begin{equation}
\label{Dc34}
\varprojlim_i \int_E (z^{-i} \, \Pi_i) \vert_E \frac{d\overline z}{\overline z} : {\rm Coker} (\overline\partial) \xrightarrow{ \ \sim \ } {\mathbb C} [[z]] \frac{d\overline z}{\overline z} \, .
\end{equation}
In particular (since the $\Pi_i$ are continuous) even though whether $(i_K)_* \, i_K^*$ is separated in its natural topology as an $LF$, i.e. direct limit of Fr\'echet spaces,
\nomenclature[D]{$LF$ \hyperref[Dc:bonus1]{space}}{} 
 is a delicate question, cf. \thref{Dc:CorExtra1}, the cokernel of \eqref{Dc33}, in the quotient of the aforesaid direct limit topology, is always separated.
\end{bonus}

\begin{proof}
We've just finished the case $K \ne E$. However the only place that we actually used this was in \eqref{Dc18} to solve $\overline D$-equation (actually the $d$-equation in disguise) for,
\begin{equation}
\label{Dc35}
\left( z^{-i} \, \Pi_i (\psi)\right) \frac{d\overline z}{z} \vert_E
\end{equation}
so as soon as all of these are zero the proof goes through mutatis mutandis -- in fact easier if $K=E$ since \thref{Dc:claim2} holds on neighbourhoods of $E \times R'$ because holomorphic and $C^0$ in a neighbourhood of $E$ just means holomorphic in a neighbouhood of the origin. As such we already have the inclusions,
\begin{equation}
\label{Dc36}
{\rm Ker} \left( \varprojlim_i \int_E (z^{-i} \, \Pi_i) \vert_E \frac{d\overline z}{\overline z} \right) \subseteq {\rm Im} (\overline\partial) \, .
\end{equation}

At the same time if for some $j \geq 0$ we can solve,
\begin{equation}
\label{Dc37}
\overline\partial (\varphi) = z^j \, \frac{d\overline z}{\overline z} \, , \quad \mbox{i.e.} \ \overline D (\varphi) = z^j
\end{equation}
then from the commutation relations \eqref{Dc11}, we must in fact have,
\begin{equation}
\label{Dc38}
\overline D \left( z^{-j} \, \Pi_j \, \varphi \right) = 1
\end{equation}
i.e. we've found a smooth function on $E$ whose derivative is the Haar measure, which is nonsense. As such we get the reverse inclusion in \eqref{Dc36}, while the fact that \eqref{Dc34} is onto is the content of Borel-Ritt Theorem. 
\end{proof}

In the same vein it seems usefull to note,

\begin{cor}\thlabel{Dc:cor1}
Let $\rho : \widetilde X \to X$ be the real blow up of any Riemann surface in any discrete set of points $Y$ then in the obvious extension of \thref{Dc:not2}, the complex,
\begin{equation}
\label{Dc39}
\Gamma (\widetilde X , {\mathcal A}_{\widetilde X}^0) \underset{\overline\partial}{\longrightarrow} \Gamma \left(\widetilde X , {\mathcal A}_{\widetilde X}^{0,1} (\log E)\right)
\end{equation}
calculates the cohomology of ${\mathcal O}_{\widetilde X}$, which in degree 1 is an extension,
\begin{equation}
\label{Dc40}
0 \longrightarrow {\rm H}^1 (X,{\mathcal O}_X) \longrightarrow {\rm H}^1 (\widetilde X , {\mathcal O}_{\widetilde X}) \longrightarrow \prod_{y \in Y} {\mathbb C} [[ z_y ]] \, d \, \overline z_y \longrightarrow 0
\end{equation}
for $z_y$ a local coordinate at $y$. In particular $\overline\partial$ is closed in the standard Fr\'echet topology of either side of \eqref{Dc39}.
\end{cor}

\begin{proof}
Just as at the end of the proof of \thref{Dc:bonus1}, the image of $\overline\partial$ mus be contained in the kernel, $K$, of the product over $y \in Y$ of the residue maps in \eqref{Dc34}, and it's onto since the Leray spectral sequence degenerates for trivial reasons. As such the image of $\overline\partial$ has finite codimension, ${\rm h}^1 (X,{\mathcal O})$, in $K$, so it's closed.
\end{proof}

All of which is, however, ony a prelude to the intervention of compact support which will require,

\begin{morenot}
\thlabel{Dc:not3}
Let $U$ be an open neighbourhood of $e \in E$ of the form $K^0 \times [0,R)$ for $e \in K^0 \subset E$ connected, with $r \mapsto b(r)$ a $C^{\infty}$ bump function of $r \geq 0$ which is identically 1 close to zero, and identically zero in a neighbourhood of $r < R$ then, if the subscript $c$ denotes compact support, we define projectors,
\begin{equation}
\label{Dc41}
P_i := b \, \Pi_i : \Gamma_c (U , {\mathcal A}^0_{\widetilde\Delta}) \longrightarrow \Gamma_c (U , {\mathcal A}^0_{\widetilde\Delta})
\end{equation}
which again we extend to $\Gamma_c (U , {\mathcal A}^{0,1}_{\widetilde\Delta} (\log E))$ by the formula,
\nomenclature[N]{Taylor series,}{compact support, projectors \hyperref[Dc42]{$P$}$_i$} 
\begin{equation}
\label{Dc42}
P_i \left( \psi \frac{d \overline z}{\overline z} \right) := P_i (\psi) \frac{d \overline z}{\overline z} \, , \quad i \geq 0 \, .
\end{equation}
Notice, however, we do not simply multiply the remainder operators, $R_{\bullet}$, of \eqref{Dc10} by the bump, but rather use that the $P_i$ are honest projectors to define,
\nomenclature[N]{Taylor series,}{compact support, remainder \hyperref[Dc43]{$Q$}$_k$} 
\begin{equation}
\label{Dc43}
Q_k = 1 - \sum^k_{i=0} P_i \, , \quad k \geq 0 \, .
\end{equation}
Needless to say, therefore, the commutativity relations \eqref{Dc11} collapse, but in a way that is far from deadly, i.e. the commutators,
\begin{equation}
\label{Dc44}
\left[P_i , \overline D\right] = - \overline D (b) \, \Pi_i \, , \quad \left[P_i , \overline\partial\right] = - \overline\partial (b) \Pi_i
\end{equation}
are flat around the boundary $E \cap U$ with compact support in $U$.
\end{morenot}

We can usefully employ this formalism in proving,

\begin{fact}
\thlabel{Dc:Fact2}
Let everything be as in \thref{Dc:not3} albeit with any open $U \subseteq \widetilde\Delta$, then, we have a closed map,
\begin{equation}
\label{Dc45}
\Gamma_c (U , {\mathcal A}^0_{\widetilde\Delta}) \underset{\overline\partial}{\longrightarrow} \Gamma_c (U , {\mathcal A}^{0,1}_{\widetilde\Delta} (\log E)) \, .
\end{equation}
\end{fact}

\begin{proof}
Without loss of generality $U$ is connected, $U \cap E \ne \varphi$, and let $(\varphi_n)$ be a sequence of functions on the left supported in compacts $C_n$ such that $\overline\partial \, \varphi_n$ converges to a form $\Psi = \psi \frac{d \overline z}{\overline z}$ with support in a compact $C$. Then, for possibly different $C_n$ and $C$ the same is true whether of $P_0 \, \varphi_n$ or $Q_0 \, \varphi_n$. Irrespectively by Hahn-Banach a subspace is closed if and only if any of its finite dimension subspaces are closed, so, we may equally suppose that all of the $\varphi_n$ belong to,
\begin{equation}
\label{Dc46}
\Gamma_c (U , {\mathcal A}^0_{\widetilde\Delta} \, {\mathfrak m} (*))
\end{equation}
for $*$ some fixed point in the boundary $E$. Indeed if $U \cap E \ne E$, and $* \notin U \cap E$ this is even of codimension $0$, but, otherwise, we need it to ensure that the operator,
\begin{equation}
\label{Dc46.bis}
\Gamma_c (U \cap E , {\mathcal A}^1_E) \longrightarrow \Gamma (U \cap E , {\mathcal A}^0_E) : \psi \longmapsto \left\{ x \longmapsto \int_*^x \psi \right\}
\end{equation}
is inverse to differentiation. Certainly, therefore, the result of applying the operator of \eqref{Dc46.bis} to any $\varphi_n \vert_E$ is a function with compact support in $U \cap E$, but, more is true, to wit,

\begin{claim}
\thlabel{Dc:ClaimExtra1}
There is a compact subset $D \subset U \cap E$ depending only on the sequence $\{\varphi_n\}$ such that every $\varphi_n \vert_E$ has support in $D$.
\end{claim}

\begin{proof}
By the definition of convergence in the $LF$ space to the right of \eqref{Dc45}, there is a compact set $D' \subset U$ containing the support of all of the $\overline\partial \, \varphi_n$, while the proposition is trivial if $U \cap E = E$, so equally we can choose $*$ in \eqref{Dc46.bis} off $D' \cap U$. As such any connected component $V_i$ of $U \cap E$ is an interval, while there is a smallest closed interval $D_i \subset\subset V_i$ containing the compact subset $D' \cap V_i$. Now on every $V_i \backslash D'$, so a fortiori on every $V_i \backslash D_i$ every $\varphi_n \vert_E$ is locally constant. There is, however, a compact subset $C_n \subset U$ such that $\varphi_n \vert_{V_i}$ is zero off $V_i \cap C_n$, so it is, in fact, zero on $V_i \backslash D_i$. As such we can take $D = \cup_i \, D_i$, which is compact because $D'$ is.
\end{proof}

In order to deal with this level of generality (which is way beyond the only case of interest to us, i.e. $E$ complemented in finitely many points) we should, take $b$ in \thref{Dc:not3} to be a sum of radial bump functions adapted to the $D_i$ so that the operators of \eqref{Dc41} take values in compactly supported functions in $U$. Irrespectively, the formalism remains the same, and since the $\varphi_n \vert_E$ converge to the operator of \eqref{Dc46.bis} applied to $\psi \vert_E$ we can suppose, without loss of generality, that $\varphi_n = Q_0 \, \varphi_n$. Now the advantage of this situation is that for $K$ as in \eqref{Dc13} all of the $K (\overline\partial \, \varphi_n)$ are defined, and, since $\varphi_n$ has compact support, equal to $-\varphi_n$, so:
\begin{equation}
\label{Dc48}
Q_n = Q_0 \, \varphi_n \longrightarrow -K \psi \ (= - K Q_0 \, \psi) \quad \mbox{in} \quad C^0 (U) \, .
\end{equation}

Again, however, this is only on compact subsets of $U$, and we need to prove $LF$ convergence similar to \thref{Dc:ClaimExtra1}, i.e.

\begin{claim}
\thlabel{Dc:ClaimExtra2}
There is a compact subset $D \subset \subset U$ containing the support of all of the $\varphi_n$ and $K\psi$.
\end{claim}

\begin{proof}
As in \thref{Dc:ClaimExtra1} start from a compact set $D'$ containing the support of all of the $\overline\partial \, \varphi_n$ and $\psi$. Next choose a very general complete distance function, then by transversality and Morse's lemma, $U$ is exhausted by relatively compact connected open sets $U_t$ (i.e. distance to some base point $<t$). In particular, we may choose $t \gg 0$ such that $D' \subseteq \overline U_t$. As such if there were an $n \in {\mathbb Z}_{\geq 0}$ such that $\varphi_n$ weren't supported in $\overline U_t$ then $\varphi_n$ would be holomorphic on $U \backslash \overline U_t$ with support contained in some $\overline U_{\!s}$, where $s > t$ is minimal with this property. As such $\varphi_n$ is identically zero on an open neighbourhood of $\partial U_s$, and since this is compact we have the absurdity that $s$ isn't minimal.
\end{proof}

Beyond this we have the not inconsiderable,

\begin{bonus}
\thlabel{Dc:ClaimExtra3}
The sequence $\varphi_n$, irrespectively of whether it belongs to the subspace of \eqref{Dc46}, converges in the $LF$ space $C_c^0(U)$.
\end{bonus}

\begin{proof}
If $U \cap E \ne E$ we've already proved this, and otherwise we've proved that if,
\begin{equation}
\label{Dc48.bis}
\overline{\overline\partial \, \Gamma_c (U , {\mathcal A}^0_{\widetilde\Delta} {\mathfrak m}(*))} \ni x \subseteq C_c^{\alpha} (U) \, ,
\end{equation}
for some $\alpha > 0$ (otherwise \eqref{Dc48} needn't converge) there is a continuous $y$ with compact support such that,
\begin{equation}
\label{Dc48.bis.2}
x = \overline \partial \, y \quad \mbox{and} \quad y(*) = 0 \, .
\end{equation}
Consequently if in the situation that $U \cap E = E$ we chose $\gamma$ to be a bump function (in both $r$ and $\theta$ directions) such that $\gamma (*) = 1$ then $\overline\partial$ of,
\begin{equation}
\label{Dc48.bis.3}
(\varphi_n - \varphi_n (*) \, \gamma) + \varphi_n (*) \, \gamma
\end{equation}
converges in rightmost space of \eqref{Dc48.bis} modulo the leftmost. In this quotient, however, $\overline\partial \, \gamma \ne 0$, since, otherwise by \eqref{Dc48.bis.2} we'd have a non-zero holomorphic function on $U$ with compact support. Consequently by \eqref{Dc48.bis.3} the $\varphi_n(*)$ converge, which was the outstanding issue in the $U \cap E = E$ case. 
\end{proof}

As such we can quickly conclude to $C_c^k (U)$ for all $k$. Indeed if we write our data as,
\begin{equation}
\label{Dc48.bis.4}
\overline D (\varphi_n) \longrightarrow \psi \in \Gamma_c (U,{\mathcal A}^0_{\widetilde\Delta})
\end{equation}
then from the commutator relations,
\begin{equation}
\label{Dc48.bis.5}
\left[ \overline D , \frac\partial{\partial \theta} \right] = \left[ \overline D , e^{\theta} \frac\partial{\partial r}\right] = 0 \, .
\end{equation}

We automatically get that for $\delta$ any differential operator of order $k$, $\delta \, \varphi_n$ converge $C^0$ in the same compact, i.e. $t$ in the proof of \thref{Dc:ClaimExtra2} is independent of~$\delta$.
\end{proof}

We can usefully note that the proof works in maximal generality, i.e.

\begin{furtherfact}
\thlabel{Dc:FactExtra1}
Let $\rho : \widetilde X \to X$ be the real blow up of a Riemann surface in a discrete set of points $Z \subseteq X$ with $U \subseteq \widetilde X$ open and $V$ a vector bundle on $X$ then for $E$ the total exceptional divisor, the map,
\begin{equation}
\label{Dc48.bis.bis}
\Gamma_c (U , V \otimes {\mathcal A}^0_{\widetilde X}) \xrightarrow{ \ \overline\partial \ } \Gamma_c (U, V \otimes {\mathcal A}^{0,1}_{\widetilde X} (\log E))
\end{equation}
is closed, and even an embedding unless $U = \widetilde X$ is compact.
\end{furtherfact}

\begin{proof}
The proposition is standard elliptic theory for $Z = \varphi$, so suppose otherwise and observe the following division of cases,
\begin{enumerate}
\item[(a)] $U = \widetilde X$, compact.
\item[(b)] $U \ne \widetilde X$, but $\widetilde X$ compact.
\item[(c)] Otherwise.
\end{enumerate}
Now in case (c), $V$ is trivial, and, cf. the proof of \thref{Dc:ClaimExtra2}, any compact subset of $U$ is contained in a relative compact $U_t$ where we may find an inverse to the Laplacian, i.e.
\begin{equation}
\label{Dc48.bis.6}
{\rm Id} = \overline\partial \, \partial \, G = G \, \overline\partial \, \partial
\end{equation}
so that in these terms the operator $K$ of \eqref{Dc13} is,
\begin{equation}
\label{Dc48.bis.7}
- G \, \partial
\end{equation}
and of course $U_t \cap Z$ is finite. At the other extreme, we can simply treat case (a) like \thref{Dc:cor1} and ignore it. The nuisance case is (b), where we have to worry about harmonic projection, $P$, so, a priori we only have the formula,
\begin{equation}
\label{Dc48.bis.8}
-K \, \overline\partial \, \varphi := - G \, \overline\partial^* (\overline\partial \, \varphi) = ({\rm Id} - P) \, \varphi \, .
\end{equation}
Now in either case choose points $*_z$ in the fibre of the exceptional divisor over $z \in Z$, with the specification that $*_z \notin U_z$ if this is possible. In particular in the nuisance case (b) there is at least one point $\infty \in Z$ where this can be done, and better still for any $m \gg 0$ sufficiently large the $\Delta_{\overline\partial}$-harmonic projector on smooth sections of,
\begin{equation}
\label{Dc48.bis.9}
E(-m \, \infty)
\end{equation}
is zero. As such we first aim to prove that $\overline\partial$ is closed on the finite codimension subspace,
\begin{equation}\label{Dc48.bis.10}
\begin{split}
&\bigcap_{z \in Z} \Gamma (U , V \otimes {\mathcal A}^0_{\widetilde X} \, {\mathfrak m}(z)) \, , \ \mbox{case (c),} \\&\bigcap_{z \ne \infty} \Gamma (U , V \otimes {\mathcal A}^0_{\widetilde X} \, {\mathfrak m}(*_z) \cap {\mathfrak m} (*_{\infty})^m) \, , \ \mbox{case (b).}
\end{split}
\end{equation}
In the former case there is nothing to be done beyond replacing the operator of \eqref{Dc46.bis} by a sum over $z \in Z$, and otherwise everything is formally the same. In the latter case we employ the operators of \eqref{Dc46.bis} to the first $m$-terms of a sequence of functions $\varphi_n$, such that $\overline\partial \, \varphi_n \to \Psi$, by way of the relations \eqref{Dc48.bis.5} to deduce that in the fibre over $\infty$ the operators $\Pi_i \, \varphi_n$, \eqref{Dc10}, $0 \leq i \leq m$, converge, and whence $\overline\partial \, (Q_m \, \varphi_n)$ converge by \eqref{Dc44} to,
\begin{equation}
\label{Dc48.bis.11}
- \overline\partial (b) \sum_{i=0}^m \Pi_i \Psi + Q_m \Psi
\end{equation}
close to $E_{\infty}$. At this juncture we can take $G$ in \eqref{Dc48.bis.8} to be the Green's operator for the bundle of \eqref{Dc48.bis.9}, so $P$ is zero, and the $Q_m \varphi_n$, whence the $\varphi_n$, converge in $C_c^0 (U)$ exactly as in \eqref{Dc48} and \thref{Dc:ClaimExtra2}, after which we can again appeal to \eqref{Dc48.bis} to get $C_c^k(U)$ convergence.
\end{proof}

Irrespectively of this generalisation, the local \thref{Dc:Fact2} already admits a far from obvious corollary, to wit:

\begin{cor}
\thlabel{Dc:CorExtra1}
Let $\rho : \widetilde\Delta \to \Delta$ be the real blow up of a disc in the origin with $x \in E$, then the $LF$ space,
\begin{equation}
\label{Dc49}
{\mathcal O}_{\widetilde\Delta , x} := \varinjlim_{V \ni x} {\mathcal O}_{\widetilde\Delta} (V)
\end{equation}
of holomorphic functions smooth up to the boundary, wherein,
\begin{equation}
\label{Dc49.bis.1}
{\mathcal O}_{\widetilde\Delta} (V) \longhookrightarrow \Gamma (V , {\mathcal A}^0_{\widetilde\Delta})
\end{equation}
is given the subspace topology, is separated.
\end{cor}

\begin{proof}
Take a neighbourhood $U = K^0 \times [0,R) \ni x$ of the form encountered in \thref{Dc:not3}, and consider the maps,
\begin{equation}
\label{Dc49.bis.2}
U \backslash x \overset{j}{\longhookrightarrow} U \overset{i}{\longhookleftarrow} x \, .
\end{equation}
Then we have a short exact sequence of sheaves,
\begin{equation}
\label{Dc49.bis.3}
0 \longrightarrow j_! \, {\mathcal O}_{U\backslash x} \longrightarrow {\mathcal O}_U \longrightarrow i_* \, {\mathcal O}_{U,x} \longrightarrow 0
\end{equation}
whose long exact cohomology sequence for compact support is,
\begin{equation}
\label{Dc49.bis.4}
0 \longrightarrow {\mathcal O}_{U,x} \longrightarrow {\rm H}_c^1 ({\mathcal O}_{U \backslash x}) \longrightarrow {\rm H}_c^1 ({\mathcal O}_U) \longrightarrow 0
\end{equation}
all of which is without topology. We have, however, a short exact sequence of complexes,
\begin{equation}
\label{Dc49.bis.5}
0 \longrightarrow {\rm H}_c^0 (U \backslash x , {\mathcal A}^{^\bullet}_{\widetilde\Delta}) \longrightarrow {\rm H}_c^0 (U,{\mathcal A}_{\widetilde\Delta}^{^\bullet}) \longrightarrow i_* \, {\mathcal A}^{^\bullet}_{\widetilde\Delta , x} \longrightarrow 0
\end{equation}
which although not \underline{To}p exact does have continuous arrows, and of course \eqref{Dc49.bis.4} is the long exact sequence associated to \eqref{Dc49.bis.5}. Better still for $V \ni x$ any neighbourhood and $\overline b$ a bump function with support in $V$ and identically 1 close to $x$, the connecting homomorphism in \eqref{Dc49.bis.4} comes from,
\begin{equation}
\label{Dc49.bis.6}
\Gamma (V , {\mathcal O}_{\widetilde\Delta}) \longrightarrow \Gamma_c (U , {\mathcal A}^{0,1}_{\widetilde\Delta} (\log E)) : f \longmapsto \overline\partial \, (\bar b) f
\end{equation}
so if the ${\rm H}_c^1$ in \eqref{Dc49.bis.4} enjoy the topology inherited from \eqref{Dc49.bis.5} and ${\mathcal O}_{\widetilde\Delta , x}$ has the topology of \eqref{Dc49} then every arrow in \eqref{Dc49.bis.4} is continuous. As such \thref{Dc:Fact2} affords a continuous injection of ${\mathcal O}_{\widetilde\Delta , x}$ into a separated space.
\end{proof}

Before coming to a conclusion we require, in the same vein,

\begin{notation}
\thlabel{Dc:not4}
\nomenclature[D]{Distributions \hyperref[Dc:not4]{on a real blow up}}{} 
Mayer-Vittoris with compact support affords sheaves on $\widetilde\Delta$,
\nomenclature[N]{Distributions on a}{real blow up of type $(1,1)$ \hyperref[Dc51]{${\mathcal D}$}$^{1,1}$} 
\nomenclature[N]{Distributions on a}{real blow up of type $(1,0)$ \hyperref[Dc51]{${\mathcal D}$}$^{1,0} (-\log E)$} 
\begin{equation}
\label{Dc51}
{\mathcal D}^{1,1}_{\widetilde\Delta} (U) := {\rm Hom}_{\rm Top} (\Gamma_c (U, {\mathcal A}^0_{\widetilde\Delta}) , {\mathbb C})
\end{equation}
$$
{\mathcal D}^{1,0}_{\widetilde\Delta} (-\log E)(U) := {\rm Hom}_{\rm Top} (\Gamma_c (U, {\mathcal A}^{01}_{\widetilde\Delta} (\log E)) , {\mathbb C})
$$
where the topological dual is taken in the topological direct limit
\begin{equation}
\label{Dc52}
\varinjlim_K \Gamma_K (U , {\mathcal A}^0_{\widetilde\Delta}) \, , \quad \mbox{$K$ compact $\subset U$}
\end{equation}
of the Fr\'echet spaces with norms: max $n^{\rm th}$ derivative over $K$, and similarly for ${\mathcal D}^{0,1}_{\widetilde\Delta} (\log E)$. In particular therefore if,
\begin{equation}
\label{Dc53}
V \underset{j}{\longhookrightarrow} \widetilde\Delta \underset{i}{\longhookleftarrow} Z
\end{equation}
is the inclusion of an open set with a compact complement,
\begin{equation}
\label{Dc54}
(j_* \, {\mathcal D}_V^{1,1})(U) = {\rm Hom}_{\rm Top} (\Gamma_c (U \cap V , {\mathcal A}^0_{\widetilde\Delta}), {\mathbb C})
\end{equation}
so that dualising the exact sequence obtained by applying $\Gamma_c$ to the acyclic sequence,
\begin{equation}
\label{Dc55}
0 \longrightarrow j_! \, {\mathcal A}_V^0 \longrightarrow {\mathcal A}^0_{\widetilde\Delta} \longrightarrow i_* \, i^* {\mathcal A}^0_{\widetilde\Delta} \longrightarrow 0
\end{equation}
yields an exact (albeit not topologically) sequence:
\begin{equation}
\label{Dc56}
0 \longrightarrow {\mathcal D}^{1,1}_Z \longrightarrow {\mathcal D}^{1,1}_{\widetilde\Delta} \longrightarrow j_* \, {\mathcal D}^{1,1}_V \, .
\end{equation}
More precisely the topology of the kernel in \eqref{Dc55}, is, cf. \thref{Dc:ClaimExtra1}, rather far from the subspace topology. Indeed if $Z$ were a sub-manifold (e.g. $E$ in our current considerations is the relevant case) then the closure of the aforesaid kernel is the space of smooth functions which are flat along $Z$. Consequently, cf. \thref{Dc:bonus1}, the kernel in \eqref{Dc56}, i.e. distributions with support in $Z$, which, interalia, enjoys the subspace topology, is the dual of functions completed in $Z$. Thus for $I_Z$ the ideal of functions vanishing on $Z$,
\begin{eqnarray}
\label{Dc57}
{\mathcal D}^{1,1}_Z (V) &= &{\rm Hom}_{\rm Top} \left(\varprojlim_n \Gamma_c (V , {\mathcal A}^0_{\widetilde\Delta / I^n}), {\mathbb C}\right) \nonumber \\
&= &\varinjlim_n {\rm Hom}_{\rm Top} \left(\Gamma_c (V , {\mathcal A}^0_{\widetilde\Delta}/I^n), {\mathbb C}\right) \, .
\end{eqnarray}

Similarly we define (albeit it really is \cite[Theorem 6.7]{verdier}) the dualising sheaf on $\widetilde\Delta$ to be the kernel $\omega_{\widetilde\Delta}$ of the dual of $\Gamma_c$ applied to \eqref{Dc5}, i.e.
\nomenclature[D]{Dualising \hyperref[Dc58]{sheaf of a real blow up}}{}
\nomenclature[N]{Dualising sheaf of a real blow up}{\hyperref[Dc58]{$\omega$}} 
\begin{equation}
\label{Dc58}
0 \longrightarrow \omega_{\widetilde\Delta} \longrightarrow  {\mathcal D}^{1,0}_{\widetilde\Delta} (-\log E) \xrightarrow{ \ \overline\partial^V \ }  {\mathcal D}^{1,1}_{\widetilde\Delta} \longrightarrow 0
\end{equation}
which by Hahn-Banach and \thref{Dc:Fact2} is not only exact on global sections over any open set, but topologically exact.
\end{notation}

We thus arrive to our principle calculation, to wit:

\begin{fact}
\thlabel{Dc:Fact4}
Let $\widetilde{\mathcal O}_{\widetilde\Delta} \hookleftarrow {\mathcal O}_{\widetilde\Delta}$ 
\nomenclature[N]{Real blow up, $\ell_1$,}{up to boundary and holomorphic \hyperref[Dc:Fact4]{$\widetilde{\mathcal O}$}} 
be the sheaf of holomorphic functions which are locally $\ell_1$ for the measure $\vert z \vert^{-1} d \overline z \, dz$, i.e. $dr \, d\theta$ in the polar coordinates of \eqref{Dc2}, with $\widetilde\omega (nE)$ 
\nomenclature[N]{Real blow up,}{differentials with $\widetilde{\mathcal O}$ coefficients \hyperref[Dc:Fact4]{$\widetilde{\omega}$}}
the locally free rank 1 $\widetilde{\mathcal O}_{\widetilde\Delta}$ module of differentials with a pole of order $n$ around $E$, then the dualising sheaf, $\omega_{\widetilde\Delta}$, of the category of ${\mathcal O}_{\widetilde\Delta}$ modules is given by a non-split extension,
\begin{equation}
\label{Dc59}
0 \longrightarrow {\rm Tors} \, \omega_{\widetilde\Delta} \longrightarrow \omega_{\widetilde\Delta} \longrightarrow \varinjlim_n \, \widetilde\omega (nE) \longrightarrow 0
\end{equation}
where the torsion sub-sheaf is necessarily contained in ${\mathcal D}_E^{1,0} (-\log E)$ and is canonically isomorphic to,
\begin{equation}
\label{Dc60}
{\mathcal H}^1_{{\rm Zar},0} (\Delta , \omega_{\Delta}), \ \mbox{i.e. non canonically} \ \  {\mathbb C} \left[ \frac1z \right] \, \frac{dz}{z} \, .
\end{equation}
\end{fact}

\begin{proof}
For any sheaf ${\mathcal F}$ of ${\mathcal A}^0$-modules let $\wedge$ denote completion in $E$ and, similarly define ${\mathcal H}_E^1$ as the quotient,
\begin{equation}
\label{Dc61}
\widehat{\mathcal F} \longrightarrow \varinjlim_n \, \widehat{\mathcal F} (n E) \longrightarrow {\mathcal H}_E^1 ({\mathcal F}) \longrightarrow 0 \, .
\end{equation}
Then following \cite[Remarque 55]{sga2} there is a residue pairing,
\begin{equation}
\label{Dc62}
\widehat{\mathcal A}^{0,1}_{\widetilde\Delta} (\log E) \times {\mathcal H}^1_E ({\mathcal A}^{1,0}_{\widetilde\Delta}) \underset{\rm Res}{\longrightarrow} {\mathcal A}^1_E
\end{equation}
i.e. although as a space $E \hookrightarrow \widetilde\Delta$ is a real boundary, the existence of a coordinate free residue about $E$ is a formal property of the ring of analytic functions ${\mathbb C} \{r,\theta\}$ which, by density, extends to smooth functions. Specifically, therefore,  if we choose an embedding $z : \Delta \hookrightarrow {\mathbb C}$ and compatible polar coordinates $r,\theta$ by way of \eqref{Dc2} then the relevant case of the pairing \eqref{Dc62} is,
\begin{equation}
\label{Dc63}
\left( \sum_{i=0}^{\infty} \psi_i (\theta) \, r^i \frac{d\overline z}{\overline z} \right) \otimes \left( \sum_{j=0}^N c_j (\theta) \, r^{-j} \frac{dz}z \right)
\end{equation}
$$
\longmapsto \sum_{i=0}^{\infty} \sum_{j=0}^N \psi_i (\theta) \, c_j (\theta) \, r^{i-j} \, \frac{dr \, d\theta}{r} \underset{\rm Res}{\longmapsto} \sum_{i=0}^N \psi_i (\theta) \, c_i(\theta) \, d\theta \, .
$$
In particular the result of the residue pairing can be integrated along $E$, so we get a canonical map,
\begin{equation}
\label{Dc64}
{\mathcal H}_E^1 ({\mathcal A}^{1,0}_{\widetilde\Delta}) \longrightarrow {\mathcal D}^{1,0}_E (-\log E)
\end{equation}
and, unsurprisingly, we assert,

\begin{claim}
\thlabel{Dc:claim3}
The image of the natural map from ${\mathcal H}_{{\rm Zar},0}^1 (\omega_{\Delta})$ to ${\mathcal H}_E^1 ({\mathcal A}^{1,0}_{\widetilde\Delta})$ in ${\mathcal D}^{1,0}_E (-\log E)$ is the kernel of the restriction of $\overline\partial^{\vee}$.
\end{claim}

\begin{proof} 
[Proof of \thref{Dc:claim3}] Fix $N \in {\mathbb Z}_{\geq 0}$ in \eqref{Dc63} and a neighbourhood $U$ as in \thref{Dc:not3} with $K^0 \ne E$, so that in the notation of \eqref{Dc43} we have,
\begin{equation}
\label{Dc65}
\psi = \sum_{i=0}^N P_i \, \psi + Q_N \, \psi \in \Gamma_c (U,{\mathcal A}^0_{\widetilde\Delta}) \, .
\end{equation}
As such if for some $0 \leq j \leq N$, $\gamma = c_j (\theta) \, r^{-j} \frac{dz}z$ is a class in the left hand side of \eqref{Dc64} then by \eqref{Dc4}, \eqref{Dc44} and \eqref{Dc63},
\begin{equation}
\label{Dc66}
\int_E {\rm Res} \, (\overline\partial (\psi) \otimes \gamma) = - \int_E c_j (\theta) \, e^{j\theta} d \left( \frac{P_j \, \psi}{z^j} \right)
\end{equation}
so $\gamma$ is a solution of $\overline\partial^{\vee} \gamma = 0$ iff $\gamma = c_j \, z^{-j} \frac{dz}z$, for some $c_j \in {\mathbb C}$. Consequently the torsion certainly contains \eqref{Dc60}, while to check the converse observe by \eqref{Dc57},
\begin{eqnarray}
\label{Dc67}
{\mathcal D}_E^{1,0} (-\log E) &= &{\rm Hom}_{\rm top} (\widehat{\mathcal A}^{0,1} (\log E) , {\mathbb C}) \nonumber \\
&= &\varinjlim_n \, {\rm Hom}_{\rm top} ({\mathcal A}^{0,1} (\log E) / (r^n) , {\mathbb C}) \, .
\end{eqnarray}
In particular therefore it admits a filtration with graded pieces,
\begin{equation}
\label{Dc68}
{\rm gr}^n \, {\mathcal D}_E^{1,0} (-\log E) = {\mathcal D}^1 (E) \otimes \left( \overline z \frac{\partial}{\partial \overline z}\, r^{-n} \right)
\end{equation}
where the latter is the sheaf of distributions on $E$. Now $\overline\partial$ respects the filtration \eqref{Dc68}, so if some torsion class, $T$, in $\widetilde\omega_{\widetilde\Delta}$ doesn't belong to ${\mathcal H}^1 (\omega_{\widetilde\Delta})$ there is a maximal $n\geq 0$ such that,
\begin{equation}
\label{Dc69}
T_n \otimes \left( \overline z \frac{\partial}{\partial \overline z} \otimes r^{-n} \right) \notin {\mathbb C} z^{-n} \cdot \frac{dz}z \subseteq {\rm gr}^n \, {\mathcal D}^{1,0}_E (-\log E)
\end{equation}
but $\overline\partial \left( T_n \otimes \left( \overline z \frac{\partial}{\partial z} \otimes r^{-n} \right) \right) = 0$, i.e. for $r^n \, \psi_n (\theta) \in {\mathcal A}^0 (-n \, E)$,
\begin{equation}
\label{Dc70}
0 = \left( T_n \otimes \overline z \frac{\partial}{\partial\overline z} \otimes r^{-n} \right) \overline\partial (r^n \psi_n (\theta)) = T_n (e^{n\theta} \, \overline D (e^{-n\theta} \psi_n))
\end{equation}
which is iff $T_n$ is a multiple, say, without loss of generality, 1, of integration against $e^{-n\theta} d\theta$, while the residue pairing, of \eqref{Dc63}, affords,
\begin{equation}
\label{Dc71}
{\rm Res} \left( \frac{d\overline z \, dz}{\vert z \vert^2} \right) = d\theta
\end{equation}
or, equivalently, a canonical identification $\left(\overline z \frac{\partial}{\partial \overline z} \right) \otimes d\theta$ with $dz/z$, so:
\begin{equation}
\label{Dc72}
T_n \otimes \left(\overline z \frac{\partial}{\partial \overline z} \otimes r^{-n} \right) = z^{-n} \otimes \left(\overline z \frac{\partial}{\partial \overline z} \otimes d\theta \right) = z^{-n} \frac{dz}z
\end{equation}
from which \eqref{Dc69} is seen to be nonsense.
\end{proof}

Now let us turn to identifying $\omega_{\widetilde\Delta}$ modulo torsion by first confirming the intervention of the $\ell_1 ({\rm loc})$ condition in \thref{Dc:Fact4} by way of,

\begin{claim}
\thlabel{Dc:claim4}
Let $U \subseteq \widetilde\Delta$ be an open neighbourhood of a compact subset (possibly everything) $K \subset E$ then, modulo torsion, the solutions of,
\begin{equation}
\label{Dc73}
\overline\partial^{\vee} T = 0 \, , \quad T \in \Gamma \left( U , {\mathcal D}^{1,0}_{\widetilde\Delta} (-\log E) \right)
\end{equation}
which have order $0$, i.e. $\Vert T \omega \Vert \ll \Vert \omega \Vert_0$, where, $\Vert \ \Vert_0$ of \eqref{Dc20} is the $C^0$ norm on ${\mathcal A}^{1,0} (\log E)$, are exactly $\Gamma (U,\widetilde\omega_{\widetilde\Delta})$.
\end{claim}

\begin{proof}
By the Riesz representation theorem operators of order 0, have measure regularity on the coefficient, i.e. $\omega \overline z / d \overline z$, in ${\mathcal A}^0$ of a differential form $\omega$. Further any measure has a Lebesgue decomposition into measures,
\begin{equation}
\label{Dc74}
T = \un_E \, T + \un_{\widetilde\Delta \backslash E} \, T
\end{equation}
so $\un_{\widetilde\Delta \backslash E} \, T$ is a holomorphic differential $h\, dz$ admitting the $C^0$ bound, on any given compact subset $C \subset U$
\begin{equation}
\label{Dc75}
\left\vert \int_C W h \, \frac{d \overline z \, dz}{\overline z} \right\vert \ll \sup_C \vert W \vert
\end{equation}
i.e. $h \, dz$ is a section of $\widetilde\omega_{\widetilde\Delta}$. Conversely, suppose, without loss of generality, that $\frac{\vert h \vert}{\vert z \vert} \, d\overline z \, dz$ is a finite measure on $U$ for some function, $h$, holomorphic off $E$, then for a test function $\psi$,
\begin{equation}
\label{Dc76}
\int_U h \,  \overline\partial (\psi) = \lim_{\delta \to 0} \, \int_{r=\delta} (h\psi) \, dz \, .
\end{equation}
As such for fixed $R > 0$ and any $\varepsilon > 0$ consider the set,
\begin{equation}
\label{Dc77}
S_{\varepsilon} := \left\{0 \leq t < R \ \biggl\vert \ \int_{r=t} \vert h \vert \cdot d\theta \geq \varepsilon / t \right\}
\end{equation}
then the characteristic function of $S_{\varepsilon}$ satisfies the estimate,
\begin{equation}
\label{Dc78}
\un_{S_{\varepsilon}} \leq \frac1\varepsilon \int_{r=t} \vert h \vert \, \vert dz \vert
\end{equation}
and whence, by our hypothesis of absolute integrability, every $S_{\varepsilon}$ has finite $dr/r$ measure. Better still the left hand side of \eqref{Dc76} is absolutely integrable, so we can compute the limit on the right using any sequence $\delta_n \to 0$ that we want, e.g. missing $S_{\varepsilon}$, so \eqref{Dc76} is $0$.
\end{proof}

Amongst what remains to do let us first check that anything with a finite pole can occur, i.e.

\begin{claim}
\thlabel{Dc:claim5}
Let everything be as in \thref{Dc:not3} with $K^0 \ne E$ and $h \, dz \in \Gamma (U , \widetilde\omega_{\widetilde\Delta} (k+1))$ for some $k \geq 0$ then if we define,
\begin{equation}
\label{Dc79}
\Gamma (U , {\mathcal D}^{1,0} (-\log E)) \ni Z_h : \psi \longrightarrow \int_U Q_k (\psi) (h \, dz)
\end{equation}
there is a torsion class $T_h \in \Gamma (U , {\mathcal D}^{1,0}_E (-\log E))$ such that,
\begin{equation}
\label{Dc80} 
\overline\partial^\vee (Z_h + T_h) = 0 \, .
\end{equation}
\end{claim}

\begin{proof}
By definition the operator $Q_k$ of \eqref{Dc43} admits a bound of the form,
\begin{equation}
\label{Dc81} 
\left\vert Q_k (\psi) \frac{d\overline z}{\overline z} \right\vert \ll \Vert \psi \Vert_{k+1} \, \vert z \vert^k \, \vert dz \vert
\end{equation}
while by hypothesis $\vert h \vert \, \vert z \vert^k \, d\overline z \, dz$ is a locally finite measure so the integrand in \eqref{Dc79} is absolute. Better still replacing $h$ by $(z^{k+1} h)$ \eqref{Dc75}-\eqref{Dc78} apply mutatis mutandis to conlude that,
\begin{equation}
\label{Dc82}
\int_U \overline\partial (Q_k (\psi)) \, h \, dz = 0 \, , \quad \psi \in \Gamma_c (U,{\mathcal A}^0) \, .
\end{equation} 
As such, by \eqref{Dc44},
\begin{eqnarray}
\label{Dc83}
\overline\partial^{\vee} Z_h (\psi) &= &\sum_{i=0}^k Z_h \, \overline\partial (b) \, \Pi_i \, (\psi) \nonumber \\
&= &\sum_{i=0}^k \int_E \left( \frac{\Pi_i (\psi)}{z^i} \right) \left( \int_0^R \frac{\overline D (b)}r z^{i+1} h \, dr \right) d\theta \, .
\end{eqnarray}
On the other hand, and critically, $K^0 \ne E$ so we can find functions $h_i (\theta)$ such that,
\begin{equation}
\label{Dc84}
dh_i = \left( \int_0^R \frac{\overline D (b)}r z^{i+1} h \, dr \right) d\theta
\end{equation}
which in turn allows us to define classes,
\begin{equation}
\label{Dc85}
{\mathcal D}_E^{1,0} (-\log E) \ni \Pi_i (h) : \Psi \longmapsto \int_E h_i(\theta) \left(\frac{\Pi_i (\psi)}{z^i} \biggl\vert_E \right)
\end{equation}
so that by \eqref{Dc3}-\eqref{Dc4}, $T_h = \underset{i=0}{\overset{k}{\sum}} \, T_i$ solves \eqref{Dc80}.
\end{proof}

The remaining step that we've found everything will require some more operators, so, exactly as in \thref{Dc:claim5} take $U$ of the form $[0,X) \times K^0$ in polar coordinates, then we have an operator,
\begin{equation}
\label{Dc86}
R : \Gamma_c (U,{\mathcal A}^0) \longrightarrow \Gamma_c (U,{\mathcal A}^0) : \psi \longmapsto - \int_r^X \psi (t,\theta) \, d\theta
\end{equation}
which again following \eqref{Dc8} we extend to forms after a choice of embedding of $\Delta$ in ${\mathbb C}$ by way of,
\begin{equation}
\label{Dc87}
R \left( \psi \frac{d\overline z}{\overline z} \right) := R(\psi) \frac{d\overline z}{\overline z}
\end{equation}
and a simple explicit calculation affords,
\begin{equation}
\label{Dc88}
\left[\overline D ,R\right] = - \frac R2 \, , \quad \left[ \overline \partial , R \right] = - \frac R2 \cdot \frac{d\overline z}{\overline z} \, .
\end{equation}
Unfortunately in the angular direction things are a little bit more complicated since integration doesn't preserve compact support. As such choose a compactly supported bump function $\beta : K^0 \to [0,1]$ with total integral $1$ over $K^0$ against $d\theta$ so that we have a projector,
\begin{equation}
\label{Dc89}
B : \Gamma_c (U,{\mathcal A}^0) \longrightarrow \Gamma_c (U,{\mathcal A}^0) : \psi \longmapsto \beta(\theta) \int_{K^0} \psi (r,\theta) \, d\theta
\end{equation}
which, in turn allows as to define operators,
\begin{equation}
\label{Dc90}
\Theta : \Gamma_c (U,{\mathcal A}^0) \longmapsto  \Gamma_c (U,{\mathcal A}^0) : \psi \longmapsto \int_*^{\theta} (\un - B) (\psi) \, d\theta
\end{equation}
where $*$ is the start of $K^0$ oriented counter clockwise, and which we extend to forms exactly as in \eqref{Dc87}, so:
\begin{equation}
\label{Dc91}
\left[ \overline D , \Theta \right] = \frac B2 \, , \quad \left[ \overline \partial , \Theta \right] = \frac B2 \cdot \frac{d\overline z}{\overline z} \, .
\end{equation}
Observe, therefore, by induction on $k \geq 0$,
\begin{equation}
\label{Dc92}
\left[ \overline D , R^{k+1} \right] = - \frac{(1+k)}2 \, R^{k+1} \, , \quad \left[ R , \Theta \right] = 0 \, , \quad \left[ \overline D , \Theta^{k+1} \right] = B\Theta^k/2 \, .
\end{equation}
Now if $H$ is a functional of order $k+1$, $k \geq 0$, in the sense that,
\begin{equation}
\label{Dc93}
\left\Vert H \psi \frac{d\overline z}{\overline z} \right\Vert \ll \Vert \psi \Vert_{k+1}
\end{equation}
for $\Vert \ \Vert_{k+1}$ the Sobolev norm of \eqref{Dc17} computed, without loss of generality, over $\overline U$, then,
\begin{equation}
\label{Dc94}
H_{k+1} := H (R^{k+1} \, \Theta^{k+1})
\end{equation}
is a functional of order $0$, so, cf. \thref{Dc:claim4}, it's a measure, i.e.
\begin{equation}
\label{Dc95}
H_{k+1} = \mu \otimes dz \, ; \quad \psi \frac{d\overline z}{\overline z} \longmapsto \mu (\psi) \, .
\end{equation}

Equally from the commutator relations \eqref{Dc92}, if $\overline\partial^{\vee} H = 0$,
\begin{equation}
\label{Dc96}
\overline\partial^{\vee} H_{k+1} = \frac{(k+1)}2 \, H_{k+1} - \frac12 \, H (BR^{k+1} \, \Theta^k) \, .
\end{equation}
Now observe that the obstruction to $H_{k+1}$ being an eigenvector with eigenvalue $(k+1)$ is particularly simple, i.e.

\begin{claim}\thlabel{Dc:claim6}
There are measures $\mu_i$ on $[0,R)$, $0 \leq i \leq k$, such that, as a functional on the coefficient of $d \overline z / \overline z$,
\begin{equation}
\label{Dc97}
-\frac12 H (BR^{k+1} \Theta^k) = \sum_{i=0}^k (r^* \mu_i) \otimes \theta^i d\theta \, .
\end{equation}
\end{claim}
\begin{proof}
As above, $H (BR^{k+1} \Theta^{k+1})$ is a measure, while, as an operator on forms with compact support $B \partial / \partial \theta = 0$, so:
\begin{equation}
\label{Dc98}
\left( \left( \frac{\partial}{\partial \theta} \right)^{\!\!\vee} \right)^{\!\!k+2} H (BR^{k+1} \, \Theta^{k+1}) = 0 \, .
\end{equation}

Better still because the boundary is, in polar coordinates, $r=0$, $\partial / \partial\theta$ is still anti-self adjoint, so, by regularity for the $d$-operator in one variable, $HB^{k+1} \, \Theta^{k+1}$ has the form of \eqref{Dc97} but for a polynomial of degree $(k+1)$ rather than $k$. On the other hand $\partial / \partial\theta$ is a right inverse for $\Theta$ so the simple expedient of partially differentiating in $\theta$ yields \eqref{Dc97}.
\end{proof}

Consequently if we restrict \eqref{Dc96} to $\widetilde\Delta \backslash E$ then we have,
\begin{equation}
\label{Dc99}
\overline\partial^{\vee} \left( H_{k+1} \left( \frac{1}{r^{k+1}}\right)\right) = \sum_{i=0}^k r^* \mu_i \left(\frac{1}{r^{k+1}} \right) \otimes \theta^i d\theta \, .
\end{equation}

Equally whether for operators or functions pulled back along the first map in \eqref{Dc2}, it's easy to solve the $\overline\partial$-equation as a function of $r$ alone. Indeed if we introduce an operator,
\begin{equation}
\label{Dc100}
\widetilde R (\psi) := R \left( \frac{\psi}r \right)
\end{equation}
on compactly supported functions in $\widetilde\Delta \backslash E$, then for any distribution $\nu$ on $[0,X)$,
\begin{equation}
\label{Dc101}
(\nu \widetilde R) : \psi \frac{d\overline z}{\overline z} \longmapsto \nu (\widetilde R \, r_* (\psi \, d \theta)) \ \mbox{solves} \ \overline\partial^{\vee} (\nu \widetilde R) = r^* \nu \otimes d \theta \, .
\end{equation}

Now, certainly, solving the $\overline\partial$-equation in such a way may well lead to solutions of the $\overline\partial$-equation that are bigger than those constructed in the proof of \thref{Dc:Fact2} but it has the advantage that we can combine them with integration by parts to produce, over $\widetilde\Delta \backslash E$, a solution,
\begin{equation}
\label{Dc102}
\nu := \sum_{i=0}^k r^* \nu_i \left( \frac{\vert \log r \vert^{k+1-i}}{r^{k+1}} \right) \otimes \theta^i : \psi \frac{d\overline z}{\overline z} \longmapsto \!\int_0^X \!\frac{\vert \log r \vert^{k+1-i}}{r^{k+1}} \, d\nu_i (r) \!\int \!\psi (r,\theta) \, \theta^i d\theta
\end{equation}
of the $\overline\partial^{\vee}$ equation to the right hand side of \eqref{Dc99} for some finite measures $\nu_i$, $0 \leq i \leq k$ on $[0,X)$. We therefore obtain, from \eqref{Dc99} and \eqref{Dc102},
\begin{equation}
\label{Dc103}
\overline\partial^{\vee} (h \, dz) = 0 \, , \quad h \, dz := H_{k+1} \left( \frac1{r^{k+1}} \right) - \nu
\end{equation}
and whence by a combination of $\overline\partial$-regularity, and the measure regularity in \eqref{Dc95} and \eqref{Dc102} a holomorphic 1-form, $h \, dz$, on $\widetilde\Delta \backslash E$ such that,
\begin{equation}
\label{Dc104}
\int_{\widetilde\Delta \backslash E} \vert h \vert \left\vert \frac{r}{\log r} \right\vert^{k+1} dr \, d\theta < \infty \, .
\end{equation}
This proves that $h \, dz$ belongs to $\widetilde\omega_{\widetilde\Delta} ((k+2)E)$ and while, we could a posteriori replace $k+2$ by $k+1$ we don't need to since,

\begin{claim}\thlabel{Dc:claim7}
Suppose for some $i \geq 1$ and $\overline\partial^{\vee} \nu$ any distribution polynomial in $\theta$ of degree at most $k$, equal to the right hand side of \eqref{Dc99} with $h \, dz$ given by \eqref{Dc103} the latter is a section over $U = K^0 \times [0,X)$ of $\widetilde\omega_{\widetilde\Delta} ((k+i)E)$ then $H$ restricts over $U \backslash E$ to a section of the same restricted from $U$.
\end{claim}

\begin{proof}
We can recover $H$ from $h$ by way of,
\begin{equation}
\label{Dc105}
H = \left( \frac{\partial^{2k+2}}{\partial r^{k+1} \partial \theta^{k+1}} \right)^{\!\!\vee} H_{k+1} \Longrightarrow \left( \frac{H}{dz} \right) \biggl\vert_{\widetilde\Delta \backslash E} = \frac{\partial^{k+1}}{\partial r^{k+1} \partial \theta^{k+1}} (r^{k+1} h)
\end{equation}
thanks to \eqref{Dc99}, \eqref{Dc102} and the polynomial hypothesis in $\nu$. On the other hand,
\begin{equation}
\label{Dc106}
r^{k+1} \frac{\partial^{k+1}}{\partial r^{k+1}} = ({\rm poly})_{k+1} \left( r \frac{\partial}{\partial r} \right)
\end{equation}
for some polynomial of degree $(k+1)$ with constant coefficients. As such for another polynomial with constant coefficients of bi-degree $2k+2$,
\begin{equation}
\label{Dc107}
z^{k+i} \left( \frac{H}{dz} \right)\biggl\vert_{U \backslash E} = z^{i-1} \, {\rm poly} \!\left( z \frac{\partial}{\partial z} , \overline z \frac{\partial}{\partial z} \right) (e^{-(k+1)\theta} z^{k+i} h) \, .
\end{equation}

Equally there are also polynomials of degree $b$ with constant coefficients depending on $a$ such that for any $a,b \geq 0$
\begin{equation}
\label{Dc108}
z^a \left( z \frac{\partial}{\partial z} \right)^{\!b} = {\rm poly} \!\left( z \frac{\partial}{\partial z} \right) z^a
\end{equation}
while powers of $e^{\theta}$ are eignevectors of both $z \partial / \partial z$ and $\overline z \partial / \partial \overline z$ so altogether there is yet another polynomial with constant coefficients such that,
\begin{equation}
\label{Dc109}
z^{k+i} \left( \frac{H}{dz} \right) \biggl\vert_{U \backslash E} = e^{-(k+i)\theta} \, {\rm poly} \!\left( z \frac{\partial}{\partial z} \right) (z^{k+i} h) \, .
\end{equation}
As such \thref{Dc:claim7} will follow from,

\begin{claim}\thlabel{Dc:claim8}
$z \frac{\partial}{\partial z}$ is an endomorphism of the sheaf $\widetilde{\mathcal O}_{\widetilde\Delta}$.
\end{claim}

\begin{proof}
It will suffice to prove that the operator maps sections to sections over the open $U$ of \thref{Dc:claim7}, so to this end fix compacts,
\begin{equation}
\label{Dc110}
C = K \times [0,Y] \subset\subset C' = K' \times [0,Y'] \subset\subset U
\end{equation}
with $K,K'$ closed intervals, and a section $h$ of $\widetilde{\mathcal O}_{\widetilde\Delta}$ over $U$. Then for $C \ni s \ne 0$, \eqref{Dc76}-\eqref{Dc78} and Cauchy's theorem afford,
\begin{equation}
\label{Dc111}
sh(s) = \frac1{2\pi} \int_{\partial C' \backslash E} \frac{zh(z) \, dz}{(z-s)}
\end{equation}
while, at least for $C'$ chosen randomly $h(z) \, dz$ is, by hypothesis, absolutely integrable over $\partial C' \backslash E$. As such we may legitimately differentiate in $s$ under the integral sign in \eqref{Dc111}, and, of course, $sh'(s) = (sh)'-h$, so, altogether it will suffice to show,
\begin{equation}
\label{Dc112}
\int_{s \in C} dr \, d\theta \int_{\partial C' \backslash E} \left\vert \frac{zh}{(z-s)^2} \right\vert \vert dz \vert < \infty \, .
\end{equation}
On the other hand if $s \in [0,Y]$ and $z$ of modulus at most $Y'$ is on a segment with argument $\kappa$ then,
\begin{equation}
\label{Dc113}
\int_0^Y \frac{d \vert s \vert}{\vert z-s \vert^2} \ll 0 \left( \frac1{\vert {\rm Im} \, z \vert} \right) = 0 \left( \frac1{\vert z \vert^{\vert \kappa \vert}} \right)
\end{equation}
and in practice $\kappa$ is the distance between $K$ and $K'$ in \eqref{Dc110}, so \eqref{Dc112} admits an estimate, of the order of,
\begin{equation}
\label{Dc114}
\frac1{{\rm dist} \, (K,K')} \int_{\partial C' \backslash E} \vert h \vert \cdot \vert dz \vert
\end{equation}
which, again for $C'$ random, is certainly finite.
\end{proof}

This concludes the proof of \thref{Dc:claim4}.
\end{proof}

Equally, we thus conclude the proof of \thref{Dc:Fact4}.
\end{proof}

%\end{document}

\newpage
%\vglue 1cm

\bibliography{elvis}{}

\providecommand{\bysame}{\leavevmode\hbox to3em{\hrulefill}\thinspace}
\providecommand{\MR}{\relax\ifhmode\unskip\space\fi MR }
\providecommand{\MRhref}[2]{%
  \href{http://www.ams.org/mathscinet-getitem?mr=#1}{#2}
}
\providecommand{\href}[2]{#2}
\begin{thebibliography}{Cam78}

\bibitem[BE02]{adamExtra}
Xavier Buff and Adam~L. Epstein, \emph{A parabolic
  {P}ommerenke-{L}evin-{Y}occoz inequality}, Fund. Math. \textbf{172} (2002),
  no.~3, 249--289. \MR{1898687}

\bibitem[Ber02]{berg}
Walter Bergweiler, \emph{On the number of critical points in parabolic basins},
  Ergodic Theory Dynam. Systems \textbf{22} (2002), no.~3, 655--669.
  \MR{1908548}

\bibitem[Cam78]{Camacho}
C\'{e}sar Camacho, \emph{On the local structure of conformal mappings and
  holomorphic vector fields in {${\bf C}^{2}$}}, Journ\'{e}es {S}inguli\`eres
  de {D}ijon ({U}niv. {D}ijon, {D}ijon, 1978), Ast\'{e}risque, vol.~59, Soc.
  Math. France, Paris, 1978, pp.~3, 83--94. \MR{542732}

\bibitem[CZ21]{olivia}
Olivia Carmello and Riccardo Zanfa, \emph{Relative topos theory via stacks},
  \htmladdnormallink{\url{https://arxiv.org/abs/2107.04417}}{https://arxiv.org/abs/2107.04417},
  (2021).

\bibitem[DH93]{DH}
Adrien Douady and John~H. Hubbard, \emph{A proof of {T}hurston's topological
  characterization of rational functions}, Acta Math. \textbf{171} (1993),
  no.~2, 263--297. \MR{1251582}

\bibitem[Dou87]{D}
Adrien Douady, \emph{Disques de {S}iegel et anneaux de {H}erman}, no. 152-153,
  1987, S\'{e}minaire Bourbaki, Vol. 1986/87, pp.~4, 151--172 (1988).
  \MR{936853}

\bibitem[DS49]{DS}
Jean Dieudonn\'{e} and Laurent Schwartz, \emph{La dualit\'{e} dans les espaces
  {$\mathcal{F}$} et {$(\mathcal{L}\mathcal{F})$}}, Ann. Inst. Fourier
  (Grenoble) \textbf{1} (1949), 61--101 (1950). \MR{38553}

\bibitem[Eps]{adamBuff}
Adam Epstein, \emph{Transversality in holomorphic dynamics},
  \htmladdnormallink{\url{https://homepages.warwick.ac.uk/~mases/Transversality.pdf}}{https://homepages.warwick.ac.uk/~mases/Transversality.pdf}.

\bibitem[Eps99]{adam}
\bysame, \emph{Infinitesimal thurston rigidity and the fatou-shishikura
  inequality},
  \htmladdnormallink{\url{https://arxiv.org/abs/9902158}}{https://arxiv.org/abs/9902158},
  (1999).

\bibitem[Eps12]{adamHub}
\bysame, \emph{Transversality principles in holomorphic dynamics},
  \htmladdnormallink{\url{https://icerm.brown.edu/materials/Slides/sp-s12-w1/Transversality_Principles_in_Holomorphic_Dynamics_\%255D_Adam_Epstein,_University_of_Warwick.pdf}}{https://icerm.brown.edu/materials/Slides/sp-s12-w1/Transversality_Principles_in_Holomorphic_Dynamics_\%255D_Adam_Epstein,_University_of_Warwick.pdf},
  (2012).

\bibitem[Gar22]{jacopo}
Jacopo Garofali, \emph{Dynamical sheaves},
  \htmladdnormallink{\url{https://arxiv.org/abs/2211.05260}}{https://arxiv.org/abs/2211.05260},
  (2022).

\bibitem[\htmladdnormallink{SGA1}{http://arxiv.org/pdf/math/0206203v2}]{sga1}
Alexander Grothendieck, \emph{Rev\^etements \'etales et groupe fondamental
  ({SGA} 1)}, Lecture Notes in Mathematics, Vol. 224, Springer-Verlag, Berlin,
  1971, S{\'e}minaire de g{\'e}om{\'e}trie alg{\'e}brique du Bois Marie
  1960--61, Augment{\'e} de deux expos{\'e}s de M. Raynaud. \MR{MR0354561}

\bibitem[\htmladdnormallink{SGA-II}{http://arxiv.org/pdf/math/0511279v1}]{sga2}
\bysame, \emph{Cohomologie locale des faisceaux coh\'erents et th\'eor\`emes de
  {L}efschetz locaux et globaux ({SGA} 2)}, Documents Math\'ematiques (Paris)
  [Mathematical Documents (Paris)], 4, Soci\'et\'e Math\'ematique de France,
  Paris, 2005, S{\'e}minaire de G{\'e}om{\'e}trie Alg{\'e}brique du Bois Marie,
  1962, Augment{\'e} d'un expos{\'e} de Mich{\`e}le Raynaud. With a preface and
  edited by Yves Laszlo, Revised reprint of the 1968 French original.
  \MR{\htmladdnormallink{2171939
  (2006f:14004)}{http://www.ams.org/mathscinet-getitem?mr=2171939}}

\bibitem[Lur18]{gir}
Jacob Lurie, \emph{Giraud's theorem},
  \htmladdnormallink{\url{https://www.math.ias.edu/~lurie/278xnotes/Lecture10-Giraud.pdf}}{https://www.math.ias.edu/~lurie/278xnotes/Lecture10-Giraud.pdf},
  (2018).

\bibitem[Mil00]{milnor2}
John Milnor, \emph{On rational maps with two critical points}, Experiment.
  Math. \textbf{9} (2000), no.~4, 481--522. \MR{1806289}

\bibitem[Mil06]{milnor}
\bysame, \emph{Dynamics in one complex variable}, third ed., Annals of
  Mathematics Studies, vol. 160, Princeton University Press, Princeton, NJ,
  2006. \MR{2193309}

\bibitem[Nic05]{verdier}
Liviu Nicolaescu, \emph{The derived category of sheaves and the {P}oincar\'e
  {V}erdier duality},
  \htmladdnormallink{\url{https://www3.nd.edu/~lnicolae/Verdier-ams.pdf}}{https://www3.nd.edu/~lnicolae/Verdier-ams.pdf},
  (2005).

\bibitem[RR70]{ramis}
Jean-Pierre Ramis and Gabriel Ruget, \emph{Complexe dualisant et
  th\'{e}or\`emes de dualit\'{e} en g\'{e}om\'{e}trie analytique complexe},
  Inst. Hautes \'{E}tudes Sci. Publ. Math. (1970), no.~38, 77--91. \MR{279338}

\bibitem[Shi87]{mitsu}
Mitsuhiro Shishikura, \emph{On the quasiconformal surgery of rational
  functions}, Ann. Sci. \'{E}cole Norm. Sup. (4) \textbf{20} (1987), no.~1,
  1--29. \MR{892140}

\bibitem[Yoc88]{yoccoz}
Jean-Christophe Yoccoz, \emph{Lin\'{e}arisation des germes de
  diff\'{e}omorphismes holomorphes de {$({\bf C}, 0)$}}, C. R. Acad. Sci. Paris
  S\'{e}r. I Math. \textbf{306} (1988), no.~1, 55--58. \MR{929279}

\end{thebibliography}
\bibliographystyle{Gamsalpha}

\newpage
%\vglue 1cm 

\printnomenclature

\end{document}